\documentclass[10pt]{article}

\usepackage{amsthm,amsfonts,amsmath,amssymb}
\usepackage{mathtools}

\usepackage{pxfonts}
\usepackage{euscript}
\usepackage{bbold,bbm}
\usepackage[utf8]{inputenc}

\usepackage{authblk}

\usepackage{hyperref}
\usepackage{cleveref}
\usepackage{graphics}
\usepackage{epstopdf} 
\usepackage{xypic}
\usepackage[all,2cell,arrow,matrix,tips]{xy} \UseAllTwocells \SilentMatrices
\usepackage{graphicx}
\usepackage{verbatim}
\usepackage{leftidx}
\usepackage{enumerate}
\usepackage{phonetic}\usepackage{sectsty}

\sectionfont{\Large}

\usepackage[dvipsnames,usenames]{color}
  \usepackage{titleps}
  \newpagestyle{main}[\small]{
    \setheadrule{.55pt}%
    \sethead[\thepage]
            [\chaptertitle]
            []
            {}
            {\sectiontitle}
            {\thepage}
  }
\newcommand\void[1]       {}

\theoremstyle{definition}

\newtheorem{thm}{Theorem}[section]
\newtheorem{prop}[thm]{Proposition}
\newtheorem{cor}[thm]{Corollary}
\newtheorem{lem}[thm]{Lemma}
\newtheorem{ass}[thm]{Assumption}

\newtheorem{conv}[thm]{Convention}
\newtheorem{conj}[thm]{Conjecture}
\theoremstyle{definition}
\newtheorem{defn}[thm]{Definition}
\newtheorem{expl}[thm]{Example}

\newtheorem{rem}[thm]{Remark}
\newtheorem{notation}[thm]{Notation}

\numberwithin{equation}{section}
\numberwithin{thm}{section}

\newcommand\be            {\begin{equation}}
\newcommand\ee            {\end{equation}}
\newcommand\bea           {\begin{eqnarray}}
\newcommand\eea         {\end{eqnarray}}
\newcommand\bnu          {\begin{enumerate}}
\newcommand\enu          {\end{enumerate}}

\allowdisplaybreaks[1]

\newlength{\fighskip} \fighskip=2pt
\newlength{\figvskip} \figvskip=3pt

\newcommand{\pf}{\begin{proof}}
\newcommand{\epf}{\end{proof}}

\newcommand\Cb            {\mathbb{C}}

\newcommand\Rb            {\mathbb{R}}

\newcommand\Hb            {\mathbb{H}}
\newcommand\mO           {\mathcal{O}}

\newcommand\BB {\mathsf{B}}

\newcommand\A           {\EuScript{A}}
\newcommand\B           {\EuScript{B}}
\newcommand\C           {\EuScript{C}}
\newcommand\D           {\EuScript{D}}
\newcommand\E          {\EuScript{E}}

\newcommand\CL          {\EuScript{L}}
\newcommand\M          {\EuScript{M}}
\newcommand\N         {\EuScript{N}}

\newcommand\CP         {\EuScript{P}}

\newcommand\V        {\EuScript{V}}

\newcommand\X         {\EuScript{X}}
\newcommand\Y         {\EuScript{Y}}

\newcommand\id            {\mathrm{id}}
\newcommand\op          {\mathrm{op}}

\newcommand\ev          {\mathrm{ev}}
\newcommand\coev      {\mathrm{coev}}

\newcommand\vect    {\EuScript{V}\mathrm{ec}}

\newcommand\Fun     {\mathrm{Fun}}
\newcommand\LMod  {\mathrm{LMod}}
\newcommand\RMod  {\mathrm{RMod}}
\newcommand\BMod {\mathrm{BMod}}

\newcommand\forget  {\mathbf{f}}

\newcommand\bk       {\mathbb{k}}

\newcommand\Alg    {\mathrm{Alg}}

\newcommand\unit    {\mathbb{1}}
\newcommand{\rev} {\mathrm{rev}}

\newcommand{\colim}{\mathrm{Colim}}

\newcommand{\UMTC} {\mathrm{UMTC}}
\newcommand{\UMTCE} {\mathrm{UMTC}_{/{\EuScript{E}}}}
\newcommand{\homm} {\mathrm{Hom}}

\newcommand{\cat}{\mathrm{Cat}_{\EuScript{E}}^{\mathrm{fs}}} 

\newcommand{\NN}{\C^{\rev} \boxtimes_{\E} \C}

\newcommand{\ccat}{\mathrm{Cat}}
\newcommand{\fcat}{\mathrm{Cat}^{\mathrm{fs}}}
\newcommand{\bal}{\mathrm{bal}}

\newcommand{\str}{\mathrm{str}}

\newcommand{\ori}{\mathrm{or}}
\newcommand{\Mfld}{\M\mathrm{fld}}
\newcommand{\Disk}{\D\mathrm{isk}}
\newcommand{\ps}{\mathrm{ps}}

\newcommand{\boe}{\boxtimes_{\E}}

\newcommand{\uty}{\mathrm{uty}}
\newcommand{\Algn}{\mathrm{\Alg}_{E_0}}
\newcommand\bit {\begin{itemize}}
\newcommand\eit {\end{itemize}}
\newcommand{\FPdim}{\mathrm{FPdim}}

\usepackage{tikz}
\usepackage{tikz-cd} 
\tikzset{shorten <>/.style={shorten >=#1,shorten <=#1}}

\title{Algebras over a symmetric fusion category and integrations}

\author[a,b]{Xiao-Xue Wei \thanks{Email: \href{mailto:xxwei@mail.ustc.edu.cn}{\tt xxwei@mail.ustc.edu.cn}}}

\affil[a]{School of Mathematical Sciences, \authorcr University of Science and Technology of China, Hefei, 230026, China} 
\affil[b]{Shenzhen Institute for Quantum Science and Engineering, \authorcr Southern University of Science and Technology, Shenzhen, 518055, China}

\date{\vspace{-5ex}}

\begin{document}
\date{}
\maketitle

\begin{abstract}
We study the symmetric monoidal 2-category of finite semisimple module categories over a symmetric fusion category. In particular, we study $E_n$-algebras in this 2-category and compute their $E_n$-centers for $n=0,1,2$. We also compute the factorization homology of stratified surfaces with coefficients given by $E_n$-algebras in this 2-category for $n=0,1,2$ satisfying certain anomaly-free conditions.
\end{abstract}

\tableofcontents

\section{Introduction}
The mathematical theory of factorization homology is a powerful tool in the study of topological quantum field theories (TQFT). It was first developed by Lurie \cite{Lu} under the name of `topological chiral homology', which records its origin from Beilinson and Drinfeld's theory of chiral homology \cite{BD, FG}. It was further developed by many people (see for example \cite{CG,AF, AFT1, AFT2, AFR, BBJ, BBJ2}) and gained its current name from Francis \cite{F}. 

\medskip
Although the general theory of factorization homology has been well established, explicitly computing the factorization homology in any concrete examples turns out to be a non-trivial challenge. On a connected compact 1-dimensional manifold (or a 1-manifold), i.e. $S^1$, the factorization homology is just the usual Hochschild homology. On a compact 2-manifold, the computation is already highly nontrivial (see for example \cite{BBJ, BBJ2, Francis}). Motivated by the study of topological orders in condensed matter physics, Ai, Kong and Zheng carried out in \cite{LiangFH} the computation of perhaps the simplest (yet non-trivial) kind of factorization homology, i.e. integrating a unitary modular tensor category (UMTC) $\A$ (viewed as an $E_2$-algebra) over a compact 2-manifold $\Sigma$, denoted by $\int_\Sigma \A$. In physics, the category $\A$ is the category of anyons (or particle-like topological defects) in a 2d (spatial dimension) anomaly-free topological order (see \cite{Wen} for a review). The result of this integration is a global observable defined on $\Sigma$. It turns out that this global observable is precisely the ground state degeneracy (GSD) of the 2d topological order on $\Sigma$. This fact remains to be true even if we introduce defects of codimension 1 and 2 as long as these defects are also anomaly-free. Mathematically, this amounts to computing the factorization homology on a disk-stratified 2-manifold with coefficient defined by assigning to each 2-cell a unitary modular tensor category, to each 1-cell a unitary fusion category (an $E_1$-algebra) and to each 0-cell an $E_0$-algebra, satisfying certain anomaly-free conditions (see \cite[Sec.\,4]{LiangFH}).

If the category $\A$ is not modular, i.e. the associated topological order is anomalous, the integral $\int_\Sigma \A$ gives a global observable beyond GSD. Mathematically, it is interesting to compute $\int_\Sigma \A$ for any braided monoidal category $\A$. In this work, we focus on a special situation that also has a clear physical meaning. It was shown in \cite{TL}, a finite onsite symmetry of a 2d symmetry enriched topological (SET) order can be mathematically described by a symmetric fusion category $\E$, and the category of anyons in this SET order can be described by a UMTC over $\E$, which is roughly a unitary braided fusion category with M\"{u}ger center given by $\E$ (see Def.\,\ref{defn-UMTCE} for a precise definition). This motivates us to compute the factorization homology on 2-manifolds but valued in the symmetric monoidal 2-category of finite semisimple module categories over $\E$, denoted by $\cat$. The symmetric tensor product in $\cat$ is defined by the relative tensor product $\boxtimes_\E$. We first study $E_i$-algebras in $\cat$ and their $E_i$-centers for $i=0,1,2$. Then we derive the anomaly-free conditions for $E_i$-algebras in $\cat$ for $i=0,1,2$. In the end, we compute the factorization homology on disk-stratified 2-manifolds with coefficients defined by assigning anomaly-free $E_i$-algebras in $\cat$ to each $i$-cells for $i=0,1,2$. The main results of this work are Thm.\,\ref{main-thm1}, Thm.\,\ref{main-thm2} and Thm.\,\ref{main-thm3}.

\medskip 
The layout of this paper is as follows. In Sec.\,2, we introduce the tensor product $\boe$ and the symmetric monoidal 2-category $\cat$. 
In Sec.\,3, we study $E_i$-algebras in $\cat$ and compute their $E_i$-centers for $i=0,1,2$. In Sec.\,4, we study the modules over a multifusion category over $\E$ and modules over a braided fusion category over $\E$. And we prove that two fusion categories over $\E$ are Morita equivalent in $\cat$ if and only if their $E_1$-centers are equivalent. 
In Sec.\,5, we recall the theory of factorization homology and compute the factorization homology of stratified surfaces with coefficients given by $E_i$-algebras in $\cat$ for $i=0,1,2$ satisfying certain anomaly-free conditions. 

\bigskip
\noindent $\textbf{Acknowledgement}$ I thank Liang Kong for introducing me to this interesting subject. I also thank Zhi-Hao Zhang for helpful discussion. I am supported by NSFC under Grant No. 11971219 and Guangdong Provincial Key Laboratory (Grant No.2019B121203002).

\section{The symmetric monoidal 2-category $\cat$}

\begin{notation}
All categories considered in this paper are small categories.
Let $\bk$ be an algebraically closed field of characteristic zero.
Let $\E$ be a symmetric fusion category over $\bk$ with a braiding $r$.
The category $\vect$ denotes the category of finite dimensional vector spaces over $\bk$ and $\bk$-linear maps.
\end{notation}

Let $\A$ be a monoidal category.
We denote $\A^{\op}$ the monoidal category which has the same tensor product of $\A$,  but the morphism space is given by $\homm_{\A^{\op}}(a, b) \coloneqq \homm_{\A}(b, a)$ for any objects $a, b \in \A$,
and $\A^{\rev}$ the monoidal category which has the same underlying category $\A$ but equipped with the reversed tensor product $a \otimes^{\rev} b \coloneqq b \otimes a$ for $a, b \in \A$. 
 A monoidal category $\A$ is rigid if every object $a \in \A$ has a left dual $a^L$ and a right dual $a^R$.
The duality functors $\delta^L: a \mapsto a^L$ and $\delta^R: a \mapsto a^R$ induce monoidal equivalences $\A^{\op} \simeq \A^{\rev}$.

A braided monoidal category $\A$ is a monoidal category $\A$ equipped with a braiding $c_{a,b}: a \otimes b \rightarrow b \otimes a$ for any $a, b \in \A$.
We denote $\overline{\A}$ the braided monoidal category which has the same monoidal category of $\A$  but equipped with the anti-braiding $\bar{c}_{a,b} = c^{-1}_{b,a}$. 

A fusion subcategory of a fusion category we always mean a full tensor subcategory closed under taking of direct summands. Any fusion category $\A$ contains a trivial fusion subcategory $\vect$.  

\subsection{Module categories}
   Let $\fcat$ be the 2-category of finite semisimple $\bk$-linear abelian categories, $\bk$-linear functors, and natural transformations. The 2-category $\fcat$ equipped with Deligne's tensor product $\boxtimes$, the unit $\vect$ is a symmetric monoidal 2-category.

Let $\C, \D$ be multifusion categories.
We define the 2-category $\LMod_{\C}(\fcat)$ as follows.
\begin{itemize}
\item Its objects are left $\C$-modules in $\fcat$. A left $\C$-module $\M$ in $\fcat$ is an object $\M$ in $\fcat$ equipped with a $\bk$-bilinear functor $\odot: \C \times \M \rightarrow \M$, a natural isomorphism $\lambda_{c, c', m}: (c \otimes c') \odot m \simeq c \odot (c' \odot m)$, and a unit isomorphism $l_m: \unit_{\C} \odot m \simeq m$ for all $c, c' \in \C, m \in \M$ and the tensor unit $\unit_{\C} \in \C$ satisfying some natural conditions.
\item Its 1-morphisms are left $\C$-module functors. 
For left $\C$-modules $\M$, $\N$ in $\fcat$, a left $\C$-module functor from $\M$ to $\N$ is a pair $(F, s^F)$, where $F: \M \rightarrow \N$ is a $\bk$-linear functor and  $s^F_{c,m}: F(c \odot m) \simeq c \odot F(m)$, $c \in \C$, $m \in \M$, is a natural isomorphism, satisfying some natural conditions.
\item Its 2-morphisms are left $\C$-module natural transformations. A left $\C$-module natural transformation between two left $\C$-module functors $(F, s^F), (G, s^G): \M \rightrightarrows \N$ is a natural transformation $\alpha: F \Rightarrow G$ such that the following diagram commutes for $c \in \C, m \in \M$:
\begin{equation}
\label{left-E-m-n}
\begin{split}
\xymatrix{
F(c \odot m) \ar[r]^{s^F} \ar[d]_{\alpha_{c \odot m}} & c \odot F(m) \ar[d]^{1 \odot \alpha_m} \\
G(c \odot m) \ar[r]^{s^G} & c \odot G(m)
}\end{split}
\end{equation}
\end{itemize}
Similarly, one can define the 2-category $\RMod_{\D}(\fcat)$ of right $\D$-modules in $\fcat$ and the 2-category $\BMod_{\C|\D}(\fcat)$ of $\C$-$\D$ bimodules in $\fcat$.
 We use $\Fun(\M, \N)$ to denote the category of $\bk$-linear functors from $\M$ to $\N$ and natural transformations.
 We use $\Fun_{\C}(\M, \N)$ (or $\Fun_{|\C}(\M, \N)$) to denote the category of left (or right) $\C$-module functors from $\M$ to $\N$ and left (or right) $\C$-module natural transformations.

\begin{rem}
There is a bijective correspondence between $\bk$-linear categories (or $\bk$-linear functors) and $\vect$-modules (or $\vect$-module functors).
For objects $\C, \M$ in $\fcat$, if $\odot: \C \times \M \to \M$ is a $\bk$-bilinear functor, it is a balanced $\vect$-module functor.
And a $\bk$-bilinear functor $\odot: \C \times \M \to \M$ is equivalent to a $\bk$-linear functor $\C \boxtimes \M \to \M$ by the universal functor $\boxtimes: \C \times \M \to \C \boxtimes \M$.
\end{rem}

\subsection{Tensor product}
The following definitions are standard (see for example \cite[Def.\,3.1]{ENO}, \cite[Def.\,2.2.1]{Liang}). 
    \begin{defn}
    \label{be-module functor}
     Let $\M \in \RMod_{\E}(\fcat)$, $\N \in \LMod_{\E}(\fcat)$ and $\D \in \fcat$.
        A \emph{balanced $\E$-module functor} is a $\bk$-bilinear functor $F: \M \times \N \rightarrow \D$ equipped with a natural isomorphism $b_{m, e, n}: F(m \odot e, n) \simeq F(m, e \odot n)$ for $m \in \M, n \in \N, e \in \E$, called the \emph{balanced $\E$-module structure} on $F$, such that the diagram 
\begin{equation}
\label{wbx1}
\begin{split}
\xymatrix{
   F(m \odot (e_1 \otimes e_2), n) \ar[rr]^{b_{m, e_1 \otimes e_2, n}} \ar[d]_{\simeq} &  & F(m, (e_1 \otimes e_2) \odot n) \ar[d]^{\simeq} \\
   F((m \odot e_1) \odot e_2, n) \ar[r]^{b_{m \odot e_1, e_2, n}} & F(m \odot e_1, e_2 \odot n) \ar[r]^{b_{m, e_1, e_2 \odot n}} & F(m, e_1 \odot (e_2 \odot n))
   }
\end{split}
\end{equation} 
 commutes for $e_1, e_2 \in \E, m \in \M, n \in \N$. 
   \end{defn}
   A \emph{balanced $\E$-module natural transformation} between two balanced $\E$-module functors $F, G: \M \times \N \rightrightarrows \D$ is a 
   natural transformation $\alpha: F \Rightarrow G$ such that the diagram 
    \[ \xymatrix{
    F(m \odot e, n) \ar[r]^{b^F_{m,e,n}} \ar[d]_{\alpha_{m \odot e, n}} & F(m, e \odot n) \ar[d]^{\alpha_{m, e \odot n}} \\
    G(m \odot e, n) \ar[r]_{b^G_{m,e,n}} & G(m, e \odot n)
    } \]
commutes for all $m \in \M, e \in \E, n \in \N$, where $b^F$ and $b^G$ are the balanced $\E$-module structures on $F$ and $G$ respectively.  
    We use $\Fun_{\E}^{\mathrm{bal}}(\M, \N; \D)$ to denote the category of balanced $\E$-module functors from $\M \times \N$ to $\D$, and balanced $\E$-module natural transformations.

   \begin{defn}
    \label{semi2}
    Let $\M \in \RMod_{\E}(\fcat)$ and $\N \in \LMod_{\E}(\fcat)$.
    The \emph{tensor product} of $\M$ and $\N$ over $\E$ is an object $\M \boxtimes_{\E} \N$ in $\fcat$, together with a balanced $\E$-module functor $\boxtimes_{\E}: \M \times \N \rightarrow \M \boxtimes_{\E} \N$, such that, for every object $\D$ in $\fcat$, composition with $\boxtimes_{\E}$ induces an equivalence of categories
$\Fun(\M \boxtimes_{\E} \N, \D) \simeq \Fun_{\E}^{\bal}(\M, \N; \D)$.
 \end{defn}

\begin{rem}
The tensor product of $\M$ and $\N$ over $\E$ is an object $\M \boxtimes_{\E} \N$ in $\fcat$ unique up to equivalence, together with a balanced $\E$-module functor $\boxtimes_{\E}: \M \times \N \rightarrow \M \boxtimes_{\E} \N$, such that for every object $\D$ in $\fcat$, for any $f \in \Fun^{\bal}_{\E}(\M, \N; \D)$, there exists a pair $(\underline{f}, \eta)$ unique up to isomorphism, such that $f \simeq^{\eta} \underline{f} \circ \boxtimes_{\E}$, i.e.
 \[  \xymatrix{
\M \times \N \ar[r]^{\boxtimes_{\E}} \ar[rd]_{f} & \M \boxtimes_{\E} \N \ar[d]^{\exists ! \underline{f}} \\
& \D   \ultwocell<>{<2>\eta}
} \]
 where $\underline{f}$ is a $\bk$-linear functor in $\Fun(\M \boxtimes_{\E} \N, \D)$, and $\eta: f \Rightarrow \underline{f} \circ \boxtimes_{\E}$ is a balanced $\E$-module natural transformation in $\Fun^{\bal}_{\E}(\M, \N; \D)$.
The notation $\simeq^{\eta}$ means that the natural isomorphism is induced by $\eta$.
 Given two objects $f, g$ and a morphism $a: f \Rightarrow g$ in $\Fun_{\E}^{\bal}(\M, \N; \D)$, there exist unique objects $\underline{f}$, $\underline{g} \in \Fun(\M \boxtimes_{\E} \N, \D)$ such that $f \simeq^{\eta} \underline{f} \circ \boxtimes_{\E}$ and $g \simeq^{\xi} \underline{g} \circ \boxtimes_{\E}$. For any choice of $(a, \eta, \xi, \underline{f}, \underline{g})$, there exists a unique morphism $b: \underline{f} \Rightarrow \underline{g}$ in $\Fun(\M \boe \N, \D)$ such that $\xi \circ a \circ \eta^{-1} = b * \id_{\boxtimes_{\E}}$.
\end{rem}

\subsection{The symmetric monoidal 2-category $\cat$}
A left $\E$-module $\M$ in $\fcat$ is automatically a $\E$-bimodule category with the right $\E$-action defined as $m \odot e \coloneqq e \odot m$, for $m \in \M$, $e \in \E$. 
\begin{defn}
    The 2-category $\cat$ consists of the following data.
    \begin{itemize}
    \item Its objects are left $\E$-modules in $\fcat$.
     \item Its 1-morphisms are left $\E$-module functors.
     \item Its 2-morphisms are left $\E$-module natural transformations.  
     \item The identity 1-morphism $1_{\M}$ for each object $\M$ is identity functor $1_{\M}$. 
     \item The identity 2-morphism $1_F$ for each left $\E$-module functor $F: \M \rightarrow \N$ is the identity natural transformation $1_F$.  
     \item The vertical composition is the vertical composition of left $\E$-module natural transformations.
    \item Horizontal composition of 1-morphisms is the composition of left $\E$-module functors.
    \item Horizontal composition of 2-morphisms is the horizontal composition of left $\E$-module natural transformations. 
    \end{itemize}
 \end{defn}       
    It is routine to check the above data satisfy the axioms (i)-(vi) of \cite[Prop.\,2.3.4]{yau}.    
    We define a pseudo-functor $\boe: \cat \times \cat \rightarrow \cat$ in Sec.\,\ref{m-cate1}. And the following theorem is proved in Sec.\,\ref{m-cate1} and Sec.\,\ref{m-cate2}.
    \begin{thm}
    The 2-category $\cat$ is a symmetric monoidal 2-category.
    \end{thm}

\section{Algebras and centers in $\cat$}

In this section, Sec.\,3.1, Sec.\,3.2 and Sec.\,3.3 study $E_0$-algebras, $E_1$-algebras and $E_2$-algebras in $\cat$, respectively. Sec.\,3.4, Sec.\,3.5 and Sec.\,3.6 study $E_0$-centers, $E_1$-centers and $E_2$-centers in $\cat$, respectively.

 \subsection{$E_0$-algebras}

\begin{defn}
We define the 2-category $\Alg_{E_0}(\cat)$ of $E_0$-algebras in $\cat$ as follows.
\begin{itemize}
\item Its objects are $E_0$-algebras in $\cat$.
An $E_0$-algebra in $\cat$ is a pair $(\A, A)$, where $\A$ is an object in $\cat$ and  $A : \E \to \A$ is a 1-morphism in $\cat$.
\item For two $E_0$-algebras $(\A, A)$ and $(\B, B)$, a 1-morphism $F: (\A, A) \to (\B, B)$ in $\Alg_{E_0}(\cat)$ is a 1-morphism $F: \A \to \B$ in $\cat$ and an invertible 2-morphism $F^0: B \Rightarrow F \circ A$ in $\cat$.
\item For two 1-morphisms $F,G : (\A, A) \rightrightarrows (\B, B)$ in $\Alg_{E_0}(\cat)$, a 2-morphism $\alpha : F \Rightarrow G$ in $\Alg_{E_0}(\cat)$ is a 2-morphism $\alpha : F \Rightarrow G$ in $\cat$ such that $(\alpha * 1_{A}) \circ F^0 = G^0$, i.e.
\be
\begin{array}{c}
\xymatrix{
\E \ar[r]^{A} \ar@/_2ex/[dr]_{B} \drtwocell<\omit>{^<-0.5>F^0} & \A \dtwocell<3>^{G}_{\;\;F}{^\alpha} \\
 & \B
}
\end{array}
=
\begin{array}{c}
\xymatrix{
\E \ar[r]^{A} \ar[dr]_{B} \drtwocell<\omit>{^<-2>G^0} & \A \ar[d]^{G} \\
 & \B
}
\end{array} 
\ee
\end{itemize}  
\end{defn}


\subsection{$E_1$-algebras}

Let $\A$ and $\B$ be two monoidal categories. A monoidal functor from $\A$ to $\B$ is a pair $(F, J^F)$, where $F: \A \to \B$ is a functor and $J^F_{x, y}: F(x \otimes y) \simeq F(x) \otimes F(y)$, $x, y \in \A$, is a natural isomorphism such that $F(\unit_{\A}) = \unit_{\B}$ and a natural diagram commutes.
A \emph{monoidal natural transformation} between two monoidal functors $(F, J^F), (G, J^G): \A \rightrightarrows \B$ is a natural transformation $\alpha: F \Rightarrow G$ such that the following diagram commutes for all $x, y \in \A$:
\begin{equation}
\label{monoidal-na-tr}
\begin{split}
 \xymatrix{
F(x \otimes y) \ar[r]^{J^F_{x,y}} \ar[d]_{\alpha_{x \otimes y}} & F(x) \otimes F(y) \ar[d]^{\alpha_x, \alpha_y} \\
G(x \otimes y) \ar[r]_{J^G_{x,y}} & G(x) \otimes G(y)
} 
\end{split}
\end{equation}

Given a monoidal category $\M$, the \emph{Drinfeld center} of $\M$ is a braided monoidal category $Z(\M)$.
The objects of $Z(\M)$ are pairs $(x, z)$, where $x \in \M$ and $z_{x, m}: x \otimes m \simeq m \otimes x, m \in \M$ is a natural isomorphism such that the following diagram commutes for $m, m' \in \M$:
\[ 
\xymatrixcolsep{0.5pc}
\xymatrixrowsep{1.2pc}
\xymatrix{
x \otimes m \otimes m' \ar[rr]^{z_{x, m \otimes m'}} \ar[rd]_{z_{x, m}, 1}&& m \otimes m' \otimes x  \\
& m \otimes x \otimes m'  \ar[ru]_{1, z_{x, m'}}&
} \]

Recall the two equivalent definitions of a central functor in Def.\,\ref{central defn 1} and Def.\,\ref{central defn 2}. The definitions of a fusion category over $\E$ and a braided fusion category over $\E$ are in \cite{DNO}.
\begin{defn} 
The 2-category $\Alg_{E_1}(\cat)$ consists of the following data.
\begin{itemize}
\item Its objects are multifusion categories over $\E$.  
\emph{A multifusion category over $\E$} is a multifusion category $\A$ equipped with a $\bk$-linear central functor $T_{\A}: \E \rightarrow \A$.  
Equivalently, a multifusion category over $\E$ is a multifusion category $\A$ equipped with a $\bk$-linear braided monoidal functor $T'_{\A}: \E \to Z(\A)$.

\item Its 1-morphisms are monoidal functors over $\E$.
A \emph{monoidal functor over $\E$} between two multifusion categories $\A$, $\B$ over $\E$ is a $\bk$-linear monoidal functor $(F, J): \A \rightarrow \B$ equipped with a monoidal natural isomorphism $u_e: F(T_{\A}(e)) \rightarrow T_{\B}(e)$ in $\B$ for each $e \in \E$, called the structure of monoidal functor over $\E$ on $F$, such that the diagram 
\begin{equation}
\begin{split}
\xymatrix{
F(T_{\A}(e) \otimes x) \ar[r]^{J_{T_{\A}(e), x}} \ar[d]_{F(z_{e, x})} & F(T_{\A}(e)) \otimes F(x) \ar[r]^{u_e, 1} & T_{\B}(e) \otimes F(x) \ar[d]^{\hat{z}_{e, F(x)}}\\
F(x \otimes T_{\A}(e)) \ar[r]^{J_{x, T_{\A}(e)}} & F(x) \otimes F(T_{\A}(e)) \ar[r]^{1, u_e} & F(x) \otimes T_{\B}(e)
}
\end{split}
\end{equation}
commutes for $e \in \E, x \in \A$. Here $z$ and $\hat{z}$ are the central structures of the central functors $T_{\A}: \E \to \A$ and $T_{\B}: \E \to \B$ respectively.

\item Its 2-morphisms are monoidal natural transformations over $\E$.
A \emph{monoidal natural transformation over $\E$} between two monoidal functors $F, G: \A \rightrightarrows \B$ over $\E$ is a monoidal natural transformation $\alpha: F \Rightarrow G$ such that the following diagram commutes for $e \in \E$:
\begin{equation} 
\label{2-morphism-AlgE1}
\xymatrixcolsep{0.5pc}
\xymatrixrowsep{1.2pc}
\xymatrix{
F(T_{\A}(e)) \ar[rr]^{\alpha_{T_{\A}(e)}} \ar[rd]_{u_e}&& G(T_{\A}(e)) \ar[ld]^{v_e} \\
& T_{\B}(e) &
} 
\end{equation}
where $u$ and $v$ are the structures of monoidal functors over $\E$ on $F$ and $G$, respectively.
\end{itemize}
\end{defn}

\begin{rem}
If $\A$ is a multifusion category over $\E$ such that $T'_{\A}: \E \to Z(\A)$ is fully faithful, then $\A$ is a indecomposable.
If $\E = \vect$,  the functor $\vect \to Z(\A)$ is fully faithful if and only if $\A$ is indecomposable.
The condition "$\E \to Z(\A) \to \A$ is fully faithful" implies the condition "$\E \to Z(\A)$ is fully faithful".
\end{rem}

\begin{lem}
\label{EAB}
Let $\A$ and $\B$ be two monoidal categories. Suppose that $T_{\A}: \E \rightarrow \A$, $T_{\B}: \E \rightarrow \B$ and $F: \A \rightarrow \B$ are monoidal functors, and $u: F \circ T_{\A} \Rightarrow T_{\B}$ is a monoidal natural isomorphism. 
Then $\A, \B$ are left $\E$-module categories, $T_{\A}, T_{\B}$ and $F$ are left $\E$-module functors, and $u$ is a left $\E$-module natural isomorphism.
\end{lem}
\begin{proof}
The left $\E$-module structure on $\A$ is defined as $e \odot a \coloneqq T_{\A}(e) \otimes a$ for all $e \in \E$ and $a \in \A$.  The left $\E$-module structure on $T_{\A}$ is induced by the monoidal structure of $T_{\A}$.
The left $\E$-module structure $s^F$ on $F$ is induced by 
$F(e \odot a) = F(T_{\A}(e) \otimes a) \rightarrow F(T_{\A}(e)) \otimes F(a) \xrightarrow{u_e, 1} T_{\B}(e) \otimes F(a) = e \odot F(a)$.
The left $\E$-module structure on $F \circ T_{\A}$ is induced by $F(T_{\A}(\tilde{e} \otimes e)) \rightarrow F(T_{\A}(\tilde{e}) \otimes T_{\A}(e)) \xrightarrow{s^F} T_{\B}(\tilde{e}) \otimes F(T_{\A}(e)) = \tilde{e} \odot F(T_{\A}(e))$ for $e, \tilde{e} \in \E$.
The natural isomorphism $u$ satisfy the diagram (\ref{left-E-m-n}) by the diagram (\ref{monoidal-na-tr}) of the monoidal natural isomorphism $u$.
\end{proof}

\begin{rem}
\label{eq-fusion|E}
A monoidal functor $F: \A \rightarrow \B$ over $\E$ is a left $\E$-module functor.
If $\A$ is a multifusion category over $\E$ and $F: \A \to \B$ is an equivalence of multifusion categories, $\B$ is a multifusion category over $\E$.
The central structure $\sigma$ on the monoidal functor $F \circ T_{\A}: \E \to \B$ is induced by
\[ \xymatrix{
F(T_{\A}(e)) \otimes b \ar@{.>}[d]_{\sigma_{e, b}} \ar[r]^{\simeq} & F(T_{\A}(e)) \otimes F(a) \ar[r]  \ar@{.>}[d]^{\sigma_{e, F(a)}} & F(T_{\A}(e) \otimes a) \ar[d]^{c_{e, a}}  \\
b \otimes F(T_{\A}(e)) & F(a) \otimes F(T_{\A}(e))  \ar[l]^{\simeq} &  F(a \otimes T_{\A}(e)) \ar[l]
} \]
for $e \in \E, b \in \B$, where $c$ is the central structure of the functor $T_{\A}: \E \to \A$.  Notice that for any object $b \in \B$, there is an object $a \in \A$ such that $b \simeq F(a)$ by the equivalence of $F$.
\end{rem}

\begin{expl}
\label{ex1}
If $\C$ is a multifusion category over $\E$, $\C^{\rev}$ is a multifusion category over $\E$ by the central functor $\E = \overline{\E} \xrightarrow{T_{\C}} \overline{Z(\C)} \cong Z(\C^{\rev})$.
\end{expl}

\begin{expl}
\label{ex2}
Let $\M$ be a left $\E$-module in $\fcat$. $\Fun_{\E}(\M, \M)$ is a multifusion category by \cite[Cor.\,9.3.3]{Etingof}. 
Moreover, $\Fun_{\E}(\M, \M)$ is a multifusion category over $\E$.
We define a functor $T: \E \rightarrow \Fun_{\E}(\M, \M), e \mapsto T^e \coloneqq e \odot -$.
The left $\E$-module structure on $T^e$ is defined as $e \odot (\tilde{e} \odot m) \rightarrow (e \otimes \tilde{e}) \odot m \xrightarrow{r_{e, \tilde{e}}, 1} (\tilde{e} \otimes e) \odot m \rightarrow \tilde{e} \odot (e \odot m)$ for $\tilde{e} \in \E$, $m \in \M$.
The monoidal structure $J^T$ on $T$ is induced by $T^{e \otimes e'} = (e \otimes e') \odot -  \simeq e \odot (e' \odot -) = T^e \circ T^{e'}$ for $e, e' \in \E$.
The central structure $\sigma$ on $T$ is induced by $T^e \circ G(m) = e \odot G(m) \simeq G(e \odot m) = G \circ T^e(m)$ for all $e \in \E, G \in \Fun_{\E}(\M, \M)$ and $m \in \M$.
\end{expl}

\begin{expl}
\label{ex3}
Let $\C$ and $\D$ be multifusion categories over $\E$.
$\C \boe \D$ is a multifusion category over $\E$. 
We define a monoidal functor $T_{\C \boe \D}: \E \simeq \E \boe \E \xrightarrow{T_{\C} \boe T_{\D}} \C \boe \D$ by $e \mapsto e \boe \unit_{\E} \mapsto T_{\C}(e) \boe T_{\D}(\unit_{\E}) = T_{\C}(e) \boe \unit_{\D}$ for $e \in \E$.
 And the central structure $\sigma$ on $T_{\C \boe \D}$ is induced by
\[ 
\xymatrix{
T_{\C \boe \D}(e) \otimes (c \boe d) \ar@{.>}[d]_{\sigma_{e, c \boe d}} \ar@{=}[r] & (T_{\C}(e) \boe \unit_{\D}) \otimes (c \boe d) \ar@{=}[r] & (T_{\C}(e) \otimes c) \boe (\unit_{\D} \otimes d) \ar[d]^{z_{e, c} \boe \hat{z}_{\unit_{\D}, d}}  \\
(c \boe d) \otimes T_{\C \boe \D}(e) \ar@{=}[r] & (c \boe d) \otimes (T_{\C} (e) \boe \unit_{\D})  & (c \otimes T_{\C}(e)) \boe (d \otimes \unit_{\D}) \ar@{=}[l]
} \]
for $e \in \E$, $c \boe d \in \C \boe \D$, where $z$ and $\hat{z}$ are the central structures of the functors $T_{\C}: \E \to \C$ and $T_{\D}: \E \to \D$ respectively.
Notice that $T_{\C \boe \D}(e) \simeq 1_{\C} \boe T_{\D}(e)$. 
\end{expl}

An algebra $A$ in a tensor category $\A$ is called \emph{separable} if the multiplication morphism $m: A \otimes A \to A$ splits as a morphism of $A$-bimodules. Namely, there is an $A$-bimodule map $e: A \to A \otimes A$ such that $m \circ e = \id_A$.

\begin{expl}
\label{Ex-A-bimodule}
Let $\C$ be a multifusion category over $\E$ and $A$ a separable algebra in $\C$. 
The category ${{}_A\!{\C}_A}$ of $A$-bimodules in $\C$ is a multifusion category by \cite[Prop.\,2.7]{DMNO}.
 Moreover, ${{}_A\!{\C}_A}$ is a multifusion category over $\E$. We define a functor $I: \E \to {{}_A\!{\C}_A}$, $e \mapsto T_{\C}(e) \otimes A$. 
The left $A$-module structure on the right $A$-module $T_{\C}(e) \otimes A$ is defined as $A \otimes T_{\C}(e) \otimes A \xrightarrow{c^{-1}_{e, A}, 1} T_{\C}(e) \otimes A \otimes A \to T_{\C}(e) \otimes A$, where $c$ is the central structure of the functor $T_{\C}: \E \to \C$.
The monoidal structure on $I$ is defined as $T_{\C}(e_1 \otimes e_2) \otimes A \simeq T_{\C}(e_1) \otimes T_{\C}(e_2) \otimes A \cong T_{\C}(e_1) \otimes A \otimes_A T_{\C}(e_2) \otimes A$ for $e_1, e_2 \in \E$.
The central structure on $I$ is induced by
\[ I(e) \otimes_A x = T_{\C}(e) \otimes A \otimes_A x \xrightarrow{c_{e, A \otimes_A x}} A \otimes_A x \otimes T_{\C}(e) \cong x \otimes_A A \otimes T_{\C}(e) \xrightarrow{1, c^{-1}_{e, A}} x \otimes_A T_{\C}(e) \otimes A = x \otimes_A I(e) \]
for $e \in \E, x \in {{}_A\!{\C}_A}$.
\end{expl}

\subsection{$E_2$-algebras}

Let $\A$ be a subcategory of a braided fusion category $\C$.
The \emph{centralizer of $\A$ in $\C$}, denoted by $\A'|_{\C}$, is defined by the full subcategory of objects $x \in \C$ such that $c_{a, x} \circ c_{x, a} = \id_{x \otimes a}$ for all $a \in \A$, where $c$ is the braiding of $\C$.
The \emph{M\"{u}ger center} of $\C$, denoted by $\C'$ or $\C'|_{\C}$, is the centralizer of $\C$ in $\C$. 
Let $\B$ be a fusion category over $\E$ such that $\E \to Z(\B)$ is fully faithful. The centralizer of $\E$ in $Z(\B)$ is denoted by $Z(\B, \E)$ or $\E'|_{Z(\B)}$.

\begin{defn} 
The 2-category $\Alg_{E_2}(\cat)$ consists of the following data.
\begin{itemize}
\item Its objects are braided fusion categories over $\E$. A \emph{braided fusion category over $\E$} is a braided fusion category $\A$ equipped with a $\bk$-linear braided monoidal embedding $T_{\A}: \E \rightarrow \A'$.
A braided fusion category $\A$ over $\E$ is non-degenerate if $T_{\A}$ is an equivalence. 
\item Its 1-morphisms are braided monoidal functors over $\E$.
A \emph{braided monoidal functor over $\E$} between two braided fusion categories $\A, \B$ over $\E$ is a $\bk$-linear braided monoidal functor  $F: \A \rightarrow \B$ equipped with a monoidal natural isomorphism $u_e: F(T_{\A}(e)) \simeq T_{\B}(e)$ in $\B$ for all $e \in \E$.

\item For two braided monoidal functors $F,G: \A \rightrightarrows \B$ over $\E$, a 2-morphism from $F$ to $G$ is a monoidal natural transformation $\alpha: F \Rightarrow G$ such that the diagram (\ref{2-morphism-AlgE1}) commutes.
\end{itemize}
\end{defn}

\begin{rem}
Let $\A$ be a braided fusion category over $\E$ and $\eta: \A \simeq \B$ is an equivalence of braided fusion categories. Then $\B$ is a braided fusion category over $\E$.
\end{rem}

\begin{expl}
If $\D$ is a braided fusion category over $\E$, $\overline{\D}$ is a braided fusion category over $\E$ by the braided monoidal embedding $\E = \overline{\E} \xrightarrow{T_{\D}} \overline{\D'} = \overline{\D}'$.
\end{expl}

\begin{expl}
\label{center-c}
Let $\C$ be a fusion category over $\E$ such that $\E \to Z(\C)$ is fully faithful. 
$Z(\C, \E)$ is a non-degenerate braided fusion category over $\E$.
Next check that $Z(\C, \E)' = \E$. On one hand, if $e \in \E$, we have $T_{\C}(e) \in Z(\C, \E)'$. On the other hand, since $Z(\C)' = \vect \subset \E$, we have $Z(\C, \E)'|_{Z(\C, \E)} \subset Z(\C, \E)'|_{Z(\C)} = (\E'|_{Z(\C)})'|_{Z(\C)} = \E$. 
The central structure on $T_{Z(\C, \E)}: \E \rightarrow Z(\C, \E)'$ is defined as $T_{\C}$. 

If $\C$ is a non-degenerate braided fusion category over $\E$, there is a braided monoidal equivalence $Z(\C, \E) \simeq \C \boxtimes_{\E} \overline{\C}$ over $\E$ by \cite[Cor.\,4.4]{DNO}.
\end{expl}

\subsection{$E_0$-centers}
A \emph{contractible groupoid} is a non-empty category in which there is a unique morphism between any two objects.
An object $\X$ in a monoidal 2-category $\BB$ is called a \emph{terminal object} if for each $\Y \in \BB$, the hom category $\BB(\Y, \X)$ is a contractible groupoid.
Here the hom category $\BB(\Y, \X)$ denotes the category of 1-morphisms from $\Y$ to $\X$ and 2-morphisms in $\BB$.

\begin{defn}
Let $\A = (\A, A) \in \Algn(\cat)$. A \emph{left unital $\A$-action on $\X \in \cat$} is a 1-morphism $F : \A \boe \X \to \X$ in $\cat$ together with an invertible 2-morphism $\alpha$ in $\cat$ as depicted in the following diagram:
\[
\begin{array}{c}
\xymatrix@=4ex{
 & \A \boe  \X \ar[dr]^{F} \\
\E \boe \X \ar[ur]^{A \boe 1_{\X}} \ar[rr] \rrtwocell<\omit>{<-2>\alpha} & & \X
}
\end{array} ,
\]
where the unlabeled arrow is given by the left $\E$-action on $\X$.
\end{defn}

\begin{defn}
Let $\X \in \cat$. The 2-category $\Algn(\cat)_{\X}$ of left unital actions on $\X$ in $\Algn(\cat)$ is defined as follows.
\bit
\item The objects are left unital actions on $\X$.
\item Let $((\A, A), F, \alpha_{\A})$ be a left unital $(\A, A)$-action on $\X$ and $((\B, B), G, \alpha_{\B})$ be a left unital $(\B, B)$-action on $\X$.
A 1-morphism $(P, \rho): ((\A, A), F,\alpha_{\A}) \to ((\B, B), G, \alpha_{\B})$ in $\Algn(\cat)_{\X}$ is a 1-morphism $P: (\A, A) \to (\B, B)$ in $\Algn(\cat)$, equipped with an invertible 2-morphism $\rho: G \circ (P \boe 1_{\X}) \Rightarrow F$ in $\cat$,  
such that the following pasting diagram equality holds.
\[
\begin{array}{c}
\xymatrix{
 & \B \boe \X \ar@/^4ex/[ddr]^{G} \\
 & \A \boe \X \ar[dr]^{F} \ar[u]|{P \boe 1_{\X}} \rtwocell<\omit>{\rho} & \\
\E \boe \X \ar[rr] \ar@/^4ex/[uur]^{B \boe 1_{\X}} \ar[ur]^{A \boe 1_{\X}} \rrtwocell<\omit>{<-3>\;\;\alpha_{\A}} \uurtwocell<\omit>{P^0 \boe 1\;\;\;} & & \X
}
\end{array}
=
\begin{array}{c}
\xymatrix{
 & \B \boe \X \ar[dr]^{G} \\
\E \boe \X \ar[ur]^{B \boe 1_{\X}} \ar[rr] \rrtwocell<\omit>{<-3>\;\;\alpha_{\B}} & & \X
}
\end{array} 
\]
Here we choose the identity 2-morphism $\id: (P \boe 1_{\X}) \circ (A \boe 1_{\X}) \Rightarrow (P \circ A) \boe 1_{\X}$ for convenience.
\item Given two 1-morphisms $(P,\rho), (Q,\sigma) : ((\A, A), F,\alpha_{\A}) \rightrightarrows ((\B, B), G,\alpha_{\B})$, a 2-morphism $\alpha : (P,\rho) \Rightarrow (Q,\sigma)$ in $\Algn(\cat)_{\X}$ is a 2-morphism $\alpha : P \Rightarrow Q$ in $\Algn(\cat)$ such that the following pasting diagram equality holds.
\[
\begin{array}{c}
\xymatrix{
\A \boe \X \rrtwocell^{P \boe 1_{\X}}_{Q \boe 1_{\X}}{\;\;\;\;\;\alpha \boe 1} \ar[dr]_{F} \drrtwocell<\omit>{<1>\sigma} & & \B \boe \X \ar[dl]^{G} \\
 & \X &
}
\end{array}
=
\begin{array}{c}
\xymatrix{
\A \boe \X \ar[rr]^{P \boe 1_{\X}} \ar[dr]_{F} \drrtwocell<\omit>{<-1>\rho} & & \B \boe \X \ar[dl]^{G} \\
 & \X &
}
\end{array} 
\]
\eit
An \emph{$E_0$-center of the object $\X$ in $\cat$} is a terminal object in $\Algn(\cat)_{\X}$.
\end{defn}

\begin{thm}
The $E_0$-center of a category $\X \in \cat$ is given by the multifusion category $\Fun_{\E}(\X, \X)$ over $\E$.
\end{thm}
\begin{proof}
Suppose $(\A, A)$ is an $E_0$-algebra in $\cat$ and $(F, u)$ as depicted in the following diagram 
\[ \xymatrix@=3ex{
\rrtwocell<>{<4>u}& \A \boe \X \ar[rd]^(0.56){F} & \\
\E \boe \X \ar[rr] \ar[ru]^(0.42){A \boe 1_{\X}} & & \X
} \]
is a unital $\A$-action on $\X$.
In other words, $F: \A \boe \X \rightarrow \X$ is a left $\E$-module functor and $u_{e, x}: F(A(e) \boe x) \rightarrow e \odot x$, $e \in \E, x \in \X$ is a natural isomorphism in $\cat$.

Recall that $(\Fun_{\E}(\X, \X), T)$ is an $E_0$-algebra in $\cat$ by Expl.\,\ref{ex2}.
\[ \xymatrix@=3ex{
\rrtwocell<>{<4>v} & \Fun_{\E}(\X, \X) \boe \X \ar[rd]^(0.6){G} & \\
\E \boe \X \ar[ru]^(0.4){T \boe 1_{\X}} \ar[rr] & & \X
} \]
Define a functor 
\[G: \Fun_{\E}(\X, \X) \boe \X \rightarrow \X, \qquad f \boe x \mapsto f(x)\] 
and a natural isomorphism 
\[v_{e, x} = \id_{e \odot x}: G(T^e \boe x) = T^e(x) = e \odot x \rightarrow e \odot x, \quad e \in \E, x \in \X.\] 
Then $((\Fun_{\E}(\X, \X), T), G, v)$ is a left unital $\Fun_{\E}(\X, \X)$-action on $\X$. 

We want to show that $\Alg_{E_0}(\cat)_{\X}(\A, \Fun_{\E}(\X, \X))$ is a contractible groupoid.
First we want to show there exists a 1-morphism $(P, \rho): \A \rightarrow \Fun_{\E}(\X, \X)$ in $\Alg_{E_0}(\cat)_{\X}$. We define a functor $P: \A \rightarrow \Fun_{\E}(\X, \X)$ by $P(a) \coloneqq  F(a \boe -)$ for all $a \in \A$ and an invertible 2-morphism $P^0_e: T^e = e \odot - \Rightarrow P(A(e)) = F(A(e) \boe -)$ as $u^{-1}_e$ for all $e \in \E$. 
The natural isomorphism $\rho$ can be defined by
\[\rho_{a, x} = \id_{F(a \boe x)}: G(P(a) \boe x) = P(a)(x) = F(a \boe x) \rightarrow F(a \boe x)\] 
for $a \in \A, x \in \X$.
Then it suffices to show that the composition of morphisms
\[ G(T^e \boe x) = e \odot x \xRightarrow{(P^0_e)_x = u^{-1}_{e, x}} F(A(e) \boe x) \xRightarrow{\rho_{A(e), x}= \id_{F(A(e) \boe x)}} F(A(e) \boe x) \xRightarrow{u_{e, x}} e \odot x \]
is equal to $v_{e, x} = \id_{e \odot x}$ by the definitions of $P^0$ and $\rho$. 

Then we want to show that if there are two 1-morphisms $(Q_i, \sigma_i): \A \rightarrow \Fun_{\E}(\X, \X)$ in $\Alg_{E_0}(\cat)_{\X}$ for $i=1,2$, there is a unique 2-morphism $\beta: (Q_1, \sigma_1) \Rightarrow (Q_2, \sigma_2)$ in $\Alg_{E_0}(\cat)_{\X}$.
The 2-morphism $\beta$ in $\Alg_{E_0}(\cat)_{\X}$ is a natural isomorphism $\beta: Q_1 \Rightarrow Q_2$ such that the equalities
\begin{equation} \label{eq:E0_cat_1}
\bigl( T \xRightarrow{Q_1^0} Q_1 \circ A \xRightarrow{\beta * 1_A} Q_2 \circ A \bigr) = \bigl( T \xRightarrow{Q_2^0} Q_2 \circ A \bigr)
\end{equation}
and
\begin{equation} \label{eq:E0_cat_2}
\bigl( Q_1(a)(x) \xrightarrow{(\beta_a)_x} Q_2(a)(x) \xrightarrow{(\sigma_2)_{a,x}} F(a \boe x) \bigr) = \bigl( Q_1(a)(x) \xrightarrow{(\sigma_1)_{a,x}} F(a \boe x) \bigr)
\end{equation}
hold for $a \in \A, x \in \X$.
The second condition (\ref{eq:E0_cat_2}) implies that $(\beta_a)_x = (\sigma_2)_{a, x}^{-1} \circ (\sigma_1)_{a, x}$.
This proves the uniqueness of $\beta$.
For the existence of $\beta$, we want to show that $\beta$ satisfy the first condition (\ref{eq:E0_cat_1}), i.e. $\beta$ is a 2-morphism in $\Alg_{E_0}(\cat)$.
Since $(Q_i, \sigma_i)$ are 1-morphisms in $\Alg_{E_0}(\cat)_{\X}$, the composed morphism 
\[ e \odot x = T^e(x) \xrightarrow{(Q^0_i)_{e, x}} Q_i(A(e))(x) \xrightarrow{(\sigma_i)_{A(e), x}} F(A(e) \boe x) \xrightarrow{u_{e, x}} e \odot x \]
is equal to $v_{e,x} = \id_{e \odot x}$.
It follows that the composition of morphisms
\[ e \odot x \xrightarrow{(Q^0_1)_{e, x}} Q_1(A(e))(x) \xrightarrow{(\sigma_1)_{A(e), x}} F(A(e) \boe x) \xrightarrow{(\sigma_2)^{-1}_{A(e), x}} Q_2(A(e))(x) \xrightarrow{(Q^0_2)^{-1}_{e,x}} e \odot x \]
is equal to $\id_{e \odot x}$, i.e. $(Q^0_2)^{-1}_{e,x} \circ (\beta_{A(e)})_x \circ (Q^0_1)_{e,x} = \id_{e \odot x}$.
This is precisely the first condition (\ref{eq:E0_cat_1}).
Hence the natural transformation $\beta: Q_1 \Rightarrow Q_2$ defined by $(\beta_a)_x = (\sigma_2)^{-1}_{a, x} \circ (\sigma_1)_{a,x}$ is the unique 2-morphism $\beta: (Q_1, \sigma_1) \Rightarrow (Q_2, \sigma_2)$.

Finally, we also want to verify that the $E_1$-algebra structure on the $E_1$-center $\Fun_{\E}(\X, \X)$ coincides with the usual monoidal structure of $\Fun_{\E}(\X, \X)$ defined by the composition of functors.
Recall that the $E_1$-algebra structure is induced by the iterated action
\[ \Fun_{\E}(\X, \X) \boe \Fun_{\E}(\X, \X) \boe \X \xrightarrow{1 \boe G} \Fun_{\E}(\X, \X) \boe \X \xrightarrow{G} \X \]
By the construction given above, the induced tensor product $\Fun_{\E}(\X, \X) \boe \Fun_{\E}(\X, \X) \to \Fun_{\E}(\X, \X)$ is given by $f \boe g \mapsto G(f \boe G(g \boe -)) = f(g(-)) = f \circ g$.
Hence, the $E_1$-algebra structure on $\Fun_{\E}(\X, \X)$ is the composition of functors, which is the usual monoidal structure on $\Fun_{\E}(\X, \X)$.
\end{proof}

\subsection{$E_1$-centers}

\begin{defn}
Let $\X \in \Alg_{E_1}(\cat)$. The \emph{$E_1$-center of $\X$ in $\cat$} is the $E_0$-center of $\X$ in $\Alg_{E_1}(\cat)$.
\end{defn}

\begin{thm}
\label{thm-3.19}
Let $\B$ be a multifusion category over $\E$. Then the $E_1$-center of $\B$ in $\cat$ is the braided multifusion category $Z(\B, \E)$ over $\E$.
\end{thm}
\begin{proof}
Let $\A$ be a multifusion category over $\E$.
A left unital $\A$-action on $\B$ in $\Alg_{E_1}(\cat)$ is a monoidal functor $F: \A \boe \B \rightarrow \B$ over $\E$ and a monoidal natural isomorphism $u$ over $\E$ shown below:
\[ \xymatrix@=3ex{
\rrtwocell<>{<4>u}& \A \boe \B \ar[rd]^(0.56){F} & \\
\E \boe \B \ar[rr]_{\odot} \ar[ru]^(0.42){T_{\A}, 1} & & \B
} \]
More precisely, $F$ is a functor equipped with natural isomorphisms $J^F: F(a_1 \boe b_1) \otimes F(a_2 \boe b_2) \xrightarrow{\simeq} F((a_1 \boe b_1) \otimes (a_2 \boe b_2))$, $a_1, a_2 \in \A$, $b_1, b_2 \in \B$, and $I^F: \unit_{\B} \xrightarrow{\simeq} F(\unit_{\A} \boe \unit_{\B})$ satisfying certain commutative diagrams.
The monoidal structure on the functor $\odot: \E \boe \B \to \B$, $e \boe b \mapsto e \odot b = T_{\B}(e) \otimes b$ is induced by
$T_{\B}(e_1 \otimes e_2) \otimes (b_1 \otimes b_2) \simeq T_{\B}(e_1) \otimes T_{\B}(e_2) \otimes b_1 \otimes b_2 \xrightarrow{1, z_{e_2, b_1}, 1} T_{\B}(e_1) \otimes b_1 \otimes T_{\B}(e_2) \otimes b_2$ for $e_1, e_2 \in \E, b_1, b_2 \in \B$, where $(T_{\B}(e_2), z) \in Z(\B)$. 
The structure of monoidal functor over $\E$ on $\odot$ is defined as $\odot(T_{\E \boe \B}(e)) = \odot(e \boe \unit_{\B}) = e \odot \unit_{\B} = T_{\B}(e) \otimes \unit_{\B} \simeq T_{\B}(e)$.
And $u$ is a monoidal natural isomorphism $u_{e, b}: F(T_{\A}(e) \boe b) \xrightarrow{\simeq} e \odot b \coloneqq T_{\B}(e) \otimes b$, $e \in \E, b \in \B$. Also one can show that $I^F = u^{-1}_{\unit_{\E}, \unit_{\B}}$.
The structure of monoidal functor over $\E$ on $F$ is $u_{e, \unit_{\B}}: F(T_{\A \boe \B}(e)) = F(T_{\A}(e) \boe \unit_{\B}) \simeq T_{\B}(e) \otimes \unit_{\B} \simeq T_{\B}(e)$.

There is an obviously left unital $Z(\B, \E)$-action on $\B$
\[ \xymatrix@=3ex{
\rrtwocell<>{<4>v}& Z(\B, \E) \boe \B \ar[rd]^(0.6){G} & \\
\E \boe \B \ar[rr]_{\odot} \ar[ru]^(0.4){T_{\B}, 1} & & \B
} \]
defined by $G: Z(\B, \E) \boe \B \xrightarrow{\forget, 1} \B \boe \B \xrightarrow{\otimes} \B$ and $v_{e, b} \coloneqq \id_{T_{\B}(e) \otimes b}: G(T_{\B}(e) \boe b) = T_{\B}(e) \otimes b \to e \odot b$, for $e \in \E, b \in \B$.
The structure of monoidal functor over $\E$ on $G$ is defined as 
$G(T_{\B}(e) \boe \unit_{\B}) = T_{\B}(e) \otimes \unit_{\B} \simeq T_{\B}(e)$.

First we want to show that $F(a \boe \unit_{\B}) \in Z(\B, \E)$ for $a \in \A$.
Notice that $F(\unit_{\A} \boe b) = F(T_{\A}(\unit_{\E}) \boe b) \xrightarrow{u_{\unit_{\E}, b}} \unit_{\E} \odot b = b$.
Since $F$ is a monoidal functor over $\E$, it can be verified that the natural transformation $\gamma$ (shown below)
\[ \xymatrix{
F(a \boe \unit_{\B}) \otimes b \ar[r]^(0.4){1, u^{-1}_{\unit_{\E}, b}} \ar@{.>}[d]_{\gamma_{a, b}}& F(a \boe \unit_{\B}) \otimes F(\unit_{\A} \boe b) \ar[r]^{J^F} & F((a \otimes \unit_{\A}) \boe (\unit_{\B} \otimes b)) \ar[d]^{\simeq, \simeq} \\
b \otimes F(a \boe \unit_{\B})   & F(\unit_{\A} \boe b) \otimes F(a \boe \unit_{\B}) \ar[l]^(0.6){u_{\unit_{\E}, b}, 1}& F((\unit_{\A} \otimes a) \boe (b \otimes \unit_{\B}))  \ar[l]^{J^F}
} \]
is a half-braiding on $F(a \boe \unit_{\B}) \in \B$, for $a \in \A, b \in \B$.
It is routine to check that the composition $T_{\B}(e) \otimes F(a \boe \unit_{\B}) \to F(a \boe \unit_{\B}) \otimes T_{\B}(e) \to T_{\B}(e) \otimes F(a \boe \unit_{\B})$ equals to identity. Then $F(a \boe \unit_{\B})$ belongs to $Z(\B, \E)$.

We define a monoidal functor $P: \A \to Z(\B, \E)$ by $P(a) \coloneqq (F(a \boe \unit_{\B}), \gamma_{a, -})$ with the monoidal structure induced by that of $F$:
\[ J^P: \big( P(a_1) \otimes P(a_2) = F(a_1 \boe \unit_{\B}) \otimes F(a_2 \boe \unit_{\B}) \xrightarrow{J^F} F((a_1 \otimes a_2) \boe (\unit_{\B} \otimes \unit_{\B})) = F((a_1 \otimes a_2) \boe \unit_{\B}) = P(a_1 \otimes a_2) \big) \]
\[ I^P: \big( \unit_{\B} \xrightarrow{I^F} F(\unit_{\A} \boe \unit_{\B}) = P(\unit_{\A}) \big) \]
The structure of monoidal functor over $\E$ on $P$ is defined as $u_{e, \unit_{\B}}: P(T_{\A}(e)) = F(T_{\A}(e) \boe \unit_{\B}) \simeq T_{\B}(e) = T_{Z(\B, \E)}(e)$ for $e \in \E$.

Then we show that there exists a 1-morphism $(P, \rho): \A \rightarrow Z(\B, \E)$ in $\Alg_{E_0}(\Alg_{E_1}(\cat))_{\B}$.
The invertible natural isomorphism $P^0: T_{\B} \Rightarrow P \circ T_{\A}$ is defined by $T_{\B}(e) = e \odot \unit_{\B} \xrightarrow{u^{-1}_{e, \unit_{\B}}} F(T_{\A}(e) \boe \unit_{\B}) = P(T_{\A}(e))$ for $e \in \E$.
The monoidal natural isomorphism $\rho: G \circ (P \boe 1_{\B}) \Rightarrow F$ is defined by 
\[ \rho_{a,b}: F(a \boe \unit_{\B}) \otimes b \xrightarrow{1, u^{-1}_{\unit_{\E}, b}} F(a \boe \unit_{\B}) \otimes F(\unit_{\A} \boe b) \xrightarrow{J^F} F((a \otimes \unit_{\A}) \boe (\unit_{\B} \otimes b)) = F(a \boe b) \]
for $a \in \A, b \in \B$. 
It is routine to check that the composition of 2-morphisms $P^0, \rho$ and $u$ is equal to the 2-morphism $v$.

Then we show that if there are two 1-morphisms $(Q_i, \sigma_i): \A \to Z(\B, \E)$ in $\Alg_{E_0}(\Alg_{E_1}(\cat))_{\B}$ for $i=1, 2$, then there exists a unique 2-morphism $\beta: (Q_1, \sigma_1) \Rightarrow (Q_2, \sigma_2)$ in $\Alg_{E_0}(\Alg_{E_1}(\cat))_{\B}$. Such a $\beta$ is a natural transformation $\beta: Q_1 \Rightarrow Q_2$ such that the equalities
\begin{equation} \label{eq:E1_cat1}
\big( Q_1(a) \otimes b \xrightarrow{\beta_a, 1} Q_2(a) \otimes b \xrightarrow{(\sigma_2)_{a, b}} F(a \boe b) \big) = \big( Q_1(a) \otimes b \xrightarrow{(\sigma_1)_{a, b}} F(a \boe b) \big) 
\end{equation}
and
\begin{equation} \label{eq:E1_cat2}
 \big( T_{\B} \xRightarrow{Q_1^0} Q_1 \circ T_{\A} \xRightarrow{\beta * 1} Q_2 \circ T_{\A} \big) = \big( T_{\B} \xRightarrow{Q^0_2} Q_2 \circ T_{\A} \big) 
 \end{equation}
hold for $a \in \A, b \in \B$.
The first condition (\ref{eq:E1_cat1}) implies that $\beta_a: Q_1(a) \to Q_2(a)$ is equal to the composition
\[ Q_1(a) = Q_1(a) \otimes \unit_{\B} \xrightarrow{(\sigma_1)_{a, \unit_{\B}}} F(a \boe \unit_{\B}) \xRightarrow{(\sigma_2)^{-1}_{a, \unit_{\B}}} Q_2(a) \otimes \unit_{\B} = Q_2(a) \]
This proves the uniqueness of $\beta$.
It is routine to check that $\beta_a$ is a morphism in $Z(\B, \E)$ and $\beta$ satisfy the second condition (\ref{eq:E1_cat2}).

Finally, we also want to verify that the $E_2$-algebra structure on the $E_1$-center $Z(\B, \E)$ coincides with the usual braiding structure on $Z(\B, \E)$.
The $E_2$-algebra structure is given by the monoidal functor $H: Z(\B, \E) \boe Z(\B, \E) \to Z(\B, \E)$, which is induced by the iterated action 
\[ Z(\B, \E) \boe Z(\B, \E) \boe \B \xrightarrow{1, G} Z(\B, \E) \boe \B \xrightarrow{G} \B \]
with the monoidal structure given by 
\[ x_1 \otimes x_2 \otimes y_1 \otimes y_2 \otimes b_1 \otimes b_2 \xrightarrow{\gamma_{y_2, b_1}} x_1 \otimes x_2 \otimes y_1 \otimes b_1 \otimes y_2 \otimes b_2 \xrightarrow{\gamma_{x_2, y_1 \otimes b_1}} x_1 \otimes y_1 \otimes b_1 \otimes x_2 \otimes y_2 \otimes b_2   \]
for $x_1 \boe y_1 \boe b_1, x_2 \boe y_2 \boe b_2$ in $Z(\B, \E) \boe Z(\B, \E) \boe \B$.
Then by the construction given above, the induced functor $H: Z(\B, \E) \boe Z(\B, \E) \to Z(\B, \E)$ maps $x \boe y$ to the object $G((1 \boe G)(x \boe y \boe \unit_{\B})) = x \otimes y \otimes \unit_{\B} = x \otimes y$ with the half-braiding
\[ x \otimes y \otimes b \xrightarrow{\gamma_{y, b}} x \otimes b \otimes y \xrightarrow{\gamma_{x, b}} b \otimes x \otimes y \]
Thus the functor $H$ coincides with the tensor product of $Z(\B, \E)$.
For $x_1 \boe y_1, x_2 \boe y_2 \in Z(\B, \E) \boe Z(\B, \E)$, the monoidal structure of $H$ is induced by
\[ H((x_1 \boe y_1) \otimes (x_2 \boe y_2)) = x_1 \otimes x_2 \otimes y_1 \otimes y_2 \xrightarrow{\gamma_{x_2, y_1}} x_1 \otimes y_1 \otimes x_2 \otimes y_2 = H(x_1 \boe y_1) \otimes H(x_2 \boe y_2) \]
Equivalently, the braiding structure on $Z(\B, \E)$ is given by $x \otimes y \xrightarrow{\gamma_{x, y}} y \otimes x$, which is the usual braiding structure on $Z(\B, \E)$.
\end{proof}

\subsection{$E_2$-centers}

\begin{defn}
Let $\X \in \Alg_{E_2}(\cat)$. The \emph{$E_2$-center of $\X$ in $\cat$} is the $E_0$-center of $\X$ in $\Alg_{E_2}(\cat)$.
\end{defn}

\begin{thm}
Let $\C$ be a braided fusion category over $\E$.
The $E_2$-center of $\C$ is the symmetric fusion category $\C'$ over $\E$.
\end{thm}
\begin{proof}
Let $\A$ be a braided fusion category over $\E$. A left unital $\A$-action on $\C$ is a braided monoidal functor $F: \A \boe \C \to \C$ over $\E$ and a monoidal natural isomorphism $u$ over $\E$ shown below:
\[ \xymatrix@=3ex{
\rrtwocell<>{<4>u}& \A \boe \C \ar[rd]^(0.56){F} & \\
\E \boe \C \ar[rr]_{\odot} \ar[ru]^(0.42){T_{\A}, 1} & & \C
} \]
More precisely, $F$ is a monoidal functor over $\E$ (recall the proof of Thm.\,\ref{thm-3.19}) such that the diagram 
\[ \xymatrix{
F(a_1 \boe x_1) \otimes F(a_2 \boe x_2) \ar[r]^{J^F} \ar[d]_{c_{F(a_1 \boe x_1), F(a_2 \boe x_2)}} & F((a_1 \otimes a_2) \boe (x_1 \otimes x_2)) \ar[d]^{\tilde{c}_{a_1, a_2}, c_{x_1, x_2}} \\
F(a_2 \boe x_2) \otimes F(a_1 \boe x_1) \ar[r]_{J^F} & F((a_2 \otimes a_1) \boe (x_2 \otimes x_1))
} \]
commutes for $a_1, a_2 \in \A$, $x_1, x_2 \in \C$, where $\tilde{c}$ and $c$ are the half-braidings of $\A$ and $\C$ respectively.
The braided structure on $\E \boe \C$ is defined as 
$T_{\C}(e_1 \otimes e_2) \otimes x_1 \otimes x_2 \xrightarrow{r_{e_1, e_2}, c_{x_1, x_2}} T_{\C}(e_2 \otimes e_1) \otimes x_2 \otimes x_1$, for $e_1 \boe x_1, e_2 \boe x_2 \in \E \boe \C$. Check that $\odot: \E \boe \C \to \C$ is a braided functor.

There is a left unital $\C'$-action on $\C$ 
\[ \xymatrix@=3ex{
\rrtwocell<>{<4>v} & \C' \boe \C \ar[rd]^(0.6){G} & \\
\E \boe \C \ar[ru]^(0.4){T_{\C}, 1} \ar[rr] & & \C
} \]
given by
$G: \C' \boe \C \to \C, (z, x) \mapsto z \otimes x$
and $v_{e, x} \coloneqq \id_{e \odot x} : G(T_{\C}(e) \boe x) = T_{\C}(e) \otimes x \to e \odot x$.

Next we want to show that there exists a 1-morphism $(P, \rho): \A \to \C'$ in $\Alg_{E_0}(\Alg_{E_2}(\cat))_{\C}$. Since $F$ is a braided monoidal functor over $\E$, the commutative diagram
\[ \xymatrix{
F(a \boe \unit_{\C}) \otimes x \ar[d]_{c_{F(a \boe \unit_{\C}), x}} \ar[r]^(0.4){1, u^{-1}_{\unit_{\E},x}} & F(a \boe \unit_{\C}) \otimes F(\unit_{\A} \boe x) \ar[d]_{c_{F(a \boe \unit_{\C}), F(\unit_{\A} \boe x)}} \ar[r]^(0.6){J^F} & F(a  \boe x) \ar[d] \ar@/^2pc/[dd]^{1} \\
x \otimes F(a \boe \unit_{\C}) \ar[d]_{c_{x, F(a \boe \unit_{\C})}} \ar[r]^(0.4){u^{-1}_{\unit_{\E},x}, 1}  & F(\unit_{\A} \boe x) \otimes F(a \boe \unit_{\C}) \ar[d]_{c_{F(\unit_{\A} \boe x), F(a \boe \unit_{\C})}} \ar[r]^(0.6){J^F} & F(a \boe x)  \ar[d] \\
F(a \boe \unit_{\C}) \otimes x \ar[r]_(0.4){1, u^{-1}_{\unit_{\E},x}} & F(a \boe \unit_{\C}) \otimes F(\unit_{\A} \boe x) \ar[r]^(0.6){J^F} & F(a \boe x)
} \]
 implies that the equality $c_{x, F(a \boe \unit_{\C})} \circ c_{F(a \boe \unit_{\C}), x} = \id_{F(a \boe \unit_{\C}) \otimes x}$ holds for $a \in \A$, $x \in \C$, i.e. $F(a \boe \unit_{\C}) \in \C'$. Then we define the functor $P$ by $P(a) \coloneqq F(a \boe \unit_{\C})$, and the monoidal structure of $P$ is induced by that of $F$.
The monoidal natural isomorphism $\rho: G \circ (P \boe 1_{\C}) \Rightarrow F$ is  defined by 
\[ \rho_{e, x}: F(a \boe \unit_{\C}) \otimes x  \xrightarrow{1 \otimes u^{-1}_{\unit_{\E}, x}} F(a \boe \unit_{\C}) \otimes F(\unit_{\A} \boe x) \xrightarrow{J^F} F(a \boe x) \]
Then $(P, \rho)$ is a 1-morphism in $\Alg_{E_0}(\Alg_{E_2}(\cat))_{\C}$.

It is routine to check that if there are two 1-morphisms $(Q_i, \sigma_i): \A \rightarrow \C'$, $i = 1, 2$, in $\Alg_{E_0}(\Alg_{E_2}(\cat))_{\C}$, there exists a unique 2-morphism $\beta: (Q_1, \sigma_1) \Rightarrow (Q_2, \sigma_2)$ in $\Alg_{E_0}(\Alg_{E_2}(\cat))_{\C}$. 
\end{proof}

\section{Representation theory and Morita theory in $\cat$}

In this section, Sec.\,4.1 and Sec.\,4.2 study the modules over a multifusion category over $\E$ and bimodules in $\cat$.
Sec.\,4.3 and Sec.\,4.4 prove that two fusion categories over $\E$ are Morita equivalent in $\cat$ if and only if their $E_1$-centers are equivalent.
Sec.\,4.5 studies the modules over a braided fusion category over $\E$.

\subsection{Modules over a multifusion category over $\E$}   
Let $\C$ and $\D$ be multifusion categories over $\E$.
We use $z$ and $\hat{z}$ to denote the central structures of the central functors $T_{\C}: \E \to \C$ and $T_{\D}: \E \to \D$ respectively.

\begin{defn}
The 2-category $\LMod_{\C}(\cat)$ consists of the following data.
\begin{itemize}
\item A class of objects in $\LMod_{\C}(\cat)$. 
 An object $\M \in \LMod_{\C}(\cat)$ is an object $\M \in \cat$ equipped with a monoidal functor $\phi: \C \rightarrow \Fun_{\E}(\M, \M)$ over $\E$. 
 
 Equivalently, an object $\M \in \LMod_{\C}(\cat)$ is an object $\M$ both in $\cat$ and $\LMod_{\C}(\fcat)$ equipped with a monoidal natural isomorphism $u^{\C}_e: T_{\C}(e) \odot - \simeq e \odot -$ in $\Fun_{\E}(\M, \M)$ for each $e \in \E$, such that the functor $(c \odot -, s^{c \odot -})$ belongs to $\Fun_{\E}(\M, \M)$ for each $c \in \C$, and the diagram
\begin{equation}
\label{rem}
\begin{split}
\xymatrix{
(T_{\C}(e) \otimes c) \odot - \ar[r] \ar[d]_{z_{e, c}, 1} & T_{\C}(e) \odot (c \odot -) \ar[r]^(0.55){(u^{\C}_e)_{c \odot -}} & e \odot (c \odot -) \ar[d]^{s^{c \odot -}_{e,-}} \\
(c \otimes T_{\C}(e)) \odot - \ar[r] & c \odot (T_{\C}(e) \odot -) \ar[r]^(0.55){1, u^{\C}_e} & c \odot (e \odot -)
}
\end{split}
\end{equation}
commutes for $e \in \E, c \in \C, - \in \M$.
We use a pair $(\M, u^{\C})$ to denote an object $\M$ in $\LMod_{\C}(\cat).$
 
\item For objects $(\M, u^{\C}), (\N, \bar{u}^{\C})$ in $\LMod_{\C}(\cat)$, a 1-morphism $F: \M \rightarrow \N$ in $\LMod_{\C}(\cat)$ is both a left $\C$-module functor $(F, s^F): \M \rightarrow \N$ and a left $\E$-module functor $(F, t^F): \M \rightarrow \N$ such that the following diagram commutes for $e \in \E, m \in \M$:
\begin{equation}
\label{d2}
\begin{split}
 \xymatrix{
F(T_{\C}(e) \odot m) \ar[r]^(0.55){(u^{\C}_e)_m} \ar[d]_{s^F_{T_{\C}(e), m}} &  F(e \odot m)\ar[d]^{t^F_{e, m}} \\
T_{\C}(e) \odot F(m) \ar[r]_(0.55){(\bar{u}^{\C}_e)_{F(m)}} & e \odot F(m)
} 
\end{split}
\end{equation}

\item For 1-morphisms $F, G: \M \rightrightarrows \N$ in $\LMod_{\C}(\cat)$, 
a 2-morphism from $F$ to $G$ is a left $\C$-module natural transformation from $F$ to $G$. A left $\C$-module natural transformation is automatically a left $\E$-module natural transformation.
\end{itemize}
\end{defn}

In the above definition, we take $\phi(c) \coloneqq c \odot -$ for $c \in \C$.
A left $\D^{\rev}$-module $\M$ is automatically a right $\D$-module, with the right $\D$-action defined by $m \odot d \coloneqq d \odot m$ for $m \in \M, d \in \D$.

\begin{defn}
The 2-category $\RMod_{\D}(\cat)$ consists of the following data.
\begin{itemize} 
\item A class of objects in $\RMod_{\D}(\cat)$.
An object $\M \in \RMod_{\D}(\cat)$ is an object $\M \in \cat$ equipped with a monoidal functor $\phi: \D^{\rev} \rightarrow \Fun_{\E}(\M, \M)$ over $\E$.

Equivalently, an object $\M \in \RMod_{\D}(\cat)$ is an object $\M$ both in $\cat$ and $\RMod_{\D}(\fcat)$ equipped with a monoidal natural isomorphism $u^{\D}_e: - \odot T_{\D}(e) \simeq e \odot -$ in $\Fun_{\E}(\M, \M)$ for each $e \in \E$ such that the functor $(- \odot d, s^{- \odot d})$ belongs to $\Fun_{\E}(\M, \M)$ for each $d \in \D$, and the diagram
\begin{equation}
\label{rem2}
\begin{split}
 \xymatrix{
- \odot (d \otimes T_{\D}(e)) \ar[r] \ar[d]_{1, \hat{z}^{-1}_{e, d}} & (- \odot d) \odot T_{\D}(e) \ar[r]^(0.55){(u^{\D}_e)_{- \odot d}} & e \odot (- \odot d) \ar[d]^{s^{- \odot d}_{e,-}} \\
- \odot (T_{\D}(e) \otimes d) \ar[r] & (- \odot T_{\D}(e)) \odot d \ar[r]_(0.55){u^{\D}_e, 1} & (e \odot -) \odot d
} 
\end{split}
\end{equation}
commutes for $e \in \E, d \in \D, - \in \M$.
 We use a pair $(\M, u^{\D})$ to denote an object $\M$ in $\RMod_{\D}(\cat)$.

\item  
For objects $(\M, u^{\D})$, $(\N, \bar{u}^{\D})$ in $\RMod_{\D}(\cat)$, 
a 1-morphism $F: \M \rightarrow \N$ in $\RMod_{\D}(\cat)$ is both a right $\D$-module functor $(F, \tilde{s}^F): \M \rightarrow \N$ and a left $\E$-module functor $(F, t^F): \M \rightarrow \N$ such that the following diagram commutes for $e \in \E, m \in \M$:
\begin{equation}
\label{d3}
\begin{split}
\xymatrix{
F(m \odot T_{\D}(e)) \ar[r]^(0.55){(u^{\D}_e)_m} \ar[d]_{\tilde{s}^F_{m, T_{\D}(e)}} & F(e \odot m) \ar[d]^{t^F_{e, m}}  \\
F(m) \odot T_{\D}(e) \ar[r]_(0.55){(\bar{u}^{\D}_e)_{F(m)}} & e \odot F(m)
}
\end{split}
\end{equation}

\item For 1-morphisms $F, G: \M \rightrightarrows \N$ in $\RMod_{\D}(\cat)$, a 2-morphism from $F$ to $G$ is a right $\D$-module natural transformation from $F$ to $G$.
\end{itemize}
\end{defn} 

\begin{rem}
Let $(\M, u^{\D})$ belongs to $\RMod_{\D}(\cat)$.
We explain the monoidal natural isomorphism $u_e: - \odot T_{\D}(e) \simeq e \odot -$ in $\Fun_{\E}(\M, \M)$.
The monoidal structure on $F: \E \to \Fun_{\E}(\M, \M), e \mapsto F^e \coloneqq - \odot T_{\D}(e)$ is defined as
$J_{e_1, e_2}: F^{e_1 \otimes e_2} = - \odot T_{\D}(e_1 \otimes e_2) \xrightarrow{r_{e_1, e_2}} - \odot T_{\D}(e_2 \otimes e_1) \to (- \odot T_{\D}(e_2)) \odot T_{\D}(e_1) = F^{e_1} \circ F^{e_2}$,
for $e_1, e_2 \in \E$.
The monoidal structure on $T: \E \to \Fun_{\E}(\M, \M), e \mapsto T^e \coloneqq e \odot -$ is defined as 
$T^{e_1 \otimes e_2} = (e_1 \otimes e_2) \odot - \to e_1 \odot (e_2 \odot -) = T^{e_1} \circ T^{e_2}$, for $e_1, e_2 \in \E$.
For each $e \in \E$, $u_e: - \odot T_{\D}(e) \to e \odot -$ is an isomorphism in $\Fun_{\E}(\M, \M)$. That is, $u_e$ is a left $\E$-module natural isomorphism. 
The monoidal natural isomorphism $u_{e}: - \odot T_{\D}(e) \to e \odot -$ satisfies the diagram
\[ \xymatrix{
- \odot T_{\D}(e_1 \otimes e_2) \ar[r]^{r_{e_1, e_2}} \ar[d]_{u_{e_1 \otimes e_2}} & - \odot T_{\D}(e_2 \otimes e_1) \ar[r] & (- \odot T_{\D}(e_2)) \odot T_{\D}(e_1) \ar[d]^{u_{e_1} * u_{e_2}} \\
(e_1 \otimes e_2) \odot - \ar[rr]& & e_1 \odot (e_2 \odot -)
} \]
where $u_{e_1}*u_{e_2}$ is defined as
\[ \xymatrix{
F^{e_1} \circ F^{e_2} = (- \odot T_{\D}(e_2)) \odot T_{\D}(e_1) \ar[rd]|{u_{e_1}*u_{e_2}} \ar[r]^(0.6){u_{e_2},1} \ar[d]_{(u_{e_1})_{- \odot T_{\D}(e_2)}} & (e_2 \odot -) \odot T_{\D}(e_1) \ar[d]^{(u_{e_1})_{e_2 \odot -}} \\
e_1 \odot (- \odot T_{\D}(e_2)) \ar[r]_{1,u_{e_2}} & e_1 \odot (e_2 \odot -) = T^{e_1} \circ T^{e_2}
} \]
For any $d_1, d_2 \in \D$, the functors $(- \odot d_1, s^{- \odot d_1})$, $(- \odot d_2, s^{- \odot d_2})$ and $(- \odot (d_1 \otimes d_2), s^{- \odot (d_1 \otimes d_2)})$ belong to $\Fun_{\E}(\M, \M)$.
Consider the diagram:
\begin{equation}
\label{eMd-bimodule}
\begin{split}
 \xymatrix{
e \odot (m \odot (d_1 \otimes d_2)) \ar[rr]^{s^{- \odot (d_1 \otimes d_2)}_{e, m}} \ar[d]_{1, \lambda^{\M}_{m, d_1, d_2}} & & (e \odot m) \odot (d_1 \otimes d_2) \ar[d]^{\lambda^{\M}_{e \odot m, d_1, d_2}} \\
e \odot ((m \odot d_1) \odot d_2) \ar[r]_{s^{- \odot d_2}_{e, m \odot d_1}} & (e \odot (m \odot d_1)) \odot d_2 \ar[r]_{s^{- \odot d_1}_{e, m}, 1} & ((e \odot m) \odot d_1) \odot d_2
} 
\end{split}
\end{equation}
where $\lambda^{\M}$ is the module associativity constraint of $\M$ in $\RMod_{\D}(\fcat)$.
Since the diagrams (\ref{rem2}) and (\ref{center1}) commute and $m \odot -: \D \to \M$ is the functor for all $m \in \M$, the above diagram commutes. 
Then the natural isomorphism $- \odot (d_1 \otimes d_2) \Rightarrow (- \odot d_1) \odot d_2$ is the left $\E$-module natural isomorphism.

Let $(\M, u^{\C})$ belong to $\LMod_{\C}(\cat)$. 
For any $c_1, c_2 \in \C$, the functors $(c_1 \odot -, s^{c_1 \odot -})$, $(c_2 \odot -, s^{c_2 \odot -})$ and $((c_1 \otimes c_2) \odot -, s^{(c_1 \otimes c_2) \odot -})$ belong to $\Fun_{\E}(\M, \M)$. 
Since the diagrams (\ref{rem}) and (\ref{center1}) commutes and $- \odot m: \C \to \M$ is the functor for all $m \in \M$, the natural isomorphism $(c_1 \otimes c_2) \odot - \Rightarrow c_1 \odot (c_2 \odot -)$ is the left $\E$-module natural isomorphism.
\end{rem}

\begin{rem}
Assume that $(\M, u^{\D})$ belongs to $\RMod_{\D}(\cat)$. The right $\E$-module structure on $\M$ is defined as $m \bar{\odot} e \coloneqq m \odot T_{\D}(e)$, $\forall m \in \M, e \in \E$.
The module associativity constraint is defined as $\lambda^1_{m, e_1, e_2}: m \odot T_{\D}(e_1 \otimes e_2) \to m \odot (T_{\D}(e_1) \otimes T_{\D}(e_2)) \to (m \odot T_{\D}(e_1)) \odot T_{\D}(e_2)$, $\forall m \in \M, e_1, e_2 \in \E$.
Another right $\E$-module structure on $\M$ is defined as $m \odot e \coloneqq e \odot m$, $\forall e \in \E, m \in \M$. The module associativity constraint is defined as $\lambda^2_{m, e_1, e_2}: m \odot (e_1 \otimes e_2) = (e_1 \otimes e_2) \odot m \xrightarrow{r_{e_1, e_2}, 1} (e_2 \otimes e_1) \odot m \to e_2 \odot (e_1 \odot m) = (m \odot e_1) \odot e_2$, $\forall m \in \M, e_1, e_2 \in \E$.

Check that the identity functor $\id: \M \to \M$ equipped with the natural isomorphism
$s^{\id}_{m, e}: \id(m \bar{\odot} e) = m \odot T_{\D}(e) \xrightarrow{(u^{\D}_e)_m} e \odot m = \id(m) \odot e$
is a right $\E$-module functor by the monoidal natural isomorphism $u^{\D}_e: - \odot T_{\D}(e) \to e \odot -$.
\end{rem}

\begin{prop}
Let $(\M, u^{\C})$ belong to $\LMod_{\C}(\cat)$. The diagram
\begin{equation}
\label{diag-ee2}
\begin{split}
\xymatrix{
 \tilde{e} \odot (T_{\C}(e) \odot m) \ar[r]^(0.55){1, (u^{\C}_e)_m} \ar[d]_{s^{T_{\C}(e) \odot -}_{\tilde{e},m}} & \tilde{e} \odot (e \odot m) \ar[r] & (\tilde{e} \otimes e) \odot m  \ar[d]^{r_{\tilde{e},e}, 1} \\
 T_{\C}(e) \odot (\tilde{e} \odot m) \ar[r]_(0.55){(u^{\C}_e)_{\tilde{e} \odot m}} & e \odot (\tilde{e} \odot m) \ar[r] & (e \otimes \tilde{e}) \odot m
}
\end{split}
\end{equation}
commutes for $e, \tilde{e} \in \E$, $m \in \M$. 
Let $(\M, u^{\D})$ belong to $\RMod_{\D}(\cat)$. The diagram
\begin{equation}
\label{diag-e-e}
\begin{split}
 \xymatrix{
 (e \odot m) \odot T_{\D}(\tilde{e}) \ar[r]^(0.55){(u^{\D}_{\tilde{e}})_{e \odot m}} \ar[d]_{s^{- \odot T_{\D}(\tilde{e})}_{e,m}} & \tilde{e} \odot (e \odot m) \ar[r] & (\tilde{e} \otimes e) \odot m  \ar[d]^{r_{\tilde{e},e}, 1} \\
 e \odot (m \odot T_{\D}(\tilde{e})) \ar[r]_(0.55){1, (u^{\D}_{\tilde{e}})_m} &  e \odot (\tilde{e} \odot m) \ar[r] & (e \otimes \tilde{e}) \odot m
} 
\end{split}
\end{equation}
commutes for $e, \tilde{e} \in \E, m \in \M$.
Here the functors $(T_{\C}(e) \odot -, s^{T_{\C}(e) \odot -})$ and $(- \odot T_{\D}(\tilde{e}), s^{- \odot T_{\D}(\tilde{e})})$ belong to $\Fun_{\E}(\M, \M)$.
\end{prop}

\begin{proof}
Consider the diagram:
\[ \xymatrix@=4ex{
\tilde{e} \odot (T_{\C}(e) \odot m) \ar@/_4.3pc/[ddd]_{s^{T_{\C}(e) \odot -}_{\tilde{e}, m}} \ar[r]^{1, (u^{\C}_e)_m} & \tilde{e} \odot (e \odot m) \ar[r] & (\tilde{e} \otimes e) \odot m \ar@/^3pc/[ddd]^{r_{\tilde{e}, e}, 1} \\
 T_{\C}(\tilde{e}) \odot (T_{\C}(e) \odot m) \ar[u]_{(u^{\C}_{\tilde{e}})_{T_{\C}(e) \odot m}} & (T_{\C}(\tilde{e}) \otimes T_{\C}(e)) \odot m  \ar[d]_{z_{\tilde{e}, T_{\C}(e)}} \ar[l] & T_{\C}(\tilde{e} \otimes e) \odot m \ar[u]^{u^{\C}_{\tilde{e} \otimes e}} \ar[d]_{r_{\tilde{e},e}, 1}  \ar[l]  \\
T_{\C}(e) \odot (T_{\C}(\tilde{e}) \odot m) \ar[d]^{1, (u^{\C}_{\tilde{e}})_{m}} & (T_{\C}(e) \otimes T_{\C}(\tilde{e})) \odot m  \ar[l]& T_{\C}(e \otimes \tilde{e}) \odot m \ar[d]_{u^{\C}_{e \otimes \tilde{e}}}  \ar[l]  \\
T_{\C}(e) \odot (\tilde{e} \odot m) \ar[r]_{(u^{\C}_e)_{\tilde{e} \odot m}} & e \odot (\tilde{e} \odot m) \ar[r] & (e \otimes \tilde{e}) \odot m 
} \]
The top and bottom hexagon diagrams commute by the monoidal natural isomorphism $u^{\C}_e: T_{\C}(e) \odot - \simeq e \odot -$.
The leftmost hexagon commutes by the diagram (\ref{rem}).
The middle-right square commutes by the central functor $T_{\C}: \E \to \C$. 
The rightmost square commutes by the naturality of $u^{\C}_e$.
Then the outward diagram commutes.
One can check the diagram (\ref{diag-e-e}) commutes.
\end{proof}

For objects $\M, \N$ in $\LMod_{\C}(\cat)$ (or $\RMod_{\D}(\cat)$), we use $\Fun^{\E}_{\C}(\M, \N)$ (or $\Fun^{\E}_{|\D}(\M, \N)$) to denote the category of 1-morphisms $\M \to \N$, 2-morphisms in $\LMod_{\C}(\cat)$ (or $\RMod_{\D}(\cat)$).
\begin{expl}
\label{fun-C-E}
$\Fun^{\E}_{\C}(\M, \M)$ is a multifusion category by \cite[Cor.\,9.3.3]{Etingof}. Moreover, $\Fun^{\E}_{\C}(\M, \M)$ is a multifusion category over $\E$. 
A functor $\hat{T}: \E \rightarrow \Fun^{\E}_{\C}(\M, \M)$ is defined as $e \mapsto \hat{T}^e \coloneqq T_{\C}(e) \odot -$. The left $\C$-module structure on $\hat{T}^e$ is defined as $s_{c, m}: T_{\C}(e) \odot (c \odot m) \rightarrow (T_{\C}(e) \otimes c) \odot m \xrightarrow{z_{e, c},1} (c \otimes T_{\C}(e)) \odot m \rightarrow c \odot (T_{\C}(e) \odot m)$ for $c \in \C$, $m \in \M$.
The left $\E$-module structure on $\hat{T}^e$ is defined as $T_{\C}(e) \odot (\tilde{e} \odot m) \xrightarrow{1, (u^{\C}_{\tilde{e}})^{-1}_m} T_{\C}(e) \odot (T_{\C}(\tilde{e}) \odot m) \xrightarrow{s_{T_{\C}(\tilde{e}), m}} T_{\C}(\tilde{e}) \odot (T_{\C}(e) \odot m) \xrightarrow{(u^{\C}_{\tilde{e}})_{T_{\C}(e) \odot m}} \tilde{e} \odot (T_{\C}(e) \odot m)$ for $\tilde{e} \in \E, m \in \M$.
 Then $\hat{T}^e$ belongs to $\Fun^{\E}_{\C}(\M, \M)$.

The monoidal structure on $\hat{T}$ is induced by $T_{\C}(e_1 \otimes e_2) \odot - \simeq (T_{\C}(e_1) \otimes T_{\C}(e_2)) \odot - \simeq T_{\C}(e_1) \odot (T_{\C}(e_2) \odot -)$ for $e_1, e_2 \in \E$.
The central structure on $\hat{T}$ is a natural isomorphism $\sigma_{e, g}: \hat{T}^e \circ g(m) = T_{\C}(e) \odot g(m) \simeq g(T_{\C}(e) \odot m) = g \circ \hat{T}^e(m)$ for any $e \in \E$, $g \in \Fun^{\E}_{\C}(\M, \M)$, $m \in \M$.
The left (or right) $\E$-module structure on $\Fun^{\E}_{\C}(\M, \M)$ is defined as $(e \odot f)(-) \coloneqq T_{\C}(e) \odot f(-)$, (or $(f \odot e)(-) \coloneqq f(T_{\C}(e) \odot -)$ ), for $e \in \E$, $f \in \Fun^{\E}_{\C}(\M, \M)$ and $- \in \M$.
\end{expl}

\begin{prop}
\label{F-in-Cat_E}
Let $(\M, u)$ and $(\N, \bar{u})$ belong to $\LMod_{\C}(\cat)$. $f: \M \rightarrow \N$ is a 1-morphism in $\LMod_{\C}(\fcat)$. Then $f$ belongs to $\LMod_{\C}(\cat)$.
\end{prop}
\begin{proof}
Notice that for a 1-morphism $f: \M \to \N$ in $\LMod_{\C}(\cat)$, the left $\C$-action on $f$ is compatible with the left $\E$-action on $f$.
Assume $(f, s): \M \to \N$ is a left $\C$-module functor. 
The left $\E$-module structure on $f$ is given by $f(e \odot m) \xrightarrow{(u^{-1}_e)_m} f(T_{\C}(e) \odot m) \xrightarrow{s_{T_{\C}(e), m}} T_{\C}(e) \odot f(m) \xrightarrow{(\bar{u}_e)_{f(m)}} e \odot f(m)$.
\end{proof}

\begin{rem}
\label{p3}
The forgetful functor $\forget: \Fun_{\C}^{\E}(\M, \N) \to \Fun_{\C}(\M, \N)$, $(f, s, t) \mapsto (f, s)$ induces an equivalence in $\cat$, where $s$ and $t$ are the left $\C$-module structure and the left $\E$-module structure on $f$ respectively.
Notice that $t$ equals to the composition of $u^{-1}, s$ and $\bar{u}$.

Let $(\M, u)$ and $(\M, \id)$ belong to $\LMod_{\C}(\cat)$. Then the identity functor $\id_{\M}: (\M, u) \to (\M, \id)$ induces an equivalence in $\LMod_{\C}(\cat)$.
\end{rem}

\begin{expl}
\label{E-module-cat1}
Let $A$ be a separable algebra in $\C$. We use $\C_A$ to denote the category of right $A$-modules in $\C$. 
By \cite[Prop.\,2.7]{DMNO}, the category $\C_A$ is a finite semisimple abelian category.  $\C_A$ has a canonical left $\C$-module structure.
The left $\E$-module structure on $\C_A$ is defined as $e \odot x \coloneqq T_{\C}(e) \otimes x$ for any $e \in \E$, $x \in \C_A$. 
Then $(\C_A, \id)$ belongs to $\LMod_{\C}(\cat)$.
\end{expl}

We use ${{}_A\!{\C}_A}$ to denote the category of $A$-bimodules in $\C$.
By Prop.\,\ref{A-fun-bim-eq}, $\Fun_{\C}(\C_A, \C_A)$ is equivalent to $({{}_A\!{\C}_A})^{\rev}$ as multifusion categories over $\E$.

\begin{prop}
\label{p1}
Let $\M \in \LMod_{\C}(\cat)$. There is a separable algebra $A$ in $\C$ such that $\M \simeq \C_A$ in $\LMod_{\C}(\cat)$. 
\end{prop}
\begin{proof} 
By \cite[Thm.\,7.10.1]{Etingof}, there is an equivalence $\eta: \M \simeq \C_A$ in $\LMod_{\C}(\fcat)$ for some separable algebra $A$ in $\C$. 
By Prop.\,\ref{F-in-Cat_E}, $\eta$ is an equivalence in $\LMod_{\C}(\cat)$.
\end{proof}


\begin{defn}
An object $\M$ in $\LMod_{\C}(\cat)$ is \emph{faithful} if there exists $m \in \M$ such that $\unit_{\C}^i \odot m \nsimeq 0$ for every nonzero subobject $\unit_{\C}^i$ of the unit object $\unit_{\C}$.
\end{defn}

\begin{rem}
Notice that $\unit_{\E} \odot m \simeq T_{\C}(\unit_{\E}) \odot m = \unit_{\C} \odot m \nsimeq 0$.
If $\C$ is an indecomposable multifusion category over $\E$, any nonzero $\M$ in $\LMod_{\C}(\cat)$ is faithful.
\end{rem}

\begin{prop}
\label{double-centralizer}
Suppose $\M$ is a faithful object in $\LMod_{\C}(\cat)$. There is an equivalence $\C \simeq \Fun^{\E}_{\Fun_{\C}^{\E}(\M, \M)}(\M, \M)$ of multifusion categories over $\E$.
\end{prop}
\begin{proof}
By Prop.\,\ref{p1}, there is a separable algebra $A$ in $\C$ such that $\M \simeq \C_A$.
By \cite[Thm.\,7.12.11]{Etingof}, the category $\Fun_{\Fun_{\C}(\M, \M)}(\M, \M)$ is equivalent to the category of $A^R \otimes A$-bimodules in the category of $A$-bimodules. The latter category is equivalent to the category ${{}_{A^R \otimes A} \!{\C}_{A^R \otimes A}}$ of $A^R \otimes A$-bimodules. 
Then the functor $\Phi: \C \rightarrow {{}_{A^R \otimes A} \!{\C}_{A^R \otimes A}}$, $x \mapsto A^R \otimes x \otimes A$ is an equivalence by the faithfulness of $\M$. 
The monoidal structure on $\Phi$ is defined as 
\[ \Phi(x \otimes y) = A^R \otimes x \otimes y \otimes A  \simeq A^R \otimes x \otimes A \otimes_{A^R \otimes A} A^R \otimes y \otimes A = \Phi(x) \otimes_{A^R \otimes A} \Phi(y) \]
for $x, y \in \C$, where the equivalence is due to $A \otimes_{A^R \otimes A} A^R \simeq \unit_{\C}$.
Recall the central structure on the monoidal functor $I: \E \to {{}_{A^R \otimes A} \!{\C}_{A^R \otimes A}}$ in Expl.\,\ref{Ex-A-bimodule}.
The structure of monoidal functor over $\E$ on $\Phi$ is induced by $\Phi(T_{\C}(e)) = A^R \otimes T_{\C}(e) \otimes A \xrightarrow{z^{-1}_{e, A^R},1} T_{\C}(e) \otimes A^R \otimes A =I(e)$ for $e \in \E$.

By Rem.\,\ref{p3} and Prop.\,\ref{A-fun-bim-eq}, we have the equivalences $\Fun^{\E}_{\Fun_{\C}^{\E}(\M, \M)}(\M, \M) \simeq \Fun_{\Fun_{\C}(\M, \M)}(\M, \M) \simeq {{}_{A^R \otimes A} \!{\C}_{A^R \otimes A}} \simeq \C$ of multifusion categories over $\E$. 
\end{proof}

\subsection{Bimodules in $\cat$}
 Let $\C$ and $\D$ be multifusion categories over $\E$.
We use $z$ and $\hat{z}$ to denote the central structures of the central functors $T_{\C}: \E \to \C$ and $T_{\D}: \E \to \D$ respectively.
\begin{defn}
The 2-category $\BMod_{\C|\D}(\cat)$ consists of the following data.
\begin{itemize}
\item A class of objects in $\BMod_{\C|\D}(\cat)$. An object $\M \in \BMod_{\C|\D}(\cat)$ is an object $\M$ both in $\cat$ and $\BMod_{\C|\D}(\fcat)$ equipped with monoidal natural isomorphisms $u^{\C}_e: T_{\C}(e) \odot - \simeq e \odot -$ and $u^{\D}_e: - \odot T_{\D}(e) \simeq e \odot -$ in $\Fun_{\E}(\M, \M)$ for each $e \in \E$ such that 
the functor $(c \odot - \odot d, s^{c \odot - \odot d})$ belongs to $\Fun_{\E}(\M, \M)$ for each $c \in \C$, $d \in \D$,
and the diagrams
\begin{equation}
\label{dig-bimodule1}
\begin{split}
 \xymatrix{
(T_{\C}(e) \otimes c) \odot - \odot d  \ar[r] \ar[d]_{z_{e,c}, 1, 1} & T_{\C}(e) \odot (c \odot - \odot d) \ar[r]^(0.55){(u^{\C}_e)_{c \odot - \odot d}} & e \odot (c \odot - \odot d) \ar[d]^{s^{c \odot - \odot d}} \\
(c \otimes T_{\C}(e)) \odot - \odot d \ar[r] & c \odot (T_{\C}(e) \odot -) \odot d \ar[r]^(0.55){1, u^{\C}_e, 1} & c \odot (e \odot -) \odot d
} 
\end{split}
\end{equation}
\begin{equation}
\label{dig-bimodule2} 
\begin{split}
\xymatrix{
c \odot - \odot (d \otimes T_{\D}(e)) \ar[r] \ar[d]_{1,1, \hat{z}^{-1}_{e, d}} & (c \odot - \odot d) \odot T_{\D}(e) \ar[r]^(0.55){(u^{\D}_e)_{c \odot - \odot d}} & e \odot (c \odot - \odot d) \ar[d]^{s^{c \odot - \odot d}} \\
c \odot - \odot (T_{\D}(e) \otimes d) \ar[r] & c \odot (- \odot T_{\D}(e)) \odot d \ar[r]^(0.55){1, u^{\D}_e, 1} & c \odot (e \odot -) \odot d
} 
\end{split}
\end{equation}
commute for all $e \in \E, c \in \C, d \in \D$.
 We use a triple $(\M, u^{\C}, u^{\D})$ to denote an object $\M$ in $\BMod_{\C|\D}(\cat)$.

\item For objects $(\M, u^{\C}, u^{\D})$, $(\N, \bar{u}^{\C}, \bar{u}^{\D})$ in $\BMod_{\C|\D}(\cat)$, 
a 1-morphism $F: \M \rightarrow \N$ in $\BMod_{\C|\D}(\cat)$ is
a 1-morphism $F: \M \rightarrow \N$ both in $\cat$ and $\BMod_{\C|\D}(\fcat)$ such that the diagrams (\ref{d2}) and (\ref{d3}) commute.

\item For 1-morphisms $F, G: \M \rightrightarrows \N$ in $\BMod_{\C|\D}(\cat)$, 
a 2-morphism from $F$ to $G$ is a $\C$-$\D$ bimodule natural transformation from $F$ to $G$. 
\end{itemize}
For objects $\M, \N$ in $\BMod_{\C|\D}(\cat)$, we use $\Fun^{\E}_{\C|\D}(\M, \N)$ to denote the category of 1-morphisms $\M \to \N$, 2-morphisms in $\BMod_{\C|\D}(\cat)$.
\end{defn}

Let $(\M, u^{\C}, u^{\D}), (\N, \bar{u}^{\C}, \bar{u}^{\D})$ belong to $\BMod_{\C|\D}(\cat)$.
A monoidal natural isomorphism $v^{\M}$ is defined as $v^{\M}_e: T_{\C}(e) \odot - \xRightarrow{u^{\C}_e} e \odot - \xRightarrow{(u^{\D}_e)^{-1}} - \odot T_{\D}(e)$ for $e \in \E$, $- \in \M$.
Similarly, a monoidal natural isomorphism $v^{\N}$ is defined as $v^{\N}_e := (\bar{u}^{\D}_e)^{-1} \circ \bar{u}^{\C}_e$. 
A 1-morphism $F: \M \rightarrow \N$ in $\BMod_{\C|\D}(\cat)$ satisfies the following diagram for $e \in \E, m \in \M$:
\begin{equation}
\label{CD-bim-fun}
\begin{split}
 \xymatrix{
F(T_{\C}(e) \odot m) \ar[r]^{(v^{\M}_e)_m} \ar[d] & F(m \odot T_{\D}(e)) \ar[d] \\
T_{\C}(e) \odot F(m) \ar[r]_{(v^{\N}_e)_{F(m)}} & F(m) \odot T_{\D}(e)
} 
\end{split}
\end{equation}

\begin{rem}
\label{CD-bim-fun-rem}
Plugging $c=\unit_{\C}$ into the diagram (\ref{dig-bimodule1}) and $d= \unit_{\D}$ into the diagram (\ref{dig-bimodule2}), the diagrams
\begin{equation}
\label{diag-bi1} 
\begin{array}{c}
\xymatrix{
T_{\C}(e) \odot (m \odot d) \ar[r] \ar[d]_{(u^{\C}_e)_{m \odot d}} & (T_{\C}(e) \odot m) \odot d  \ar[d]^{(u^{\C}_e)_m, 1} \\
e \odot (m \odot d) \ar[r]_{s^{- \odot d}} & (e \odot m) \odot d
}
\end{array}
\qquad 
\begin{array}{c}
\xymatrix{
(c \odot m) \odot T_{\D}(e) \ar[r] \ar[d]_{(u^{\D}_e)_{c \odot m}} & c \odot (m \odot T_{\D}(e)) \ar[d]^{1, (u^{\D}_e)_m} \\
e \odot (c \odot m) \ar[r]_{s^{c \odot -}} & c \odot (e \odot m)
}
\end{array}
 \end{equation}
 commute for $m \in \M$.
 Since the diagrams (\ref{diag-bi1}) and (\ref{diag-ee2}) commute, the diagram
\begin{equation}
\label{diag-bi2}
\begin{split}
 \xymatrix{
(T_{\C}(e) \odot m) \odot T_{\D}(\tilde{e}) \ar[d] \ar[r]^(0.55){(u^{\C}_e)_m, 1} &  (e \odot m) \odot T_{\D}(\tilde{e}) \ar[r]^(0.55){(u^{\D}_{\tilde{e}})_{e \odot m}} & \tilde{e} \odot (e \odot m) \ar[r] & (\tilde{e} \otimes e)\odot m \ar[d]^{r_{e, \tilde{e}}, 1} \\
T_{\C}(e) \odot (m \odot T_{\D}(\tilde{e})) \ar[r]_(0.55){1, (u^{\D}_{\tilde{e}})_m} & T_{\C}(e) \odot (\tilde{e} \odot m) \ar[r]_(0.55){(u^{\C}_e)_{\tilde{e} \odot m}} & e \odot (\tilde{e} \odot m) \ar[r] & (e \otimes \tilde{e}) \odot m
} 
\end{split}
\end{equation}
commutes for $e, \tilde{e} \in \E$, $m \in \M$.
Since the diagrams (\ref{diag-bi1}), (\ref{rem}) and (\ref{rem2})
commute, the diagrams
\[ \xymatrix{
T_{\C}(e) \odot (m \odot d) \ar[r]^{(v^{\M}_{e})_{m \odot d}} \ar[d] & (m \odot d) \odot T_{\D}(e) \ar[r] & m \odot (d \otimes T_{\D}(e))\ar[d]^{1, \hat{z}^{-1}_{e, d}} \\
(T_{\C}(e) \odot m) \odot d \ar[r]_{(v^{\M}_e)_m, 1} &(m \odot T_{\D}(e)) \odot d \ar[r] & m \odot (T_{\D}(e) \otimes d)
} \]
\[ \xymatrix{
(T_{\C}(e) \otimes c) \odot m \ar[r] \ar[d]_{z_{e, c}, 1} & T_{\C}(e) \odot (c \odot m)  \ar[r]^{(v^{\M}_e)_{c \odot m}} & (c \odot m) \odot T_{\D}(e) \ar[d] \\
(c \otimes T_{\C}(e)) \odot m \ar[r] &c \odot (T_{\C}(e) \odot m) \ar[r]_{1, (v^{\M}_e)_m} & c \odot (m \odot T_{\D}(e))
} \]
commute for $e \in \E, d \in \D, c \in \C, m \in \M$.
\end{rem}

\begin{prop}
Let $\A, \B$ be multifusion categories over $\E$. There is an equivalence of 2-categories
\[  \LMod_{\A \boxtimes_{\E} \B^{\rev}}(\cat) \simeq \BMod_{\A|\B}(\cat) \]
\end{prop}
\begin{proof}
An object $\M \in \LMod_{\A \boe \B^{\rev}}(\cat)$ is an object $\M \in \cat$ equipped with a monoidal functor $\phi: \A \boxtimes_{\E} \B^{\rev} \rightarrow \Fun_{\E}(\M, \M)$ over $\E$.
 Given an object $\M$ in $\LMod_{\A \boe \B^{\rev}}(\cat)$, we want to define an object $(\M, u^{\A}, u^{\B})$ in $\BMod_{\A|\B}(\cat)$.
 The left $\A$-action on $\M$ is defined as $a \odot m \coloneqq \phi^{a \boe \unit_{\B}}(m)$ for $a \in \A$, $m \in \M$, and the unit $\unit_{\B} \in \B^{\rev}$. And the right $\B$-action on $\M$ is defined as $m \odot b \coloneqq \phi^{\unit_{\A} \boe b}(m)$ for $b \in \B$, $m \in \M$, and the unit $\unit_{\A} \in \A$.  
 By Expl.\,\ref{ex3}, we have $T_{\A \boe \B^{\rev}}(e) = T_{\A}(e) \boe \unit_{\B}$ and $T_{\A \boe \B^{\rev}}(e) \simeq \unit_{\A} \boe T_{\B}(e)$. 
 Recall the central structure on $T: \E \to \Fun_{\E}(\M, \M)$ in Expl.\ref{ex2}. 
 The structure of monoidal functor over $\E$ on $\phi$ gives the monoidal natural isomorphisms $u^{\A}$ and $u^{\B}$ and the commutativity of diagrams (\ref{dig-bimodule1}) and (\ref{dig-bimodule2}).
 
 Given objects $\M, \N$ and a 1-morphism $f: \M \rightarrow \N$ in $\LMod_{\A \boe \B^{\rev}}(\cat)$, $f$ satisfy the diagrams (\ref{d2}) and (\ref{d3}).
For two 1-morphisms $f, g: \M \rightrightarrows \N$ in $\LMod_{\A \boe \B^{\rev}}(\cat)$, a 2-morphism $\alpha: f \Rightarrow g$ in $\LMod_{\A \boe \B^{\rev}}(\cat)$ is a left $\A \boe \B^{\rev}$-module natural transformation. 
If $\B^{\rev} = \E$, $\alpha$ is a left $\A$-module natural transformation.
If $\A = \E$, $\alpha$ is a right $\B$-module natural transformation. 
 
Conversely, given an object $(\M, u^{\A}, u^{\B})$ in $\BMod_{\A|\B}(\cat)$, we want to define a monoidal functor $\phi: \A \boe \B^{\rev} \rightarrow \Fun_{\E}(\M, \M)$ over $\E$. 
For $a \boe b \in \A \boe \B^{\rev}$, we define $\phi^{a \boe b} \coloneqq (a \boe b) \odot - \coloneqq a \odot - \odot b$ for $- \in \M$. 
For $a_1 \boe b_1$, $a_2 \boe b_2 \in \A \boe \B^{\rev}$, the monoidal structure on $\phi$ is defined as
$\phi^{(a_1 \boe b_1) \otimes (a_2 \boe b_2)} = \phi^{(a_1 \otimes a_2) \boe (b_1 \otimes^{\rev} b_2)} = (a_1 \otimes a_2) \odot - \odot (b_2 \otimes b_1) \simeq a_1 \odot (a_2 \odot - \odot b_2) \odot b_1 = \phi^{a_1 \boe b_1} \circ \phi^{a_2 \boe b_2}$. 
The structure of monoidal functor over $\E$ on $\phi$ is defined as $\phi^{T_{\A \boe \B^{\rev}}(e)} = \phi^{T_{\A}(e) \boe \unit_{\B}} = T_{\A}(e) \odot - \odot \unit_{\B} \xrightarrow{u^{\A}, 1} e \odot - \odot \unit_{\B}  \simeq e \odot - = T^e$ for $e \in \E$. 

Given an object $(\N, \bar{u}^{\A}, \bar{u}^{\B})$ and a 1-morphism $f: \M \rightarrow \N$ in $\BMod_{\A|\B}(\cat)$, we want to define a 1-morphism $f$ in $\LMod_{\A \boe \B^{\rev}}(\cat)$.
The left $\A \boe \B^{\rev}$-module structure on $f$ is defined as $f((a \boe b) \odot m) = f(a \odot m \odot b) \xrightarrow{s^f} a \odot f(m \odot b) \xrightarrow{1, t^f} a \odot f(m) \odot b = (a \boe b) \odot f(m)$ for $a \boe b \in \A \boe \B^{\rev}$, $m \in \M$, where $s^f$ and $t^f$ are the left $\A$-module structure and the right $\B$-module structure on $f$ respectively. It is routine to check that $f$ satisfy the diagram (\ref{d2}).

For 1-morphisms $f, g: \M \rightrightarrows \N$ in $\BMod_{\A|\B}(\cat)$, a 2-morphism $\alpha: f \Rightarrow g$ in $\BMod_{\A|\B}(\cat)$ is an $\A$-$\B$ bimodule natural transformation. It is routine to check that $\alpha: f \Rightarrow g$ is a left $\A \boe \B^{\rev}$-module natural transformation. 
\end{proof}

\begin{expl}
Let $\C$ be a multifusion category over $\E$. 
The left $\E$-module structure on $\C$ is defined as $e \odot c \coloneqq T_{\C}(e) \otimes c$ for $e \in \E, c \in \C$. 
For $c \in \C$, the functor $(c \otimes -, s^{c \otimes -}): \C \to \C$ belongs to $\Fun_{\E}(\C, \C)$, where the natural isomorphism
$s^{c \otimes -}_{e,-}: c \otimes (e \odot -) = c \otimes T_{\C}(e) \otimes - \xrightarrow{z^{-1}_{e,c},1} T_{\C}(e) \otimes c \otimes - = e \odot (c \otimes -) $.
Then $(\C, \id_e: T_{\C}(e) \otimes - = e \odot -)$ belongs to $\LMod_{\C}(\cat)$.

For $c \in \C$, the functor $(- \otimes c, s^{- \otimes c}): \C \to \C$ belongs to $\Fun_{\E}(\C, \C)$, where the natural isomorphism 
$s^{- \otimes c}_{e,-}: (e \odot -) \otimes c = (T_{\C}(e) \otimes -) \otimes c \xrightarrow{\simeq} T_{\C}(e) \otimes (- \otimes c) = e \odot (- \otimes c)$.
The category $\C$ equipped with the monoidal natural isomorphism $u_e: - \otimes T_{\C}(e) \xrightarrow{z_{e,-}^{-1}} T_{\C}(e) \otimes - = e \odot -$ belongs to $\RMod_{\C}(\cat)$.

For $c, \tilde{c} \in \C$, the functor $c \otimes - \otimes \tilde{c}: \C \to \C$ equipped with the natural isomorphism
\[ s^{c \otimes - \otimes \tilde{c}}_{e,-}: c \otimes (e \odot -) \otimes \tilde{c} = c \otimes T_{\C}(e) \otimes - \otimes \tilde{c} \xrightarrow{z_{e, c}^{-1},1,1} T_{\C}(e) \otimes c \otimes - \otimes \tilde{c} = e \odot (c \otimes - \otimes \tilde{c})  \]
beongs to $\Fun_{\E}(\M, \M)$.
Then $(\C, \id_e, u_e)$ belongs to $\BMod_{\C|\C}(\cat)$. 
\end{expl}

\begin{thm}
\label{C-C-bi-functor}
Let $\C$ be a multifusion category over $\E$ such that $\E \to Z(\C)$ is fully faithful. There is an equivalence of multifusion categories over $\E$:
\[ \Fun^{\E}_{\C | \C}(\C, \C) \simeq Z(\C, \E)\]
\end{thm}

\begin{proof}
Let us recall the proof of a monoidal equivalence $\Fun_{\C \boxtimes \C^{\rev}}(\C, \C) \simeq Z(\C)$ in \cite[Prop.\,7.13.8]{Etingof}.
Let $F$ belong to $\Fun_{\C \boxtimes \C^{\rev}}(\C, \C)$. Since $F$ is a right $\C$-module functor, we have $F = d \otimes -$ for some $d \in \C$. Since $F$ is a left $\C$-module functor, we have a natural isomorphism
\[ d \otimes (x \otimes y) = F(x \otimes y) \xrightarrow{s_{x, y}} x \otimes F(y) = x \otimes (d \otimes y) \quad x, y \in \C \]
Taking $y = \unit_{\C}$, we obtain a natural isomorphism $\gamma_d = s_{-, \unit_{\C}}: d \otimes - \xrightarrow{\simeq} - \otimes d$. The compatibility conditions of $\gamma_d$ correspond to the axioms of module functors. Then $(d, \gamma_d)$ belongs to $Z(\C)$.
 And the composition of $\C$-bimodule functors of $\C$ corresponds to the tensor product of objects of $Z(\C)$.

Moreover, $F$ belongs to $\Fun^{\E}_{\C | \C}(\C, \C)$. 
Taking $m =\unit_{\C}, F= d \otimes -$ in the diagram (\ref{CD-bim-fun}), the following square commutes:
\[
\xymatrixcolsep{0.5pc}
\xymatrixrowsep{0.8pc}
\xymatrix{
d \otimes (T_{\C}(e) \otimes \unit_{\C}) \ar[d]_{\gamma_d,1} \ar[rr]^{1, z_{e, \unit_{\C}}} & & d \otimes (\unit_{\C} \otimes T_{\C}(e)) \ar[d]^{} \\
 T_{\C}(e) \otimes (d \otimes \unit_{\C}) \ar[rr]^{z_{e, d \otimes \unit_{\C}}} \ar[rd]_{z_{e,d},1} && (d \otimes \unit_{\C}) \otimes T_{\C}(e) \\
 & d \otimes T_{\C}(e) \otimes \unit_{\C} \ar[ru]_{1,z_{e, \unit_{\C}}} &
}\]
The triangle commutes by the diagram (\ref{center1}).
Then we obtain $z_{e,d} \circ \gamma_d = \id_{d \otimes T_{\C}(e)}$, i.e. $(d, \gamma_{d}) \in Z(\C, \E)$.
It is routine to check that the functor $\Fun^{\E}_{\C | \C}(\C, \C) \rightarrow Z(\C, \E)$ is a monoidal functor over $\E$.
\end{proof}

\begin{expl}
\label{expl-ABC-E}
Let $A, B$ be separable algebras in a multifusion category $\C$ over $\E$.
We use ${{}_A\!{\C}_B}$ to denote the category of $A$-$B$ bimodules in $\C$.
 The left $\E$-module structure on ${{}_A\!{\C}_B}$ is defined as $e \odot x \coloneqq T_{\C}(e) \otimes x$ for $e \in \E$, $x \in {{}_A\!{\C}_B}$.
We use $q_x$ and $p_x$ to denote the left $A$-action and right $B$-action on $x$ respectively. 
The right $B$-action on $T_{\C}(e) \otimes x$ is induced by $T_{\C}(e) \otimes x \otimes B \xrightarrow{1, p_x} T_{\C}(e) \otimes x$. The left $A$-action on $T_{\C}(e) \otimes x$ is induced by $A \otimes T_{\C}(e) \otimes x \xrightarrow{z^{-1}_{e, A}, 1} T_{\C}(e) \otimes A \otimes x \xrightarrow{1, q_x} T_{\C}(e) \otimes x$. 
The module associativity constraint is given by $\lambda_{e_1, e_2, x}: (e_1 \otimes e_2) \odot x = T_{\C}(e_1 \otimes e_2) \otimes x \to T_{\C}(e_1) \otimes T_{\C}(e_2) \otimes x = e_1 \odot (e_2 \odot x)$, for $e_1, e_2 \in \E, x \in {}_A\!\C_B$.
 The unit isomorphism is given by $l_x: \unit_{\E} \odot x = T_{\C}(\unit_{\E}) \otimes x = \unit_{\C} \otimes x \to x$.
 Check that $\lambda_{e_1, e_2, x}$ and $l_x$ belong to ${}_A\!\C_B$.

The right $\E$-action on ${}_A\!\C_B$ is defined as $x \odot e \coloneqq x \otimes T_{\C}(e)$, $e \in \E, x \in {}_A\!\C_B$.
The left $A$-action on $x \otimes T_{\C}(e)$ is defined as $A \otimes x \otimes T_{\C}(e) \xrightarrow{q_x, 1} x \otimes T_{\C}(e)$.
The right $B$-action on $x \otimes T_{\C}(e)$ is defined as $x \otimes T_{\C}(e) \otimes B \xrightarrow{1, z_{e, B}} x \otimes B \otimes T_{\C}(e) \xrightarrow{p_x, 1} x \otimes T_{\C}(e)$.
The module associativity constraint is defined as 
$\lambda_{x, e_1, e_2}: x \odot (e_1 \otimes e_2) = x \otimes T_{\C}(e_1 \otimes e_2) \to x \otimes T_{\C}(e_1) \otimes T_{\C}(e_2) = (x \odot e_1) \odot e_2$, for $x \in {}_A\!\C_B$, $e_1, e_2 \in \E$.
The unit isomorphism is defined as $r_x: x \odot \unit_{\E} = x \otimes T_{\C}(\unit_{\E}) = x \otimes \unit_{\C} \to x$.
Check that $\lambda_{x, e_1, e_2}$ and $r_x$ belong to ${}_A\!\C_B$.
Check that ${}_A\!\C_B$ equipped with the monoidal natural isomorphism $v_e: T_{\C}(e) \otimes x \xrightarrow{z_{e, x}} x \otimes T_{\C}(e)$ belongs to $\BMod_{\E|\E}(\cat)$.

Also one can check that ${}_A\!\C$ belongs to $\BMod_{\E|\C}(\cat)$ and $\C_B$ belongs to $\BMod_{\C|\E}(\cat)$.
\end{expl}

\begin{expl}
Let $\M$ belongs to $\RMod_{\D}(\cat)$. 
Then $\M$ belongs to $\BMod_{\E|\D}(\cat)$.
The $\E$-$\D$ bimodule structure on $\M$ is defined as $(e \odot m) \odot d \xrightarrow{(s^{- \odot d}_{e, m})^{-1}} e \odot (m \odot d)$ for any $e \in \E, m \in \M$.
Since $(- \odot d, s^{- \odot d})$ belongs to $\Fun_{\E}(\M, \M)$ and the diagram (\ref{eMd-bimodule}) commutes, $M$ is an $\E$-$\D$ bimodule category.

The functor $e \odot - \odot d: \M \to \M$ equipped with the natural isomorphism
$s^{e \odot - \odot d}_{\tilde{e},-}: \tilde{e} \odot ((e \odot -) \odot d) \xrightarrow{s^{- \odot d}_{\tilde{e}, e\odot -}} (\tilde{e} \odot (e \odot -)) \odot d \xrightarrow{s^{e \odot -}_{\tilde{e},-}, 1} (e \odot (\tilde{e} \odot -)) \odot d$
is a left $\E$-module functor, where $s^{e \odot -}_{\tilde{e},-}: \tilde{e} \odot (e \odot - ) \simeq (\tilde{e} \otimes e) \odot -  \xrightarrow{r_{\tilde{e}, e}} (e \otimes \tilde{e}) \odot - \simeq e \odot (\tilde{e} \odot -)$ for $\tilde{e} \in \E$.
The object $\M$ both in $\cat$ and $\BMod_{\E|\D}(\fcat)$ equipped with the monoidal natural isomorphisms $u^{\E}_e = \id: e \odot - = e \odot -$ and $u^{\D}_e: - \odot T_{\D}(e) \simeq e \odot -$ belongs to $\BMod_{\E|\D}(\cat)$.
The monoidal natural isomorphism $u^{\D}_e$ satisfies the diagram (\ref{dig-bimodule2}) by the diagrams (\ref{rem2}) and (\ref{diag-e-e}).
\end{expl}

\begin{expl}
\label{11}
Let $\C, \D$ be multifusion categories over $\E$ and $(\M, u^{\C}, u^{\D}) \in \BMod_{\C|\D}(\cat)$. The $\D$-$\C$ bimodule structure on the category $\M^{L|\op|L}$ is defined as $d \odot^L m \odot^L c \coloneqq c^L \odot m \odot d^L$ for $d \in \D, c \in \C, m \in \M$. Then $(\M^{L|\op|L}, \tilde{u}^{\D}, \tilde{u}^{\C})$ belongs to $\BMod_{\D|\C}(\cat)$.
The left $\E$-module structure on $\M^{L|\op|L}$ is defined as $e \odot^L m \coloneqq e^L \odot m$ for $e \in \E$, $m \in \M$.
The monoidal natural isomorphism $\tilde{u}^{\D}$ is defined as $T_{\D}(e) \odot^L m = m \odot T_{\D}(e)^L \simeq m \odot T_{\D}(e^L) \xrightarrow{u^{\D}_{e^L}} e^L \odot m$.
The monoidal natural isomorphism $\tilde{u}^{\C}$ is defined as $m \odot^L T_{\C}(e) = T_{\C}(e)^L \odot m \simeq T_{\C}(e^L) \odot m \xrightarrow{u^{\C}_{e^L}} e^L \odot m$.
\end{expl}

\begin{expl}
\label{33}
Let $\C, \D, \CP$ be multifusion categories over $\E$, and $(\M, u^{\C}, u^{\D}) \in \BMod_{\C|\D}(\cat)$, $(\N, \bar{u}^{\C}, \bar{u}^{\CP}) \in \BMod_{\C|\CP}(\cat)$. Then $(\Fun^{\E}_{\C}(\M, \N), \tilde{u}^{\D}, \tilde{u}^{\CP})$ belongs to $\BMod_{\D|\CP}(\cat)$.
The left $\E$-module structure on $\Fun^{\E}_{\C}(\M, \N)$ is defined as $(e \odot f)(-) \coloneqq T_{\C}(e) \odot f(-)$, for $e \in \E, f \in \Fun_{\C}^{\E}(\M, \N)$.
The $\D$-$\CP$ bimodule structure on $\Fun^{\E}_{\C}(\M, \N)$ is defined as $(d \odot f \odot p)(-) \coloneqq f(- \odot d) \odot p$ for any $d \in \D, p \in \CP$. 
Let $v^{\M}_e \coloneqq (u^{\D}_e)^{-1} \circ u^{\C}_e$ and $v^{\N}_e \coloneqq (\bar{u}^{\CP}_e)^{-1} \circ \bar{u}^{\C}_e$.
The monoidal natural isomorphism $\tilde{u}^{\D}$ is defined as $(T_{\D}(e) \odot f)(-) = f(- \odot T_{\D}(e)) \xrightarrow{(v^{\M}_e)^{-1}} f(T_{\C}(e) \odot -) \xrightarrow{s^f} T_{\C}(e) \odot f(-) = (e \odot f)(-)$.
The monoidal natural isomorphism $\tilde{u}^{\CP}$ is defined as $(f \odot T_{\CP}(e))(-) = f(-) \odot T_{\CP}(e) \xrightarrow{(v^{\N}_e)^{-1}} T_{\C}(e) \odot f(-) = (e \odot f)(-)$.
\end{expl}

\subsection{Invertible bimodules in $\cat$}

\begin{defn}
Let $\C$ be a multifusion category over $\E$, and $(\M, u^{\M}) \in \RMod_{\C}(\cat)$, $(\N, u^{\N}) \in \LMod_{\C}(\cat)$ and $\D \in \cat$.
A \emph{balanced $\C$-module functor $F: \M \times \N \to \D$ in $\cat$} consists of the following data.
\begin{itemize}
\item
 $F: \M \times \N \to \D$ is an $\E$-bilinear bifunctor. That is, for each $n \in \N$, $(F(-,n), s^{F1}): \M \to \D$ is a left $\E$-module functor, where 
 \[ s^{F1}_{e,m}: F(e \odot m, n) \simeq e \odot F(m, n), \quad \forall e \in \E, m \in \M \]
is a natural isomorphism. For each $g: n \to n'$ in $\N$, $F(-, g): F(-, n) \Rightarrow F(-, n')$ is a left $\E$-module natural transformation.
And for each $m \in \M$, $(F(m, -), s^{F2}): \N \to \D$ is a left $\E$-module functor, where 
\[ s^{F2}_{e, n}: F(m, e \odot n) \simeq e \odot F(m, n), \quad \forall e \in \E, n \in \N \]
is a natural isomorphism. For each $f: c \to c'$ in $\C$, $F(f, -): F(c, -) \Rightarrow F(c', -)$ is a left $\E$-module natural transformation.
\item $F:\M \times \N \to \D$ is a balanced $\E$-module functor (recall Def.\,\ref{be-module functor}), 
where the balanced $\E$-module structure on $F$ is defined as
\[ \hat{b}_{m, e, n}: F(m \odot e, n) = F(e \odot m, n) \xrightarrow{s^{F1}_{e, m}} e \odot F(m, n) \xrightarrow{(s^{F2}_{e,n})^{-1}} F(m, e \odot n). \]
  \item 
  $F: \M \times \N \to \D$ is a balanced $\C$-module functor (recall Def.\,\ref{be-module functor}), where $b_{m, c, n}: F(m \odot c, n) \simeq F(m, c \odot n)$, $\forall m \in \M, c \in \C, n \in \N$, is the balanced $\C$-module structure on $F$.  
  And $b_{m, c, n}$ is a left $\E$-module natural isomorphism. That is, the following diagram commutes
\begin{equation}
\label{diag-em}
\begin{split}
\xymatrix@=4ex{
F(e \odot (m \odot c), n) \ar[r]^{s^{F1}_{e, m \odot c}} \ar[d]_{s^{- \odot c}_{e, m}, 1} & e \odot F(m \odot c, n) \ar[dd]^{1, b_{m, c, n}} \\
F((e \odot m) \odot c, n) \ar[d]_{b_{e \odot m, c, n}} & \\
F(e \odot m, c \odot n) \ar[r]_{s^{F1}_{e, m}} & e \odot F(m, c \odot n)
}
\end{split}
\end{equation}
where the functor $(- \odot c, s^{- \odot c}) \in \Fun_{\E}(\M, \M)$, $\forall c \in \C$.
\end{itemize}
such that the followng diagram commutes
\begin{equation}
\label{diag-b-hatb}
\begin{split}
 \xymatrix{
F(m \odot T_{\C}(e), n) \ar[rr]^{b_{m, T_{\C}(e)}, n} \ar[d]_{(u^{\M}_e)_m, 1} & & F(m, T_{\C}(e) \odot n) \ar[d]^{1, (u^{\N}_e)_n} \\
F(e \odot m, n ) \ar[r]_{s^{F1}_{e, m}} \ar@/^1pc/[rr]^{\hat{b}_{m, e, n}} & e \odot F(m, n) \ar[r]_{(s^{F2}_{e, n})^{-1}} & F(m, e \odot n) 
}
\end{split}
\end{equation}
We use $\Fun^{\bal|\E}_{\C}(\M, \N; \D)$ to denote the category of balanced $\C$-module functors in $\cat$, and natural transformations both in $\Fun^{\bal}_{\C}(\M, \N; \D)$ and $\cat$.

The \emph{tensor product of $\M$ and $\N$ over $\C$} is an object $\M \boxtimes_{\C} \N$ in $\cat$, together with a balanced $\C$-module functor $\boxtimes_{\C}: \M \times \N \rightarrow \M \boxtimes_{\C} \N$ in $\cat$, such that, for every object $\D$ in $\cat$, composition with $\boxtimes_{\C}$ induces an equivalence $\Fun_{\E}(\M \boxtimes_{\C} \N, \D) \simeq \Fun^{\bal | \E}_{\C}(\M, \N; \D)$.
\end{defn}

\begin{prop}
  For $e_1, e_2 \in \E$, $m \in \M, n \in \N$,
the following diagram commutes
\[ \xymatrix{
F(e_1 \odot m, e_2 \odot n) \ar[r]^{s^{F1}_{e_1, m}} \ar[d]_{s^{F2}_{e_2, n}} & e_1 \odot F(m, e_2 \odot n)   \ar[r]^{1, s^{F2}_{e_2, n}} & e_1 \odot e_2 \odot F(m,n)  \ar[ld]^(0.45){r_{e_1, e_2}, 1} \\
e_2 \odot F(e_1 \odot m, n) \ar[r]_{1, s^{F1}_{e_1, m}} & e_2 \odot e_1 \odot F(m, n)  &
}  \]
\end{prop}

\begin{proof}
Since $F: \M \times \N \to \D$ is a balanced $\E$-module functor, the following outward diagram commutes.
\[ \xymatrix@=4ex{
  F((e_1 \otimes e_2) \odot m, n) \ar[r]^{s^{F1}_{e_1 \otimes e_2, m}} \ar[d]_{r_{e_1, e_2}, 1} & (e_1 \otimes e_2) \odot F(m, n) \ar[d]^{r_{e_1, e_2}} \ar[r]^{(s^{F2}_{e_1 \otimes e_2, n})^{-1}} & F(m, (e_1 \otimes e_2) \odot n) \\
 F(e_2 \odot e_1 \odot m, n) \ar[r]^{s^{F1}_{e_2 \otimes e_1, m}} \ar[d]_{s^{F1}_{e_2, e_1 \odot m}} & e_2 \odot e_1 \odot F(m, n) &   \\
 e_2 \odot F(e_1 \odot m, n) \ar[r]_{(s^{F2}_{e_2, n})^{-1}} \ar[ru]_(0.65){1, s^{F1}_{e_1, m}} &   F(e_1 \odot m, e_2 \odot n) \ar[r]_{s^{F1}_{e_1, m}} & e_1 \odot F(m, e_2 \odot n) \ar[uu]_{(s^{F2}_{e_1, e_2 \odot n})^{-1}}  \ar@/_1pc/[luu]_(0.4){1, s^{F2}_{e_2, n}} 
} \]
The two triangles commute since $(F(-,n), s^{F1}): \M \to \D$ and $(F(m,-), s^{F2}): \N \to \D$ are left $\E$-module functors. The square commutes by the naturality of $s^{F1}$. Then the pentagon commutes.
\end{proof}

\begin{prop}
\label{app-rem-prop}
For $e \in \E, m \in \M, c \in \C, n \in \N$, the diagram
\begin{equation}
\label{diag-emne}
\begin{split}
 \xymatrix{
F((e \odot m) \odot c, n) \ar[r]^{b_{e \odot m, c, n}} \ar[d]_{(s^{- \odot c}_{e,m})^{-1}} & F(e \odot m, c \odot n) \ar[r]^(0.48){\hat{b}_{m, e, c \odot n}} & F(m, e \odot (c \odot n)) \ar[d]^{s^{c \odot -}_{e, n}} \\
F(e \odot (m \odot c), n) \ar[r]_{\hat{b}_{m \odot c, e, n}} & F(m \odot c, e \odot n) \ar[r]_(0.48){b_{m, c, e \odot n}} & F(m, c \odot (e \odot n))
} 
\end{split}
\end{equation}
commutes, where the functors $(- \odot c, s^{- \odot c}) \in \Fun_{\E}(\M, \M)$ and $(c \odot -, s^{c \odot -}) \in \Fun_{\E}(\N, \N)$, $\forall c \in \C$.
\end{prop}

\begin{proof}
Consider the following diagram:
\begin{equation}
\label{app-rem}
\begin{split}
 \xymatrix@=4ex{
F(m \odot (T_{\C}(e) \otimes c) , n) \ar[r]^{z_{e,c}, 1} \ar[d]^{(u^{\M}_e)_m, 1} \ar@/_4.5pc/[dddd]_{b_{m, T_{\C}(e) \otimes c, n}} & F(m \odot (c \otimes T_{\C}(e)), n)  \ar@/^4.5pc/[dddd]^{b_{m, c \otimes T_{\C}(e), n}}  \ar[d]_{(u^{\M}_e)_{m \odot c}, 1} \\ 
F((e \odot m) \odot c, n) \ar[d]^{b_{e \odot m, c, n}} \ar[r]_{(s^{- \odot c}_{e,m})^{-1}} & 
 F(e \odot (m \odot c), n) \ar[d]_{\hat{b}_{m \odot c, e, n}}  \\
 F(e \odot m, c \odot n) \ar[d]^{\hat{b}_{m, e, c \odot n}} & F(m \odot c, e \odot n) \ar[d]_{b_{m, c, e \odot n}}  \\
 F(m, e \odot (c \odot n)) \ar[r]^{s^{c \odot -}_{e, n}}  & F(m, c \odot (e \odot n))  \\
F(m, (T_{\C}(e) \otimes c) \odot n)) \ar[r]_{1, z_{e,c}} \ar[u]_{1, (u^{\N}_e)_{c \odot n}} & F(m, (c \otimes T_{\C}(e)) \odot n)) \ar[u]^{1, (u^{\N}_e)_{n}}
} 
\end{split}
\end{equation}
Here $z$ is the central structure of the central functor $T_{\C}: \E \to \C$.
The middle-top and middle-down squares commute by the diagrams (\ref{rem}) and (\ref{rem2}).
 The leftmost diagram commutes by the diagram
 \[ 
 \xymatrix@=4ex{
 F(m \odot (T_{\C}(e) \otimes c), n) \ar[rr]^{b_{m, T_{\C}(e) \otimes c, n}} \ar[d] & & F(m, (T_{\C}(e) \otimes c) \odot n) \ar[d] \\
F((m \odot T_{\C}(e)) \odot c, n) \ar[r]^{b_{m \odot T_{\C}(e), c, n}} \ar[d]_{(u^{\M}_e)_m, 1} & F(m \odot T_{\C}(e), c \odot n) \ar[r]^{b_{m, T_{\C}(e), c \odot n}} \ar[d]^{(u^{\M}_e)_m, 1} & F(m, T_{\C}(e) \odot (c \odot n))   \ar[d]^{1, (u^{\N}_e)_{c \odot n}}\\
F((e \odot m) \odot c, n)  \ar[r]_{b_{e \odot m, c, n}} & F(e \odot m, c \odot n) \ar[r]_{\hat{b}_{m, e, c \odot n}} & F(m, e \odot (c \odot n)) 
 } \]
The top pentagon commutes by the balanced $\C$-module functor $F: \M \times \N \to \D$. The left-down square commutes by the naturality of the balanced $\C$-module structure $b$ on $F$.
The right-down square commutes by the diagram (\ref{diag-b-hatb}).
One can check that the rightmost diagram of (\ref{app-rem}) commutes. Then the middle hexagon of (\ref{app-rem}) commutes.
\end{proof}

\begin{cor}
By the commutativities of the diagrams (\ref{diag-em}) and (\ref{diag-emne}), the following diagram commutes
\[ 
\xymatrix@=4ex{
F(m \odot c, e \odot n) \ar[r]^{s^{F2}_{e, n}} \ar[d]_{b_{m, c, e \odot n}} & e \odot F(m \odot c, n) \ar[dd]^{1, b_{m, c, n}} \\
F(m, c \odot (e \odot n)) \ar[d]_{1, (s^{c \odot -}_{e, n})^{-1}} & \\
F(m, e \odot (c \odot n)) \ar[r]_{s^{F2}_{e, c \odot n}} & e \odot F(m, c \odot n)
} \]
\end{cor}

\begin{expl}
\label{22}
Let $\C, \D, \CP$ be multifusion categories over $\E$, $(\M, u^{\C}, u^{\D}) \in \BMod_{\C|\D}(\cat)$ and $(\N, \bar{u}^{\D}, \bar{u}^{\CP}) \in \BMod_{\D|\CP}(\cat)$. Then $(\M \boxtimes_{\D} \N, \tilde{u}^{\C}, \tilde{u}^{\CP})$ belongs to $\BMod_{\C|\CP}(\cat)$.
The left $\E$-module structure on $\M \boxtimes_{\D} \N$ is defined as $e \odot (m \boxtimes_{\D} n) \coloneqq (e \odot m) \boxtimes_{\D} n$, for $e \in \E$, $m \boxtimes_{\D} n \in \M \boxtimes_{\D} \N$.
The $\C$-$\CP$ bimodule structure on $\M \boxtimes_{\D} \N$ is defined as $c \odot (m \boxtimes_{\D} n) \odot p \coloneqq (c \odot m) \boxtimes_{\D} (n \odot p)$, for $c \in \C, p \in \CP$.
The monoidal natural isomorphism $\tilde{u}^{\C}$ is induced by $T_{\C}(e) \odot (m \boxtimes_{\D} n) = (T_{\C}(e) \odot m) \boxtimes_{\D} n \xrightarrow{(u^{\C}_e)_m, 1} (e \odot m) \boxtimes_{\D} n = e \odot (m \boxtimes_{\D} n)$.
The monoidal isomorphism $\tilde{u}^{\CP}$ is induced by
$ (m \boxtimes_{\D} n) \odot T_{\CP}(e) = m \boxtimes_{\D} (n \odot T_{\CP}(e)) \xrightarrow{1, (\bar{u}^{\CP}_e)_n} m \boxtimes_{\D}(e \odot n) \xrightarrow{1, (\bar{u}^{\D}_e)^{-1}_n} m \boxtimes_{\D} (T_{\D}(e) \odot n) \xrightarrow{b^{-1}_{m, T_{\D}(e), n}} (m \odot T_{\D}(e)) \boxtimes_{\D} n \xrightarrow{(u^{\D}_e)_m, 1} (e \odot m) \boxtimes_{\D} n = e \odot (m \boxtimes_{\D} n)$,
where $b$ is the balanced $\D$-module structure on $\boxtimes_{\D}: \M \times \N \to \M \boxtimes_{\D} \N$.
\end{expl}

Let $\C$ be a multifusion category over $\E$ and $\M \in \LMod_{\C}(\cat)$. Then $\M$ is \emph{enriched} in $\C$. That is, there exists an object $[x, y]_{\C} \in \C$ and a natural isomorphism $\homm_{\M}(c \odot x, y) \simeq \homm_{\C}(c, [x, y]_{\C})$ for $c \in \C$, $x, y \in \M$.
The category $\C_A$ is enriched in $\C$ and we have $[x, y]_{\C} = (x \otimes_A y^R)^L$ for $x, y \in \C_A$ by \cite[Expl.\,7.9.8]{Etingof}. 
By Prop.\,\ref{E-free-C}, the diagram
\[ \xymatrix{
T_{\C}(e) \otimes x \otimes_A y^R \ar[r]^{c_{e, x \otimes_A y^R}} \ar[d]_{c_{e, x}, 1} & x \otimes_A y^R \otimes T_{\C}(e)  \\
x \otimes T_{\C}(e) \otimes_A y^R \ar[r] & x \otimes_A T_{\C}(e) \otimes y^R \ar[u]_{1, c_{e,y^R}}
} \]
commutes for $e \in \E, x, y \in \C_A$, where $c$ is the central structure of the central functor $T_{\C}: \E \to \C$.

Let $\C$ be a multifusion category over $\E$ and $A, B$ be separable algebras in $\C$. By Prop.\,\ref{prop-ACB-Fun}, we have the following statements.
\begin{itemize}
\item There is an equivalence ${}_A\!\C \boxtimes_{\C} \C_B \xrightarrow{\simeq} {}_A\!\C_B$, $x \boxtimes_{\C} y \mapsto x \otimes y$ in $\BMod_{\E|\E}(\cat)$.
\item There is an equivalence $\Fun_{\C}(\C_A, \C_B) \xrightarrow{\simeq} {}_A\!\C_B$, $f \mapsto f(A)$ in $\BMod_{\E|\E}(\cat)$, whose inverse is defined as $x \mapsto - \otimes_A x$.
\end{itemize} 

\begin{prop}
\label{bimodule-eq}
Let $\C, \B, \D$ be multifusion categories over $\E$ and $\M \in \BMod_{\C|\B}(\cat)$ and $\N \in \BMod_{\C|\D}(\cat)$ . 
The functor $\Phi: \M^{L|\op|L} \boxtimes_{\C} \N \to \Fun^{\E}_{\C}(\M, \N)$, $m \boxtimes_{\C} n \mapsto [-, m]^R_{\C} \odot n$, is an equivalence of $\B$-$\D$-bimodules in $\cat$.
\end{prop}
\begin{proof}
There are equivalences of categories $\M^{L|\op|L} \boxtimes_{\C} \N \simeq \Fun_{\C}(\M, \N) \simeq \Fun_{\C}^{\E}(\M, \N)$ by \cite[Cor.\,2.2.5]{Liang} and Rem.\,\ref{p3}.
The $\B$-$\D$ bimodule structure on $\Phi$ is induced by
\[ (b \odot^L m) \boxtimes_{\C} (n \odot d) = (m \odot b^L) \boxtimes_{\C}( n \odot d) \mapsto [-, m \odot b^L]^R_{\C} \odot (n \odot d) \simeq ([- \odot b, m]^R_{\C} \odot n) \odot d = b \odot ([-, m]^R_{\C} \odot n) \odot d \]
for $m \in \M, n \in \N, b \in \B, d \in \D$,
where the equivalence is due to the canonical isomorphisms
$\homm_{\C}(c, [-, m \odot b^L]_{\C}) \simeq \homm_{\M}(c \odot -, m \odot b^L) \simeq \homm_{\M}(c \odot - \odot b, m) \simeq \homm_{\C}(c, [- \odot b, m]_{\C})$ for $c \in \C$.
The left $\E$-module structure on $\Phi$ is induced by the left $\B$-module structure on $\Phi$.  
Recall Expl.\,\ref{11}, \ref{22} and \ref{33}. It is routine to check that $\Phi$ satisfy the diagram (\ref{CD-bim-fun}).
\end{proof}

\begin{defn}
Let $\C, \D$ be multifusion categories over $\E$ and $\M \in \BMod_{\C|\D}(\cat)$.
$\M$ is right dualizable, if there exists an $\N \in \BMod_{\D|\C}(\cat)$ equipped with bimodule functors $u: \D \to \N \boxtimes_{\C} \M$ and $v: \M \boxtimes_{\D} \N \to \C$ in $\cat$ such that the composed bimodule functors
\[ \M \simeq \M \boxtimes_{\D} \D \xrightarrow{1_{\M} \boxtimes_{\D} u} \M \boxtimes_{\D} \N \boxtimes_{\C} \M \xrightarrow{v \boxtimes_{\C} 1_{\M}} \C \boxtimes_{\C} \M \simeq \M \]
\[ \N \simeq \D \boxtimes_{\D} \N \xrightarrow{u \boxtimes_{\D} 1_{\N}} \N \boxtimes_{\C} \M \boxtimes_{\D} \N \xrightarrow{1_{\N} \boxtimes_{\C} v} \N \boxtimes_{\C} \C \simeq \N \]
in $\cat$ are isomorphic to the identity functor.
In this case, the $\D$-$\C$ bimodule $\N$ in $\cat$ is left dualizable.
\end{defn}

\begin{prop}
The right dual of $\M$ in $\BMod_{\C|\D}(\cat)$ is given by a $\D$-$\C$ bimodule $\M^{L|\op|L}$ in $\cat$ equipped with two maps $u$ and $v$ defined as follows:
\begin{align}
u:& \D \to \Fun_{\C}(\M, \M) \simeq \M^{L|\op|L} \boxtimes_{\C} \M,  \qquad d \mapsto - \odot d, \nonumber \\
 v:& \M \boxtimes_{\D} \M^{L|\op|L} \to \C, \qquad x \boxtimes_{\D} y \mapsto [x, y]^R_{\C} \label{v-evaluation}
\end{align}
\end{prop}

\begin{proof}
By \cite[Thm.\,4.6]{LiangFH}, the object $\M^{L|\op|L}$ in $\BMod_{\D|\C}(\fcat)$, equipped with the maps $u$ and $v$, are the right dual of $\M$ in $\BMod_{\C|\D}(\fcat)$.
It is routine to check that $u$ is a $\D$-bimodule functor in $\cat$ and $v$ is a $\C$-bimodule functor in $\cat$.
\end{proof}

\begin{defn}
\label{defn-invertible}
Let $\C, \D$ be multifusion categories over $\E$.
An $\M \in \BMod_{\C|\D}(\cat)$ is \emph{invertible} if there is an equivalence $\D^{\rev} \simeq \Fun_{\C}^{\E}(\M, \M)$ of multifusion categories over $\E$.
If such an invertible $\M$ exists, $\C$ and $\D$ are said to be \emph{Morita equivalent in $\cat$}.
\end{defn}

\begin{prop}
Let $\M$ belong to $\BMod_{\C|\D}(\cat)$. The following conditions are equivalent.
\begin{itemize}
\item[(i)] $\M$ is invertible,
\item[(ii)] The functor $\D^{\rev} \rightarrow \Fun^{\E}_{\C}(\M, \M), d \mapsto - \odot d$ is an equivalence of multifusion categories over $\E$,
\item[(iii)] The functor $\C \rightarrow \Fun^{\E}_{|\D}(\M, \M), c \mapsto c \odot -$ is an equivalence of multifusion categories over $\E$.
\end{itemize}
\end{prop}
\begin{proof}
We obtain (i) $\Leftrightarrow$ (ii) by the Def.\,\ref{defn-invertible}.
Since $\Fun^{\E}_{\Fun^{\E}_{\C}(\M, \M)}(\M, \M)$ and $\C$ are equivalent as multifusion categories over $\E$ by Prop.\ref{double-centralizer}, we obtain (ii) $\Leftrightarrow$ (iii).
\end{proof}

\subsection{Characterization of Morita equivalence in $\cat$}
\begin{conv}
Throughout this subsection, we consider multifusion categories $\C$ over $\E$ with the property that $\E \to Z(\C)$ is fully faithful.
\end{conv}
Let $\C$ and $\D$ be multifusion categories over $\E$. 
We use $\beta$ and $\gamma$ to denote the central structures of the central functors $T_{\C}: \E \to \C$ and $T_{\D}: \E \to \D$ respectively.
\begin{thm}
\label{Morita-only-if}
 Let $\M$ be invertible in $\BMod_{\C|\D}(\cat)$. 
The left action of $Z(\C, \E)$ and the right action of $Z(\D, \E)$ on $\Fun^{\E}_{\C|\D}(\M, \M)$ induce an equivalence of multifusion categories over $\E$
\[ Z(\C, \E) \xrightarrow{L} \Fun^{\E}_{\C|\D}(\M, \M) \xleftarrow{R} Z(\D, \E) \]
Moreover, $Z(\C, \E)$ and $Z(\D, \E)$ are equivalent as braided multifusion categories over $\E$.
\end{thm}
\begin{proof}
Since $\M$ is invertible, the functor $\C \to \Fun_{\D^{\rev}}(\M, \M)$, $z \mapsto z \odot -$ is a monoidal equivalence over $\E$.
Then the induced monoidal equivalence $L: Z(\C, \E) \xrightarrow{(z, \beta_{z, -}) \mapsto (z \odot -, \beta_{z, -})} \Fun^{\E}_{\C|\D}(\M, \M)$ is constructed as follows.
\begin{itemize}
\item An object $z \in Z(\C, \E)$ is an object $z \in \C$, equipped with a half-braiding $\beta_{z,c}: z \otimes c \rightarrow c \otimes z$ for all $c \in \C$, such that the composition $z \otimes T_{\C}(e) \xrightarrow{\beta_{z,T_{\C}(e)}} T_{\C}(e) \otimes z \xrightarrow{\beta_{T_{\C}(e),z}} z \otimes T_{\C}(e), e \in \E$, equals to identity.

\item An object $z \odot -$ in $\Fun^{\E}_{\C|\D}(\M, \M)$ is an object $z \odot -$ in $\Fun^{\E}_{\D^{\rev}}(\M, \M)$ for $z \in \C$, equipped with a natural isomorphism $z \odot c \odot - \xrightarrow{\beta_{z,c}} c \odot z \odot -$ for $c \in \C, - \in \M$.
The left $\E$-module structure on $z \odot -$ is induced by Prop.\,\ref{F-in-Cat_E}.
Notice that $z \odot -$ satisfies the diagram (\ref{CD-bim-fun}) by the last diagram in Rem.\,\ref{CD-bim-fun-rem} and the equality $\beta_{T_{\C}(e),z} = \beta_{z, T_{\C}(e)}^{-1}$.
\end{itemize}

It is routine to check that $L$ is a monoidal functor over $\E$.
By the same reason, the functor $R: Z(\D, \E) \simeq Z(\D, \E)^{\rev} \xrightarrow{\simeq} \Fun^{\E}_{\C|\D}(\M, \M)$ is defined by $(a, \gamma_{a,-})\mapsto (- \odot a, \gamma_{a,-})$, where the second $\gamma_{a,-}$ is a natural isomorphism $- \odot a \odot d \xrightarrow{\gamma_{a, d}} - \odot d \odot a$ for $d \in \D$.
Thus $Z(\C, \E) \simeq Z(\D, \E)$.

Suppose $R^{-1} \circ L: Z(\C, \E) \to Z(\D, \E)$ carries $z, z'$ to $d, d'$, respectively. The diagram 
\[ \xymatrix{
z \odot (z' \odot x) \ar[r]^{\simeq} \ar[d]_{\beta_{z, z'}, 1_x} & (z' \odot x) \odot d \ar[d]^{\simeq} \ar[r]^{\simeq} & (x \odot d') \odot d \ar[d]^{1_x, \gamma_{d', d}} \\
z' \odot (z \odot x) \ar[r]_{\simeq} & z' \odot (x \odot d) \ar[r]_{\simeq} & (x \odot d) \odot d'
} \]
commutes for $x \in \M$.
Since the isomorphism $z \odot - \simeq - \odot d$ is a left $\C$-module natural isomorphism, the left square commutes. Since the isomorphism $z' \odot - \simeq - \odot d'$ is a right $\D$-module natural isomorphism, the right square commutes. 
Then the commutativity of the outer square implies that the equivalence $R^{-1} \circ L$ preserves braidings.
The equivalence $R^{-1} \circ L$, equipped with the monoidal natural isomorphism $L(T_{\C}(e)) = T_{\C}(e) \odot - \xrightarrow{v^{\M}_e} - \odot T_{\D}(e) = R(T_{\D}(e))$, is the braided equivalence over $\E$.
\end{proof}

\begin{lem}
\label{lem-4.17}
Let $\C$ be a fusion category over $\E$ such that the central functor $T_{\C}: \E \to \C$ is fully faithful. Let $\forget: Z(\C, \E) \rightarrow \C$ and $I_{\C}: \C \rightarrow Z(\C, \E)$ denote the forgetful functor and its right adjoint. 
\begin{itemize}
\item[(1)] There is a natural isomorphism $I_{\C}(x) \cong [\unit_{\C}, x]_{Z(\C, \E)}$ for all $x \in \C$.
\item[(2)] The object $A \coloneqq I_{\C}(\unit_{\C})$ is a connected \'{e}tale algebra in $Z(\C, \E)$; moreover for any $x \in \C$, the object $I_{\C}(x)$ has a natural structure of a right $A$-module.
\item[(3)] The functor $I_{\C}$ induces an equivalence of fusion categories $\C \simeq Z(\C, \E)_A$ over $\E$. 
Notice that $Z(\C, \E)_A$ is the category of right $A$-modules in $Z(\C,\E)$.
\end{itemize}
\end{lem}
\begin{proof}
For any $z \in Z(\C, \E), x \in \C$, we have the equivalences
$\homm_{Z(\C, \E)}(z, I_{\C}(x)) \simeq \homm_{\C}(z, x) \simeq \homm_{Z(\C, \E)}(z, [\unit_{\C}, x]_{Z(\C, \E)})$. By Yoneda lemma, we obtain $I_{\C}(x) \cong [\unit_{\C}, x]_{Z(\C, \E)}$.

Since $T_{\C}: \E \to \C$ is fully faithful, the forgetful functor $\forget: Z(\C, \E) \to \C$ is surjective by \cite[Lem.\,3.12]{DNO}.
By \cite[Lem.\,3.5]{DMNO}, the object $A$ is a connected \'etale algebra and there is a monoidal equivalence $\C \simeq Z(\C, \E)_A$. 
More explicitly,  for any object $x \in \C$, the object $I_{\C}(x) = [\unit_{\C}, x]_{Z(\C, \E)}$ is a right $A$-module and the monoidal functor 
\[ I_{\C} = [\unit_{\C}, -]_{Z(\C, \E)}: \C \rightarrow Z(\C, \E)_A  \]
is a monoidal equivalence.
The left $A$-module structure on $I_{\C}(x)$ is given by
$ A \otimes I_{\C}(x) \xrightarrow{\beta_{A, I_{\C}(x)}} I_{\C}(x) \otimes A \rightarrow I_{\C}(x)$.
One can check that for $x = \forget(z) \in \C$ with $z \in Z(\C, \E)$, one have $I_{\C}(x) \cong z \otimes A$ (as $A$-modules).
The monoidal structure on $I_{\C}$ is induced by
\[ \mu_{x,y}: I_{\C}(x \otimes y) = [\unit_{\C}, \forget(z) \otimes y]_{Z(\C, \E)} \simeq z \otimes [\unit_{\C}, y]_{Z(\C, \E)} = z \otimes A \otimes_A I_{\C}(y) = I_{\C}(x) \otimes_A I_{\C}(y) \]
for $x, y \in \C$.
Since $\forget$ is surjective, $\mu_{x, y}$ is always an isomorphism. 
$Z(\C, \E)_A$ can be identified with a subcategory of the fusion category ${}_A\!{Z(\C, \E)}_A$.
Recall the central structure on the functor $\E \to {}_A\!{Z(\C, \E)}_A$ by Ex.\,\ref{Ex-A-bimodule}. 
The structure of monoidal functor over $\E$ on $I_{\C}$ is induced by $I_{\C}(T_{\C}(e)) = [\unit_{\C}, T_{\C}(e)]_{Z(\C,\E)} \simeq T_{\C}(e) \otimes A$.
\end{proof}

\begin{lem}
\label{lem-FPdim}
Let $\C$ and $\D$ be fusion categories over $\E$ such that the central functors $T_{\C}: \E \to \C$ and $T_{\D}: \E \to \D$ are fully faithful. 
Suppose that $Z(\C, \E)$ is equivalent to $Z(\D, \E)$ as braided fusion categories over $\E$. 
We have $\FPdim(\C) = \FPdim(\D)$ and $\FPdim(I_{\C}(\unit_{\C})) = \FPdim(I_{\D}(\unit_{\D})) = \frac{\FPdim(\C)}{\FPdim(\E)}$, where $\FPdim$ is the Frobenius-Perron dimension.
\end{lem}
\begin{proof}
$Z(\C, \E)$ is a subcategory of $Z(\C)$. By \cite[Thm.\,3.14]{EGNO1}, we obtain the equation
\[ \FPdim(Z(\C, \E))\FPdim(Z(\C, \E)') = \FPdim(Z(\C)) \FPdim(Z(\C, \E) \cap Z(\C)')\]
Since the equations $Z(\C, \E)' = \E$, $Z(\C)' = \vect$ and $\FPdim(Z(\C)) = \FPdim(\C)^2$ (recall \cite[Thm.\,7.16.6]{Etingof}) hold, we get the equation
\begin{equation}
\label{eq-overE-center}
 \FPdim(Z(\C, \E)) = \frac{\FPdim(\C)^2}{\FPdim(\E)} 
\end{equation}
 Since $Z(\C, \E) \simeq Z(\D, \E)$ and the numbers $\FPdim(\C)$ and $\FPdim(\D)$ are positive, $\FPdim(\C) = \FPdim(\D)$.
  
Since $\forget: Z(\C, \E) \to \C$ is surjective, we get the equation
\[  \FPdim(I_{\C}(\unit_{\C})) = \frac{\FPdim(Z(\C, \E))}{\FPdim(\C)} = \frac{\FPdim(\C)}{\FPdim(\E)} \]
by \cite[Lem.\,6.2.4]{Etingof} and the equation (\ref{eq-overE-center}).
Then we have $\FPdim(I_{\C}(\unit_{\C})) = \FPdim(I_{\D}(\unit_{\D}))$.
\end{proof}

\begin{lem}
\label{lem-xtod}
Suppose that $f: Z(\C) \xrightarrow{\simeq} Z(\D)$ is an equivalence of braided multifusion categories and $u_e: f(T_{\C}(e)) \simeq T_{\D}(e)$ is a monoidal natural isomorphism in $Z(\D)$ for all $e \in \E$. Then $f$ induces an equivalence $Z(\C, \E) \simeq Z(\D, \E)$ of braided multifusion categories over $\E$.
\end{lem}
\begin{proof}
Suppose that $f: Z(\C) \to Z(\D)$ maps $(x, \beta_{x, -})$ to $(f(x), \gamma_{f(x), -})$. If the object $(x, \beta_{x, -})$ belongs to $Z(\C, \E)$, the object $(f(x), \gamma_{f(x), -})$ belongs to $Z(\D, \E)$ by the commutativity of the following diagram.
\[\xymatrix{
f(x \otimes T_{\C}(e)) \ar[d]_{\beta_{x,T_{\C}(e)}} \ar[r] & f(x) \otimes f(T_{\C}(e)) \ar[r]^{1,u_e} \ar[d]^{\gamma_{f(x), f(T_{\C}(e))}} & f(x) \otimes T_{\D}(e) \ar[d]^{\gamma_{f(x), T_{\D}(e)}} \\
f(T_{\C}(e) \otimes x) \ar[d]_{\beta_{T_{\C}(e), x}} \ar[r]  & f(T_{\C}(e)) \otimes f(x) \ar[r]^{u_e, 1} \ar[d]^{\gamma_{f(T_{\C}(e)), f(x)}} & T_{\D}(e) \otimes f(x) \ar[d]^{\gamma_{T_{\D}(e), f(x)}} \\
f(x \otimes T_{\C}(e)) \ar[r] & f(x) \otimes f(T_{\C}(e)) \ar[r]^{1,u_e} & f(x) \otimes T_{\D}(e)
}\]
Since $f$ is the braided functor, the left two squares commute.
The right-upper square commutes by the naturality of $\gamma_{f(x), -}$.
The right-down square commutes by reason that $u_e$ is a natural isomorphism in $Z(\D)$. 
Since the equation $\beta_{T_{\C}(e),x} \circ \beta_{x,T_{\C}(e)} = \id$ holds, we obtain the equation $\gamma_{T_{\D}(e),f(x)} \circ \gamma_{f(x),T_{\D}(e)} = \id$. Then $f$ induces an equivalence $Z(\C, \E) \simeq Z(\D, \E)$.
\end{proof}

\begin{expl}
\label{4.18}
Let $\C$ be a fusion category over $\E$ and $A$ a separable algebra in $\C$.
By \cite[Rem.\,7.16.3]{Etingof}, there is a monoidal equivalence $\Phi: Z(\C) \to Z({{}_A\!{\C}_A})$, $(z, \beta_{z,-}) \mapsto (z \otimes A, \beta_{z \otimes A})$,
where $\beta_{z \otimes A}$ is induced by
\[ z \otimes A \otimes_A x \cong z \otimes x \xrightarrow{\beta_{z, x}} x \otimes z \cong x \otimes_A A \otimes z \xrightarrow{1, \beta^{-1}_{z,A}} x \otimes_A z \otimes A, \quad \forall x \in {}_A\!\C_A \]
$\Phi$ induces the monoidal equivalence $Z(\C, \E) = \E'|_{Z(\C)} \simeq \E'|_{Z({}_A\!\C_A)} = Z({}_A\!\C_A, \E)$.
Recall the central structure on the functor $I: \E \to {}_A\!{\C}_A$ in Ex.\,\ref{Ex-A-bimodule}. 
We obtain $\Phi(T_{\C}(e)) = T_{\C}(e) \otimes A = I(e)$. 
Then $Z(\C, \E) \simeq Z({{}_A\!{\C}_A}, \E)$ is the monoidal equivalence over $\E$.

Let $\C_A$ be an indecomposable left $\C$-module in $\fcat$. By \cite[Prop.\,8.5.3]{Etingof}, $\Phi: Z(\C) \simeq Z({}_A\!\C_A)$ is the equivalence of braided fusion categories. By Lem.\ref{lem-xtod}, $\Phi: Z(\C, \E) \simeq Z({}_A\!\C_A, \E)$ is the equivalence of braided fusion categories over $\E$.
\end{expl}

\begin{lem}
\label{rem-FPdim}
Let $\C$ be a fusion category over $\E$ and $\M$ an indecomposable left $\C$-module in $\fcat$. Then $\FPdim(\C) = \FPdim(\Fun_{\C}(\M, \M))$.
\end{lem}
\begin{proof}
Since $\M$ is a left $\C$-module in $\fcat$, there is a separable algebra $A$ in $\C$ such that $\M \simeq \C_A$.
Recall the equivalences $Z(\C, \E) \simeq Z({}_A\!\C_A, \E)$ in Expl.\,\ref{4.18} and ${}_A\!\C_A \simeq \Fun_{\C}(\C_A, \C_A)^{\rev}$ in Prop.\,\ref{A-fun-bim-eq}. Then we get the equations
\[\frac{\FPdim(\C)^2}{\FPdim(\E)} = \FPdim(Z(\C, \E)) = \FPdim(Z({}_A\!\C_A, \E)) = \frac{\FPdim({}_A\!\C_A)^2}{\FPdim(\E)} = \frac{\FPdim(\Fun_{\C}(\C_A, \C_A))^2}{\FPdim(\E)} \]
The first and third equations are due to the equation (\ref{eq-overE-center}).
Since the Frobenius-Perron dimensions are positive, the result follows.
\end{proof}

Thm.\,8.12.3 of \cite{Etingof} says that two finite tensor categories $\C$ and $\D$ are Morita equivalent if and only if $Z(\C)$ and $Z(\D)$ are equivalent as braided tensor categories. 
The statement and the proof idea of Thm.\,\ref{Thm-Morita-eq} comes from which of Thm.\,8.12.3 in \cite{Etingof}.

\begin{thm}
\label{Thm-Morita-eq}
Let $\C$ and $\D$ be fusion categories over $\E$ such that the central functors $T_{\C}: \E \to \C$ and $T_{\D}: \E \to \D$ are fully faithful. $\C$ and $\D$ are Morita equivalent in $\cat$ if and only if $Z(\C, \E)$ and $Z(\D, \E)$ are equivalent as braided fusion categories over $\E$.
\end{thm}
\begin{proof}
The "only if" direction is proved in Thm.\,\ref{Morita-only-if}.

Let $\C, \D$ be fusion categories over $\E$ such that there is an equivalence $a: Z(\C, \E) \xrightarrow{\simeq} Z(\D, \E)$ as braided fusion categories over $\E$. 
Since $I_{\D}(\unit_{\D})$ is a connected \'{e}tale algebra in $Z(\D, \E)$, $L \coloneqq a^{-1}(I_{\D}(\unit_{\D}))$ is a connected \'{e}tale algebra in $Z(\C, \E)$. 
By Lem.\,\ref{lem-4.17}, there is an equivalence 

\[ \D \simeq Z(\C, \E)_L \] 
of fusion categories over $\E$.

By \cite[Prop.\,2.7]{DMNO}, the category $\C_L$ of $L$-modules in $\C$ is semisimple.
Note that the algebra $L$ is indecomposable in $Z(\C, \E)$ but $L$ might be decomposable as an algebra in $\C$, i.e. the category ${{}_L\!{\C}_L}$ is a multifusion category. It has a decomposition
\[ {{}_L\!{\C}_L} = \bigoplus_{i, j \in J} \big({{}_L\!{\C}_L}\big)_{ij} \]
where $J$ is a finite set and each $\big({{}_L\!{\C}_L} \big)_{ii}$ is a fusion category. 
Let $L = \bigoplus_{i \in J} L_i$ be the decomposition of $L$ such that ${{}_{L_i}\!{\C}_{L_i}}  \simeq \big({{}_L\!{\C}_L}\big)_{ii}$. Here $L_i, i \in J$, are indecomposable algebras in $\C$ such that the multiplication of $L$ is zero on $L_i \otimes L_j, i \neq j$. 

Next we want to show that there is an equivalence $Z(\C, \E)_L \simeq {{}_{L_i}\!{\C}_{L_i}}$ of fusion categories over $\E$.
Consider the following commutative diagram of monoidal functors over $\E$:
\[ \xymatrix{
Z(\C, \E) \ar[rr]^{z \mapsto z \otimes L_i} \ar[d]_{z \mapsto z \otimes L} & & Z({{}_{L_i}\!{\C}_{L_i}}, \E) \ar[d]^{\forget} \\
\qquad \quad Z(\C,\E)_L \subset {}_L\!Z(\C, \E)_L \ar[r]_(0.75){\forget} &  {{}_L\!{\C}_L}  \ar[r]_{\pi_i} & {{}_{L_i}\!{\C}_{L_i}} 
} \]
$\pi_i$ is projection and $\pi_i(x \otimes L) = x \otimes L_i$.
The top arrow is the equivalence by Expl.\,\ref{4.18}.
Next we calculate the Frobenius-Perron dimensions of the categories $Z(\C, \E)_L$ and ${{}_{L_i}\!{\C}_{L_i}}$:
\[ \FPdim(Z(\C, \E)_L) = \frac{\FPdim(Z(\C, \E))}{\FPdim(L)} = \FPdim(\C) = \FPdim({{}_{L_i}\!{\C}_{L_i}}) \]
The first equation is due to \cite[Lem.\,3.11]{DMNO}.
The second equation is due to $\FPdim(L) = \FPdim(I_{\C}(\unit_{\C}))= \FPdim(Z(\C, \E))/\FPdim(\C)$ by Lem.\,\ref{lem-FPdim}.
The third equation is due to Lem.\,\ref{rem-FPdim}.
Since $\pi_i \circ \forget$ is also surjective, $\pi_i \circ \forget$ is an equivalence.
Then we have monoidal equivalences over $\E$: $\D \simeq Z(\C, \E)_L \simeq {{}_{L_i}\!{\C}_{L_i}} \simeq \Fun_{\C}(\C_{L_i}, \C_{L_i})^{\rev}$.
\end{proof}

\subsection{Modules over a braided fusion category over $\E$}
Let $\C$ and $\D$ be braided fusion categories over $\E$. 
In this subsection, fusion categories $\M$ over $\E$ with the property that $\E \to Z(\M)$ is fully faithful.

\begin{defn}
\label{defn-m-over-braided-fusion}
The 2-category $\LMod_{\C}(\Alg(\cat))$ consists of the following data.
\begin{itemize}
\item A class of objects in $\LMod_{\C}(\Alg(\cat))$. An object $\M \in \LMod_{\C}(\Alg(\cat))$ is a fusion category $\M$ over $\E$ equipped with a braided monoidal functor $\phi_{\M}: \overline{\C} \rightarrow Z(\M, \E)$ over $\E$.

\item For objects $\M, \N$ in $\LMod_{\C}(\Alg(\cat))$, a 1-morphism $F: \M \rightarrow \N$ in $\LMod_{\C}(\Alg(\cat))$ is a monoidal functor $F: \M \rightarrow \N$ equipped with a monoidal isomorphism $u^{\M\N}: F \circ \phi_{\M} \Rightarrow \phi_{\N}$ such that the diagram
\begin{equation}
\label{braided-functor-E}
\begin{split}
 \xymatrix{
F(\phi_{\M}(c) \otimes m) \ar[r] \ar[d]_{\beta^{\M}_{c,m}} & F(\phi_{\M}(c)) \otimes F(m) \ar[r]^{u^{\M\N}_c, 1} & \phi_{\N}(c) \otimes F(m) \ar[d]^{\beta^{\N}_{c,F(m)}} \\
F(m \otimes \phi_{\M}(c)) \ar[r] & F(m) \otimes F(\phi_{\M}(c)) \ar[r]_{1, u^{\M\N}_c} & F(m) \otimes \phi_{\N}(c)
} 
\end{split}
\end{equation}
commutes for $c \in \overline{\C}, m \in \M$, where $(\phi_{\M}(c), \beta^{\M})  \in Z(\M, \E)$ and $(\phi_{\N}(c), \beta^{\N}) \in Z(\N, \E)$.

\item For 1-morphisms $F, G: \M \rightrightarrows \N$ in $\LMod_{\C}(\Alg(\cat))$, a 2-morphism $\alpha: F \Rightarrow G$ in $\LMod_{\C}(\Alg(\cat))$ is a monoidal isomorphism $\alpha$ such that the diagram 
\[ 
\xymatrixcolsep{0.5pc}
\xymatrixrowsep{1.2pc}
\xymatrix{
F(\phi_{\M}(c)) \ar[rr]^{\alpha_{\phi_{\M}(c)}} \ar[rd]_{u^{\M\N}_c} && G(\phi_{\M}(c)) \ar[ld]^{\tilde{u}^{\M\N}_c} \\
&\phi_{\N}(c) &
} \]
commutes for $c \in \overline{\C}$, where $u^{\M\N}$ and $\tilde{u}^{\M\N}$ are the monoidal isomorphisms on $F$ and $G$ respectively.
\end{itemize}
\end{defn}

\begin{rem}
If $F: \M \rightarrow \N$ is a 1-morphism in $\LMod_{\C}(\Alg(\cat))$, $F$ is a left $\overline{\C}$-module functor and a monoidal functor over $\E$.
By Lem.\,\ref{EAB}, the left $\overline{\C}$-module structure $s^F$ on $F$ is defined as
$F(c \odot m) = F(\phi_{\M}(c) \otimes m) \rightarrow F(\phi_{\M}(c)) \otimes F(m)  \xrightarrow{u^{\M\N}_c, 1} \phi_{\N}(c) \otimes F(m) = c \odot F(m)$ for all $c \in \overline{\C}, m \in \M$.
Let $u^{\C\M}: \phi_{\M} \circ T_{\C} \Rightarrow T_{\M}$ and $u^{\C\N}: \phi_{\N} \circ T_{\C} \Rightarrow T_{\N}$ be the structures of monoidal functors over $\E$ on $\phi_{\M}$ and $\phi_{\N}$ respectively.
The structure of monoidal functor over $\E$ on $F$ is induced by the composition
$v: F \circ T_{\M} \xRightarrow{1, (u^{\C\M})^{-1}} F \circ \phi_{\M} \circ T_{\C}  \xRightarrow{u^{\M\N}, 1} \phi_{\N} \circ T_{\C} \xRightarrow{u^{\C\N}} T_{\N}$.
\end{rem}

 The 2-category $\RMod_{\D}(\Alg(\cat))$ consists of the following data.
\begin{itemize}
\item An object $\M \in \RMod_{\D}(\Alg(\cat))$ is a fusion category $\M$ over $\E$ equipped with a braided monoidal functor $\phi_{\M}: \D \rightarrow Z(\M, \E)$ over $\E$.
\item 1-morphisms and 2-morphisms are similar with which in the Def.\,\ref{defn-m-over-braided-fusion}.
\end{itemize}
And the 2-category $\BMod_{\C|\D}(\Alg(\cat))$ consists of the following data.
\begin{itemize}
\item An object $\M \in \BMod_{\C|\D}(\Alg(\cat))$ is a fusion category $\M$ over $\E$ equipped with a braided monoidal functor $\phi_{\M}: \overline{\C} \boxtimes_{\E} \D \rightarrow Z(\M, \E)$ over $\E$.   
An object $\M \in \BMod_{\C|\D}(\Alg(\cat))$ is closed if $\phi_{\M}$ is an equivalence.
\item 1-morphisms and 2-morphisms are similar with which in the Def.\,\ref{defn-m-over-braided-fusion}.
\end{itemize}

\section{Factorization homology}
In this section, Sec.\,5.1 recalls the definitions of unitary categories, unitary fusion categories and unitary modular tensor categories over $\E$ (see \cite[Def.\,3.15, 3.16, 3.21]{TL}). Sec.\,5.2 recalls the theory of factorization homology. Sec.\,5.3 and Sec.\,5.4 compute the factorization homology of stratified surfaces with coefficients given by $\UMTCE$'s.

\subsection{Unitary categories}
\begin{defn}
A \emph{$\ast$-category} $\C$ is a $\Cb$-linear category equipped with a functor $\ast: \C \rightarrow \C^{\op}$ which acts as the identity map on objects and is anti-linear and involutive on morphisms. More explicitly, for any objects $x, y \in \C$, there is a map $\ast: \homm_{\C}(x, y) \to \homm_{\C}(y, x)$, such that
\[ (g \circ f)^{\ast} = f^{\ast} \circ g^{\ast}, \quad  (\lambda f)^{\ast} = \bar{\lambda} f^{\ast}, \quad (f^{\ast})^{\ast} = f \]
for $f: u \to v, g: v \to w, h: x \to y$, $\lambda \in \Cb^{\times}$. 
Here $\Cb$ denotes the field of complex numbers.

A \emph{$\ast$-functor} between two $\ast$-categories $\C$ and $\D$ is a $\Cb$-linear functor $F: \C \to \D$ such that $F(f^{\ast}) = F(f)^{\ast}$ for all $f \in \homm_{\C}(x, y)$.
A $\ast$-category is called \emph{unitary} if it is finite and the $\ast$-operation is positive, i.e. $f \circ f^{\ast} = 0$ implies $f=0$.
\end{defn}


\begin{defn}
A \emph{unitary fusion category} $\C$ is both a fusion category and a unitary category such that $\ast$ is compatible with the monoidal structures, i.e.
\[ (g \otimes h)^{\ast} = g^{\ast} \otimes h^{\ast}, \quad \forall g: v \to w, h: x \to y \]
\[ \alpha^{\ast}_{x, y, z} = \alpha_{x,y,z}^{-1}, \quad \gamma^{\ast}_x = \gamma^{-1}_x, \quad \rho^{\ast}_x = \rho^{-1}_x \]
for $x, y, z, v, w \in \C$, where $\alpha$, $\gamma, \rho$ are the associativity, the left unit and the right unit constraints respectively.
A unitary braided fusion category is a unitary fusion category $\C$ with a braiding $c$ such that $c^{\ast}_{x, y} = c^{-1}_{x, y}$ for any $x, y \in \C$.

A \emph{monoidal $\ast$-functor} between unitary fusion categories is a monoidal functor $(F, J): \C \to \D$, such that $F$ is a $\ast$-functor and $J^{\ast}_{x,y} = J^{-1}_{x,y}$ for $x, y \in \C$.
A \emph{braided $\ast$-functor} between unitary braided fusion categories is both a monoidal $\ast$-functor and a braided functor.
\end{defn}

\begin{rem}
Let $\C$ be a unitary fusion category. $\C$ admits a canonical spherical structure.
The unitary center $Z^{\ast}(\C)$ is defined as the fusion subcategory of the Drinfeld center $Z(\C)$, where $(x, c_{x,-}) \in Z^{\ast}(\C)$ if $c^{\ast}_{x,-} = c^{-1}_{x,-}$. $Z^{\ast}(\C)$ is a unitary braided fusion category and $Z^{\ast}(\C)$ is braided equivalent to $Z(\C)$ by \cite[Prop.\,5.24]{GHR}.
\end{rem}

\begin{defn}
A \emph{unitary $\E$-module category} $\C$ is an object $\C$ in $\cat$ 
such that $\C$ is a unitary category, and
the $\ast$ is compatible with the $\E$-module structure, i.e.
\[ (i \odot j)^{\ast} = i^{\ast} \odot j^{\ast}, \qquad \lambda^{\ast}_{e, \tilde{e}, x} = \lambda^{-1}_{e, \tilde{e}, x}, \qquad l^{\ast}_x = l^{-1}_x \]
for $i: e \to \tilde{e} \in \E, j: x \to y \in \C$, where $\lambda$ and $l$ are the module associativity and the unit constraints respectively.
Notice that symmetric fusion categories are all unitary.

Let $\C, \D$ be unitary $\E$-module categories. 
An \emph{$\E$-module $\ast$-functor} is an $\E$-module functor $(F, s): \C \to \D$ such that $F$ is a $\ast$-functor and $s^{\ast}_{e, x} = s^{-1}_{e,x}$ for $e \in \E, x \in \C$.
\end{defn}

\begin{rem}
Let $\C$ be an indecomposable unitary $\E$-module category.  
Then the full subcategory $\Fun^{\ast}_{\E}(\C, \C) \subset \Fun_{\E}(\C, \C)$ of $\E$-module $\ast$-functors is a unitary fusion category.
And the embedding $\Fun^{\ast}_{\E}(\C, \C) \to \Fun_{\E}(\C, \C)$ is the monoidal equivalence by \cite[Thm,\,5.3]{GHR}.
\end{rem}

\begin{defn}
\label{defn-UMTCE}
A \emph{unitary fusion category over $\E$} is a unitary fusion category $\A$ equipped with a braided $\ast$-functor $T_{\A}': \E \to Z(\A)$ such that the central functor $\E \to \A$ is fully faithful.
A \emph{unitary braided fusion category over $\E$} is a unitary braided fusion category $\C$ equipped with a braided $\ast$-embedding $T_{\C}: \E \to \C'$.
A \emph{unitary modular tensor category over $\E$ (or $\UMTCE$)} is a unitary braided fusion category $\C$ over $\E$ such that $\C' \simeq \E$.
\end{defn}

Let $\C$ be a unitary fusion category.
\begin{defn}
Let $(A, m: A \otimes A \to A, \eta: \unit_{\C} \to A)$ be an algebra in $\C$. 
A \emph{$\ast$-Frobenius algebra in $\C$} is an algebra $A$ in $\C$ such that the comultiplication $m^{\ast}: A \to A \otimes A$ is an $A$-bimodule map. 
Let $A$ be a $\ast$-Frobenius algebra in $\C$ and $\M$ a unitary left $\C$-module category. 
A \emph{left $\ast$-A-module in $\M$} is a left $A$-module $(M, q: A \otimes M \to M)$ such that $q^{\ast}: M \to A \otimes M$ is a left $A$-module map. 
\end{defn}

\begin{rem}
A $\ast$-Frobenius algebra in $\C$ is separable.
The full subcategory ${}_A\!\M^{\ast} \subset {}_A\!\M$ of left $\ast$-$A$-modules in $\M$ is a unitary category.
The embedding ${}_A\!\M^{\ast} \to {}_A\!\M$ is an equivalence.
Similarly, one can define $\M_A^{\ast}$ and ${}_A\!\M^{\ast}_A$.
\end{rem}

If the object $(x^L, \ev_x: x^L \otimes x \to \unit_{\C}, \coev_x: \unit_{\C} \to x \otimes x^L)$ is a left dual of $x$ in $\C$, then $(x^L, \coev^{\ast}_x: x \otimes x^L \to \unit_{\C}, \ev^{\ast}_x: \unit_{\C} \to x^L \otimes x)$ is the right dual of $x$ in $\C$.
Here we choose the duality maps $\ev_x$ and $\coev_x$ are \emph{normalized}.
That is, the induced composition
\[ \homm_{\C}(\unit_{\C}, x \otimes -) \xrightarrow{\ev_x} \homm_{\C}(x^L, -) \xrightarrow{\ev^{\ast}_x} \homm_{\C}(\unit_{\C}, - \otimes x) \]
is an isometry.
Then the normalized left dual $x^L$ is unique up to canonical unitary isomorphism. 
Let $(A, m, \eta)$ be a $\ast$-Frobenius algebra in $\C$. The object $(A, \eta^{\ast} \circ m: A \otimes A \to \unit_{\C}, m^{\ast} \circ \eta: \unit_{\C} \to A \otimes A)$ is the left (or right) dual of $A$ in $\C$.

\begin{defn}
A $\ast$-Frobenius algebra $A$ in $\C$ is \emph{symmetric} if the two morphisms $\Phi_1 = \Phi_2$ in $\homm_{\C}(A, A^L)$, where
\[ \Phi_1 \coloneqq [(\eta^{\ast} \circ m) \otimes \id_{A^L}] \circ (\id_A \otimes \coev_A) \quad \hbox{and} \quad \Phi_2 \coloneqq [\id_{A^L} \otimes (\eta^{\ast} \circ m)] \circ (\ev^{\ast}_A \otimes \id_A) \]
\end{defn}

The following proposition comes from Hao Zheng's lessons. 
\begin{prop}
Let $\M$ be a unitary left $\C$-module category. Then there exists a symmetric $\ast$-Frobenius algebra $A$ such that $\M \simeq \C^{\ast}_A$ as unitary left $\C$-module categories.
\end{prop}

\subsection{Factorization homology for stratified surfaces}

The theory of factorization homology (of stratified spaces) is in \cite{AF, AFT2, Francis}.  

\begin{defn}
Let $\mathrm{Mfld}^{\ori}_n$ be the topological category whose objects are oriented $n$-manifolds without boundary. For any two oriented $n$-manifolds $M$ and $N$, the morphism space $\homm_{\mathrm{Mfld}^{\ori}_n}(M, N)$ is the space of all orientation-preserving embeddings $e: M \to N$, endowed with the compact-open topology. We define $\Mfld^{\ori}_n$ to be the symmetric monoidal $\infty$-category associated to the topological category $\mathrm{Mfld}^{\ori}_n$. The symmetric monoidal structure is given by disjoint union.
\end{defn}

\begin{defn}
The symmetric monoidal $\infty$-category $\Disk^{\ori}_n$ is the full subcategory of $\Mfld^{\ori}_n$ whose objects are disjoint union of finitely many $n$-dimensional Euclidean spaces $\coprod_I \Rb^n$ equipped with the standard orientation.
\end{defn}

\begin{defn}
Let $\V$ be a symmetric monoidal $\infty$-category. An \emph{$n$-disk algebra} in $\V$ is a symmetric monoidal functor $A: \Disk^{\ori}_n \to \V$.
\end{defn}

Let $\V_{\uty}$ be the symmetric monoidal (2,1)-category of unitary categories. The tensor product of $\V_{\uty}$ is Deligne tensor product $\boxtimes$. 
Expl.\,3.5 of \cite{LiangFH} gives examples of 0-, 1-, 2-disk algebras in $\V_{\uty}$.
A unitary braided fusion category gives a 2-disk algebra in $\V_{\uty}$. A 1-disk algebra in $\V_{\uty}$ is a unitary monoidal category. A 0-disk algebra in $\V_{\uty}$ is a pair $(\CP, p)$, where $\CP$ is a unitary category and $p \in \CP$ is a distinguished object.
We guess that the $n$-disk algebra in $\V_{\uty}$ equipped with the compatible $\E$-module structure, is the $n$-disk algebra both in $\V_{\uty}$ and $\cat$, for $n = 0, 1, 2$.

\begin{ass}
Let $\V^{\E}_{\uty}$ be the symmetric monoidal (2,1)-category of unitary $\E$-module categories.
We assume that a unitary braided fusion category over $\E$ gives a 2-disk algebra in $\V^{\E}_{\uty}$, 
a unitary fusion category over $\E$ gives a 1-disk algebra in $\V^{\E}_{\uty}$, 
and a unitary $\E$-module category equipped with a distinguished object gives a 0-disk algebra in $\V^{\E}_{\uty}$.
\end{ass}

\begin{defn}
An \emph{(oriented) stratified surface} is a pair
 $(\Sigma, \Sigma \xrightarrow{\pi} \{0, 1, 2\})$ 
 where $\Sigma$ is an oriented surface and $\pi$ is a map. The subspace $\Sigma_i \coloneqq \pi^{-1}(i)$ is called the \emph{$i$-stratum} and its connected components are called \emph{$i$-cells}. These data are required to satisfy the following properties.
\begin{enumerate}
\item[(1)] $\Sigma_0$ and $\Sigma_0 \cup \Sigma_1$ are closed subspaces of $\Sigma$.
\item[(2)] For each point $x \in \Sigma_1$, there exists an open neighborhood $U$ of $x$ such that $(U, U \cap \Sigma_1, U \cap \Sigma_0) \cong (\Rb^2, \Rb^1, \emptyset)$.
\item[(3)] For each point $x \in \Sigma_0$, there exists an open neighborhood $V$ of $x$ and a finite subset $I \subset S^1$, such that $(V, V \cap \Sigma_1, V \cap \Sigma_0) \cong (\Rb^2, C(I)\backslash \{\hbox{cone point} \}, \{\hbox{cone point} \})$, where $C(I)$ is the open cone of $I$ defined by $C(I) = I \times [0,1) / I \times \{0\}$.
\item[(4)] Each 1-cell is oriented, and each 0-cell is equipped with the standard orientation.
\end{enumerate}
\end{defn}

There are three important types of stratified 2-disks shown in \cite[Expl.\,3.14]{LiangFH}.

\begin{defn}
We define $\mathrm{Mfld}^{\str}$ to be the topological category whose objects are stratified surfaces and morphism space between two stratified surfaces M and N are embeddings $e: M \rightarrow N$ that preserve the stratifications, and the orientations on 1-, 2-cells. We define $\Mfld^{\str}$ to be the symmetric monoidal $\infty$-category associated to the topological category $\mathrm{Mfld}^{\str}$. The symmetric monoidal structure is given by disjoint union.
\end{defn}

\begin{defn}
Let $M$ be a stratified surface.
We define $\Disk^{\str}_M$ to be the full subcategory of $\Mfld^{\str}$ consisting of those disjoint unions of stratified 2-disks that admit at least one morphism into $M$.
\end{defn}

\begin{defn}
Let $\V$ be a symmetric monoidal $\infty$-category. A \emph{coefficient} on a stratified surface $M$ is a symmetric monoidal functor $A: \Disk^{\str}_M \rightarrow \V$.
\end{defn}
A coefficient $A$ provides a map from each $i$-cell of $M$ to an $i$-disk algebra in $\V$.

\begin{defn}
Let $\V$ be a symmetric monoidal $\infty$-category, $M$ a stratified surface, 
and $A: \Disk_{M}^{\str} \rightarrow \V$ a coefficient.
The \emph{factorization homology} of $M$ with coefficient in $A$ is an object of $\V$ defined as follows:
\[ \int_{M} A := \mathrm{Colim}\big((\Disk_{M}^{\str})_{/M} \xrightarrow{i} \Disk_{M}^{\str} \xrightarrow{A} \V \big) \] 
where $(\Disk_{M}^{\str})_{/M}$ is the over category of stratified 2-disks embedded in $M$. And the notation $\mathrm{Colim}\big((\Disk_{M}^{\str})_{/M} \xrightarrow{A \circ i} \V \big)$ denotes the colimit of the functor $A \circ i$.
\end{defn}

\begin{defn}
A \emph{collar-gluing} for an oriented $n$-manifold $M$ is a continuous map $f: M \rightarrow [-1, 1]$ to the closed interval such that restriction of $f$ to the preimage of $(-1, 1)$ is a manifold bundle. 
We denote a collar-gluing $f: M \rightarrow [-1, 1]$ by the open cover $M_{-} \cup_{M_{0} \times \mathbb{R}} M_{+} \simeq M$, where $M_{-} = f^{-1}([-1, 1))$, $M_{+} = f^{-1}((-1, 1])$ and $M_{0} = f^{-1}(0)$.
\end{defn}

\begin{thm}(\cite{AF} Lem.\,3.18).
Suppose $\V$ is presentable and the tensor product $\otimes: \V \times \V \rightarrow \V$ preserves small colimits for both variables. Then the factorization homology satisfies $\otimes$-excision property.
That is, for any collar-gluing $M_{-} \cup_{M_{0} \times \mathrm{R}} M_{+} \simeq M$, there is a canonical equivalence:
\[ \int_{M} A \simeq \int_{M_{-}} A \bigotimes\limits_{\int_{M_0 \times \mathbb{R}} A} \int_{M_{+}} A \]
\end{thm}

\begin{rem}
If $U$ is contractible, there is an equivalence $\int_{U} A \simeq A$ in $\V$.
\end{rem}

Generalization of the $\otimes$-excision property is the \emph{pushforward property.}
Let $M$ be an oriented $m$-manifold, $N$ an oriented $n$-manifold, possibly with boundary, and $A$ an $m$-disk algebra in a $\otimes$-presentable $\infty$-category $\V$. Let $f: M \rightarrow N$ be a continuous map which fibers over the interior and the boundary of $N$.
There is a pushforward functor $f_*$ sends an $m$-disk algebra $A$ on $M$ to the $n$-disk algebra $f_* A$ on $N$.
Given an embedding $e: U \rightarrow N$ where $U = \mathbb{R}^{n}$ or $\mathbb{R}^{n-1} \times [0, 1)$, an $n$-disk algebra $f_{*}A$ is defined as $(f_{*} A)(U) \coloneqq \int_{f^{-1}(e(U))} A$. 
Then there is a canonical equivalence in $\V$
\begin{equation}
\label{thm-pushford}
\int_{N} f_{*} A \simeq \int_{M} A 
\end{equation}

\subsection{Preparation}

\begin{lem}
\label{lem-pre1}
 Let $\C$ be a multifusion category over $\E$ such that $\E \to Z(\C)$ is fully faithful. Then the functor $\C \boxtimes_{Z(\C, \E)} \C^{\rev} \to \Fun_{\E}(\C, \C)$ given by $a \boxtimes_{Z(\C, \E)} b \mapsto a \otimes - \otimes b$ is an equivalence of multifusion categories over $\E$.
\end{lem}
\begin{proof}
$\C^{\rev}$ and $\C$ are the same as categories.
 The composed equivalence (as categories):
\[\C \boxtimes_{Z(\C, \E)} \C \xrightarrow{\id \boxtimes_{Z(\C, \E)} \delta^L} \C \boxtimes_{Z(\C, \E)} \C^{\op}  \xrightarrow{v} \C^{\rev} \boxtimes_{\E} \C \]
 carries $a \boxtimes_{Z(\C, \E)} b \mapsto a \boxtimes_{Z(\C, \E)} b^L \mapsto [a, b^L]_{\C^{\rev} \boxtimes_{\E} \C}^R$, where $v$ is induced by Thm.\,\ref{C-C-bi-functor} and Eq.\,(\ref{v-evaluation}).
Notice that the object $\C$ in $\LMod_{\C^{\rev} \boe \C}(\cat)$ is faithful.
 The composed equivalence
 \[ \C^{\rev} \boxtimes_{\E} \C \xrightarrow{\delta^R \boxtimes_{\E} \id} \C^{\op} \boxtimes_{\E} \C \rightarrow \Fun_{\E}(\C, \C)    \]
 \[ c \boxtimes_{\E} d \mapsto c^R \boxtimes_{\E} d \mapsto [-, c^R]^R_{\E} \odot d \]
maps $[a, b^L]_{\C^{\rev} \boxtimes_{\E} \C}^R$ to a functor $f \in \Fun_{\E}(\C, \C)$. 
Note that 
$\homm_{\C}([x, c^R]_{\E}^R \odot d, y) \simeq \homm_{\E}([x, c^R]^R_{\E}, [d, y]_{\E}) \simeq \homm_{\E}(\unit_{\E}, [d, y]_{\E} \otimes [x, c^R]_{\E}) \simeq \homm_{\C^{\op} \boe \C}(c^R \boe d, x \boe y) \simeq \homm_{\C^{\rev} \boxtimes_{\E} \C}(c \boxtimes_{\E} d, x^L \boxtimes_{\E} y)$,
which implies
\[ \homm_{\C}(f(x), y) \simeq \homm_{\C^{\rev} \boxtimes_{\E} \C}([a, b^L]_{\NN}^R, x^L \boxtimes_{\E} y) \simeq 
  \homm_{\C}(a \otimes x \otimes b, y)\]
 i.e. $f \simeq a \otimes - \otimes b$. 
 Here the second equivalence above holds by the equivalence $(x^L \boxtimes_{\E} y) \otimes [a, b^L]_{\NN} \simeq  [a, (x^L \boe y) \odot b^L]_{\NN} =  [a, y \otimes b^L \otimes x^L]_{\NN}$.

Then the functor $\Phi: \C \boxtimes_{Z(\C, \E)} \C^{\rev} \rightarrow \Fun_{\E}(\C, \C), a \boxtimes_{Z(\C, \E)} b \mapsto a \otimes - \otimes b$ is a monoidal equivalence.
Recall the central structures of the functors $T_{\C \boxtimes_{Z(\C, \E)}\C^{\rev}}: \E \to \C \boxtimes_{Z(\C, \E)} \C^{\rev}$ and $T: \E \to \Fun_{\E}(\C, \C)$ in Expl.\,\ref{ex3} and Expl.\,\ref{ex2} respectively. 
The structure of monoidal functor over $\E$ on $\Phi$ is induced by $\Phi \circ T_{\C \boxtimes_{Z(\C, \E)} \C^{\rev}}(e) = T_{\C}(e) \otimes - \otimes \unit_{\C} \simeq T_{\C}(e) \otimes - = T^e$.
\end{proof}

\begin{lem}
\label{lem-functor}
Let $\C$ be a multifusion category over $\E$ such that $\E \to Z(\C)$ is fully faithful and $\X$ a left $\C$-module.
There is an equivalence in $\cat$
\[ \C \boxtimes_{Z(\C, \E)} \Fun_{\C}(\X, \X) \simeq \Fun_{\E}(\X, \X) \]
\end{lem}
\begin{proof}
Corollary 3.6.18 of \cite{SL} says that there is an equivalence
\[ \Fun_{\C}(\X, \X) \simeq \Fun_{\C \boxtimes_{\E} \C^{\rev}}(\C, \Fun_{\E}(\X, \X)) \]
We have equivalences $\C \boxtimes_{Z(\C, \E)} \C^{\rev} \boxtimes_{\C^{\rev}} \C^{\op} \boxtimes_{\C \boe \C^{\rev}} \Fun_{\E}(\X, \X) \simeq \Fun_{\E}(\C, \C) \boxtimes_{\C \boe \C^{\rev}} \Fun_{\E}(\X, \X) \simeq \C^{\op} \boe \C \boxtimes_{\C \boe \C^{\rev}} \Fun_{\E}(\X, \X) \simeq \Fun_{\E}(\X, \X)$.
The first equivalence holds by the Lem.\,\ref{lem-pre1}.
\end{proof}

\begin{lem}
\label{Lem-last-step-main1}
Let $\C$ be a semisimple finite left $\E$-module. There is an equivalence $\C^{\op} \boxtimes_{\Fun_{\E} (\C, \C)} \C \simeq \E$ in $\cat$.
\end{lem}
\begin{proof}
The left $\Fun_{\E}(\C, \C)$-action on $\C$ is defined as $f \odot x \coloneqq f(x)$ for $f \in \Fun_{\E}(\C, \C), x \in \C$.
The composed equivalence
\[  \C^{\op} \boxtimes_{\Fun_{\E}(\C, \C)} \C \simeq \C \boxtimes_{\Fun_{\E}(\C, \C)^{\rev}} \C^{\op} \simeq \E \]
carries $a \boxtimes_{\Fun_{\E}(\C, \C)} b \mapsto b \boxtimes_{\Fun_{\E}(\C, \C)^{\rev}} a \mapsto [b, a]_{\E}^R$, where the second equivalence is due to Thm.\,\ref{double-centralizer} and Eq.\,(\ref{v-evaluation}).
\end{proof}

\subsection{Computation of factorization homology}

Modules over a fusion category over $\E$ and modules over a braided fusion category over $\E$ can be generalized to the unitary case automatically.
Let $\C$ be a unitary fusion category over $\E$.
A closed object in $\LMod_{\C}(\V^{\E}_{\uty})$ is an object $\M \in \V^{\E}_{\uty}$ equipped with a monoidal equivalence $(\psi, u): \C \to \Fun_{\E}(\M, \M)$ over $\E$ such that $\psi$ is a monoidal $\ast$-functor and $u^{\ast}_e = u^{-1}_e$ for $e \in \E$.
Let $\A$ and $\B$ be unitary braided fusion categories over $\E$.
A closed object in $\BMod_{\A|\B}(\Alg(\V^{\E}_{\uty}))$ is a unitary fusion category $\M$ over $\E$ equipped with a braided monoidal equivalence $(\phi, u): \overline{\A} \boe \B \to Z(\M, \E)$ over $\E$ such that $\phi$ is a braided $\ast$-functor and $u^{\ast}_e = u^{-1}_e$ for $e \in \E$.

\begin{defn}
A coefficient system $A: \Disk^{\str}_M \rightarrow \V^{\E}_{\uty}$ on a stratified surface $M$ is called \emph{anomaly-free in $\cat$} if the following conditions are satisfied:
\begin{itemize}
\item The target label for a 2-cell is given by a $\UMTCE$.
\item The target label for a 1-cell between two adjacent 2-cells labeled by $\A$(left) and $\B$(right) is given by a closed object in $\BMod_{\A|\B}(\Alg(\V^{\E}_{\uty}))$. 
\item The target label for a 0-cell as the one depicted in Figure \ref{f3} is given by a 0-disk algebra $(\CP, p)$ in $\V^{\E}_{\uty}$, where the unitary $\E$-module category $\CP$ is equipped with the structure of a closed left $\int_{M\setminus \{0\}} A$-module, i.e. 
\[ \int_{M\setminus \{0\}} A \simeq  \Fun_{\E}(\CP, \CP)  \]
\end{itemize}
\end{defn}

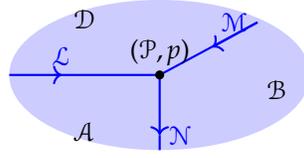
\begin{figure}
\center
\begin{tikzpicture}
\fill[fill=blue!20!white] (0,0) ellipse (2 and 1);
\draw [thick, blue, ->] (-2,0)--(-1.3,0)  node [above] {$\CL$};
\draw [thick, blue, -] (-1.4,0)--(0,0);
\draw [thick, blue, -] (0,0)--(1.3,0.7) node [left]{$\M$};
\draw [thick, blue, ->] (1.3, 0.7)--(0.7, 0.37);
\draw (-1, 0.5) node[above]{$\D$};
\draw (-1, -0.5) node[below]{$\A$};
\draw (1.3, -0.2) node[right]{$\B$};
\draw [thick, blue, ->] (0,0)--(0,-0.8) node[right]{$\N$};
\draw [thick, blue, -] (0,-0.8)--(0,-1) ;
\fill[fill=black] (0,0) circle [radius=0.06cm] node[above]{$(\CP, p)$};
\end{tikzpicture}
\caption{The figure depicts a stratified 2-disk with an anomaly-free coefficient system $A$ in $\cat$ determined by its target labels.}
 \label{f3}
\end{figure}

\begin{expl}
A stratified 2-disk $M$ is shown in Fig.\,\ref{f3}. An anomaly-free coefficient system $A$ on $M$ in $\cat$ is determined by its target labels shown in Fig.\,\ref{f3}
\begin{itemize}
\item The target labels for 2-cells: $\A$, $\B$ and $\D$ are $\UMTCE$'s.
\item 
The target labels for 1-cells: $\CL$ is a closed object in 
$\BMod_{\A|\D}(\Alg(\V^{\E}_{\uty}))$, $\M$ a closed object in $\BMod_{\D|\B}(\Alg(\V^{\E}_{\uty}))$ and $\N$ is a closed object in $\BMod_{\A|\B}(\Alg(\V^{\E}_{\uty}))$. 
\item The target labels for 0-cells: $(\CP, p)$ is a closed left module over $\CL \boxtimes_{\A^{\rev} \boe \D} (\M \boxtimes_{\B} \N^{\rev})$.
\end{itemize}
\end{expl}

The data of the coefficient system $A: \Disk^{\str}_{M} \to \V^{\E}_{\uty}$ shown in Fig.\,\ref{f3} are denoted as
\[ A = (\A, \B, \D; \CL, \M, \N; (\CP, p)) \]

\begin{expl}
\label{expl-disk-points}
Let $\C$ be a $\UMTCE$.
Consider an open disk $\mathring{\mathbb{D}}$ with two 0-cells $p_1$, $p_2$. And a coefficient system assigns $\C$ to the unique 2-cell and assigns $(\C, x_1), (\C, x_2)$ to the 0-cells $p_1, p_2$, respectively. 
By the $\otimes$-excision property, we have
\[ \int_{(\mathring{\mathbb{D}}; \emptyset; p_1 \sqcup p_2)} (\C; \emptyset; (\C, x_1), (\C, x_2)) \simeq \big( \C; \emptyset; (\C, x_1) \otimes_{\C} (\C, x_2) \big) \simeq \big(\C; \emptyset; (\C, x_1 \otimes x_2) \big)  \]
Notice the equivalence $\C \otimes_{\C} \C \simeq \C$ is defined as $x \otimes_{\C} y \mapsto x \otimes y$, whose inverse is defined as $m \mapsto \unit_{\C} \otimes_{\C} m$ for $x, y, m \in \C$.

Consider an open disk $\mathring{\mathbb{D}}$ with finitely many 0-cells $p_1, \dots, p_n$. And a coefficient system assigns $\C$ to the unique 2-cell and assigns $(\C, x_1), \dots, (\C, x_n)$ to the 0-cells $p_1, \dots, p_n$, respectively. 
We have
\[ \int_{(\mathring{\mathbb{D}}; \emptyset; p_1, \dots, p_n)} (\C; \emptyset; (\C, x_1), \dots, (\C, x_n)) \simeq \big(\C; \emptyset; (\C, x_1 \otimes \dots \otimes x_n) \big)  \]
\end{expl}

\begin{thm}
\label{main-thm1}
    Let $\C$ be a $\UMTCE$ and $x_1, \dots, x_n \in \C$. Consider the stratified sphere $S^2$ without $1$-stratum but with finitely many 0-cells $p_1, \dots, p_n$. Suppose a coefficient system assigns $\C$ to the unique 2-cell and assigns $(\C, x_1), \dots, (\C, x_n)$ to the 0-cells $p_1, \dots, p_n$, respectively. We have 
\begin{equation}
\label{thm}
\int_{(S^2; \emptyset; p_1, \dots ,p_n)} (\C; \emptyset; (\C, x_1), \dots, (\C, x_n)) \simeq \big(\E, [\unit_{\C}, x_1 \otimes \dots \otimes x_n ]_{\E}\big) 
\end{equation}
\end{thm}

\begin{proof}
If we map the open stratified disk $(\mathring{\mathbb{D}}; \emptyset; p_1, \dots, p_n)$ to the open stratified disk $(\mathring{\mathbb{D}}; \emptyset; p)$ and map the points $p_1, \dots, p_n$ to the point $p$. We have the following equivalence by Expl.\,\ref{expl-disk-points}
\[ \int_{(\mathring{\mathbb{D}}; \emptyset; p_1, \dots, p_n)} (\C; \emptyset; (\C, x_1), \dots, (\C, x_n))  \simeq \int_{(\mathring{\mathbb{D}}; \emptyset; p)} (\C; \emptyset; (\C, x_1 \otimes \dots \otimes x_n)) \]

On the stratified sphere $(S^2; \emptyset; p)$, we add an oriented 1-cell $S^1 \setminus {p}$ from $p$ to $p$, labelled by the 1-disk algebra $\C$ obtained by forgetting its 2-disk algebra structure.
We project the stratified sphere $(S^2; S^1\setminus {p}; p)$ directly to a closed stratified 2-disk $(\mathbb{D}; S^1 \setminus {p}; p)$ as shown in Fig.\,\ref{f1} (a). 
Notice that this projection preserves the stratification.
Applying the pushforward property (\ref{thm-pushford}) and the $\otimes$-excision property, we reduce the problem to the computation of the factorization homology of the stratified 2-disk. 
\[ \int_{(S^2; \emptyset; p_1, \dots, p_n)} (\C; \emptyset; (\C, x_1), \dots, (\C, x_n)) \simeq \int_{(\mathbb{D}; S^1\setminus p; p)} \big(\C \boe \overline{\C}; \C; (\C, x_1 \otimes \dots \otimes x_n)\big) \]
Notice that $\C \boxtimes_{\E} \overline{\C} \simeq Z(\C, \E)$.
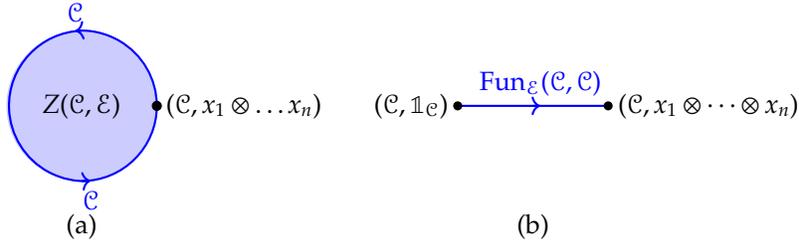
\begin{figure}
\center    
\begin{tikzpicture}
\fill[fill=blue!20!white] (0,0) circle [radius = 1cm];
\draw[thick, blue, ->] (0,-1) arc(-90:95:1cm) node[above] {$\C$};
\draw[thick, blue, ->] (0,1) arc(92:275:1cm) node [below] {$\C$};
\filldraw (1,0) circle [radius=0.06cm] node [right] {$(\C, x_1 \otimes \dots x_n)$};
\draw [thick, blue, ->] (5,0)--(6.1,0)  node [above] {$\Fun_{\E}(\C, \C)$};   
\draw [thick, blue] (6,0)--(7,0)node [right] {\textcolor{black}{$(\C, x_1 \otimes \dots \otimes x_n)$}}; 
\fill[fill=black] (7,0) circle [radius=0.06cm];
\fill[fill=black!] (5,0) circle [radius =0.06cm] node [left] {$(\C, \unit_{\C})$};
\node at (0,-1.6) {(a)}; \node at (6, -1.6) {(b)};
\node at (0, 0) {$Z(\C, \E)$};
 \end{tikzpicture}
 \caption{The figure depicts the two steps in computing the factorization homology of a sphere with the coefficient system given by a $\UMTCE$.}
 \label{f1}
 \end{figure}
Next we project the stratified 2-disk vertically onto the closed interval $[-1, 1]$ as shown in Fig.\,\ref{f1} (b). Notice that $\C \boxtimes_{Z(\C, \E)} \C^{\rev} \simeq \Fun_{\E}(\C, \C)$. 
The final result is expressed as a tensor product:
\[ \int_{(S^2; \emptyset; p_1, \dots , p_n)} (\C; \emptyset; (\C, x_1), \dots, (\C, x_n)) \simeq \big(\C \boxtimes_{\Fun_{\E}(\C, \C)} \C, \unit_{\C} \boxtimes_{\Fun_{\E}(\C, \C)}( x_1 \otimes \dots \otimes x_n)\big)  \]
By Lem.\,\ref{Lem-last-step-main1} and Lem.\,\ref{thm-internal-inverse},
the composed equivalence
 \[\C \boxtimes_{\Fun_{\E}(\C, \C)} \C \simeq \C^{\op} \boxtimes_{\Fun_{\E}(\C, \C)} \C \simeq \C \boxtimes_{\Fun_{\E}(\C, \C)^{\rev}} \C^{\op} \simeq \E\]
carries $x \boxtimes_{\Fun_{\E}(\C, \C)} y \mapsto x^R \boxtimes_{\Fun_{\E}(\C, \C)} y \mapsto y \boxtimes_{\Fun_{\E}(\C, \C)^{\rev}} x^R \mapsto [y, x^R]_{\E}^R \cong [x^R, y]_{\E}$. 
Taking $x = \unit_{\C}$ and $y = x_1 \otimes \dots \otimes x_n$ in the above composed equivalence, we obtain Eq.\,(\ref{thm}).  
\end{proof}

\begin{thm}
\label{main-thm2}
    Let $\C$ be a $\UMTCE$ and $x_1, \dots, x_n \in \C$. 
     Let $\Sigma_g$ be a closed stratified surface of genus $g$ without 1-stratum but with finitely many 0-cells $p_1, \dots, p_n$. 
    Suppose a coefficient system assigns $\C$ to the unique 2-cell and assigns $(\C, x_1), \dots, (\C, x_n)$ to the 0-cells $p_1, \dots, p_n$, respectively. We have 
\begin{equation}
\label{thm2}
\int_{(\Sigma_g; \emptyset; p_1, \dots ,p_n)} (\C; \emptyset; (\C, x_1), \dots, (\C, x_n)) \simeq \big(\E, [\unit_{\C}, x_1 \otimes \dots \otimes x_n \otimes (\eta^{-1}(A) \otimes_{T(A)} \eta^{-1}(A) )^{\otimes g}]_{\E}  \big) 
\end{equation} 
where $A$ is a symmetric $\ast$-Frobenius algebra in $\E$ such that there exists an equivalence $\eta: \C \simeq \E_A$ in $\V^{\E}_{\uty}$ and $T: \E \rightarrow \C'$ is the braided embedding.
\end{thm}
\begin{proof}
Since $\C$ is a unitary $\E$-module category, there exists a symmetric $\ast$-Frobenius algebra $A$ in $\C$ such that $\C \simeq^{\eta} \E_A$ in $\V^{\E}_{\uty}$.
Notice that Eq.\,(\ref{thm2}) holds for genus $g = 0$ by Thm.\,\ref{main-thm1}. Now we assume $g > 0$. The proof of Thm.\,\ref{main-thm1} implies that $\int_{S^1 \times \mathbb{R}} \C \simeq \Fun_{\E}(\C, \C)$. 
By Prop.\ref{prop-ACB-Fun}, Lem.\,\ref{equivalencethm} and Lem.\,\ref{lem-AEA}, the composed equivalence of categories 
\[ \Fun_{\E}(\E_A, \E_A)  \simeq {}_A\!\E_A   \simeq {}_A\!\E \boe \E_A   \simeq \E_A \boe \E_A  \]
 carries $\id \mapsto A  \mapsto  \bar{p} \boe \bar{q}   \mapsto \bar{p} \boe \bar{q}$,
 where $\bar{p} \boe \bar{q} \coloneqq \mathrm{colim} \big((A \otimes A) \boxtimes_{\E} A \rightrightarrows A \boxtimes_{\E} A \big)$.
Then the equivalence $\Fun_{\E}(\C, \C) \simeq \C \boxtimes_{\E} \C$ carries $\id_{\C} \mapsto p \boe q \coloneqq \colim \big(\eta^{-1}(A \otimes A) \boxtimes_{\E} \eta^{-1}(A) \rightrightarrows \eta^{-1}(A) \boxtimes_{\E} \eta^{-1}(A) \big)$.

Therefore, we have $\int_{S^1 \times \mathbb{R}} \C \simeq \Big(\C \boxtimes_{\E} \C, p \boe q \big)$.
As a consequence, when we compute the factorization homology, we can replace a cylinder $S^1 \times \mathbb{R}$ by two open 2-disks with two 0-cells as shown on the Fig.\,\ref{cylinder}, both of which are labelled by 
$\Big(\C \boxtimes_{\E} \C, p \boe q \Big)$,
or labelled by $(\C, p)$ and $(\C, q)$.

In this way, the genus is reduced by one.
By induction, we obtain the equation
\begin{align*}
& \int_{(\Sigma_g; \emptyset; p_1, \dots ,p_n)} (\C; \emptyset; (\C, x_1), \dots, (\C, x_n)) \simeq \int_{(\Sigma_{g-1}; \emptyset; p_1, \dots, p_n, p_{n+1}, p_{n+2})} \big( \C; \emptyset; (\C, x_1), \dots, (\C, x_n), (\C, p), (\C, q) \big) \\
&\simeq \int_{(\Sigma_0; \emptyset; p_1, \dots, p_n, \dots, p_{n+2g-1}, p_{n+2g})} \big( \C; \emptyset; (\C, x_1), \dots, (\C, x_n), (\C \boe \C, p \boe q)^g \big) \\
&\simeq \big( \E, [\unit_{\C}, x_1 \otimes \dots \otimes x_n \otimes (p \otimes q)^{\otimes g} ]_{\E} \big)   
 \end{align*}
where the notation $(\C \boe \C, p \boe q)^g$ denotes $g$ copies of $(\C \boe \C, p \boe q)$ and
\begin{align*}
p \otimes q & \simeq \colim \big( \eta^{-1}(A \otimes A) \otimes \eta^{-1}(A) \rightrightarrows \eta^{-1}(A) \otimes \eta^{-1}(A) \big) \\
& \simeq \colim \big( \eta^{-1}(A) \otimes T(A) \otimes \eta^{-1}(A) \rightrightarrows \eta^{-1}(A) \otimes \eta^{-1}(A) \big) \\
& \simeq \eta^{-1}(A) \otimes_{T(A)} \eta^{-1}(A)
\end{align*}
Since factorization homology and $p \boe q$ are both defined by colimits, we exchange the order of two colimits in the first equivalence.
The second equivalence is induced by the composed equivalence $\eta^{-1}(A \otimes A) \simeq A \odot \eta^{-1}(A) = T(A) \otimes \eta^{-1}(A) \simeq \eta^{-1}(A) \otimes T(A)$. 
Since $T(A)$ is an algebra in $\C$, we obtain the last equivalence.
\end{proof}

\begin{expl}
The unitary category $\Hb$ denotes the category of finite dimensional Hilbert spaces.
 Let $\E = \Hb$ and $\C = \UMTC$. 
We want to choose an algebra $A \in \Hb$ such that $\C \simeq^{\eta} \Hb_A$.
Suppose that $\eta^{-1}(A) \cong A$ and $T(A) \cong A$. Then $\eta^{-1}(A) \otimes_{T(A)} \eta^{-1}(A) \cong A \otimes_A A \cong A$ and
\[ \int_{(\Sigma_g; \emptyset; p_1, \dots ,p_n)} (\C; \emptyset; (\C, x_1), \dots, (\C, x_n))   \simeq \big(\Hb, \homm(\unit_{\C}, x_1 \otimes \dots \otimes x_n \otimes  A^{\otimes g})  \big) 
\]
The set $\mO(\C)$ denotes the set of isomorphism classes of simple objects in $\C$.
If $\eta^{-1}(A) = \oplus_{i \in \mO(\C)} i^R \otimes i = T(A)$, the distinguished object is $\homm(\unit_{\C}, x_1 \otimes \dots \otimes x_n \otimes (\oplus_i i^R \otimes i)^{\otimes g})$.
If $A = \oplus_{i \in \mO(\C)} \Cb$ and $\eta^{-1}(A) = \oplus_{i \in \mO(\C)} i = T(A)$, the distinguished object is $\homm(\unit_{\C}, x_1 \otimes \dots \otimes (\oplus_{i \in \mO(\C)} i)^{\otimes g})$.
\end{expl}

 \begin{figure}
 \center
\begin{tikzpicture}
\fill[fill=blue!20!white] (0,0) rectangle (2,1);
\fill[fill=blue!20!white] (2,0.5) ellipse (0.2cm and 0.5cm);
\fill[fill=white] (0,0.5) ellipse (0.2cm and 0.5cm);
\draw[thick, blue, ->] (0,1)--(1.1,1) node [above] {$\M$};
\draw[thick,blue] (2,0.5) ellipse (0.2cm and 0.5cm);
\draw[thick, blue] (0,0.5) ellipse (0.2cm and 0.5cm);
\draw[thick, blue] (1,1)--(2,1);
\draw[thick, blue] (0,0) --(0.3,0) node [above] {$\C$}--(2,0);
\fill[fill=blue!20!white] (6,1) [out= 0, in= 130] to (7,0.5) [out =220, in=0] to (6,0)--(6,1);
\fill[fill=blue!20!white] (9.5,1) [out=180, in=50] to (8.5, 0.5) [out=-50, in=180] to (9.5,0)--(9.5,1);
\fill[white] (6,0.5) ellipse (0.2cm and 0.5cm); \fill[white] (9.5, 0.5) ellipse (0.2cm and 0.5cm);
\draw[thick, blue,->] (6,1) [out= 0, in= 130] to (7,0.5); 
\draw[thick, blue] (7,0.5) [out= 220, in= 0] to (6,0) ; 
\draw[thick, blue, ->] (9.5,1) [out=180, in= 50] to (8.5,0.5); 
\draw[thick, blue] (8.5,0.5)[out=-50, in=180] to (9.5,0);
\draw[thick, blue] (6,0.5) ellipse (0.2cm and 0.5cm); 
\draw[thick, blue] (9.5, 0.5) ellipse (0.2cm and 0.5cm);
\fill[black] (7, 0.5) circle [radius=0.06cm] node [right] {\small{$\X \boxtimes_{\E} \X^{\op}$}}; 
\fill[black] (8.5, 0.5) circle [radius=0.06cm];
\node[blue] at (6.5, 1.1) {$\M$}; \node[blue] at (9,1.1) {$\M^{\rev}$};
\node[blue] at (6.5, 0.4) {$\C$}; \node[blue] at (9, 0.4) {$\C$};
\node at (1, -0.4) {(a)}; \node at (8, -0.4) {(b)};
\end{tikzpicture} 
 \caption{Figure (a) shows a stratified cylinder with a coefficient system $(\C; \M; \emptyset)$, where $\C$ is a $\UMTCE$ and $\M$ is closed in $\BMod_{\C|\C}(\Alg(\V^{\E}_{\uty}))$. Figure (b) shows a disjoint union of two open disks with 2-cells labeled by $\C$, 1-cells labeled by $\M$ and $\M^{\rev}$, and 0-cells labeled by $\X$ and $\X^{\op}$.}
\label{cylinder}
\end{figure}
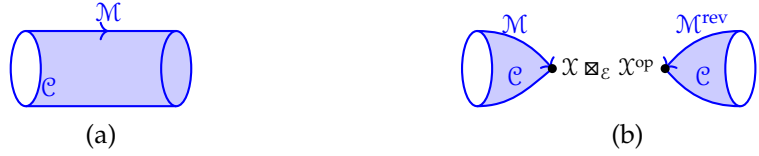

\begin{thm}
\label{main-thm3}
Let $(S^1 \times \Rb; \Rb)$ be the stratified cylinder shown in Fig.\,\ref{cylinder}.
in which the target label $\C$ is a $\UMTCE$ and the target label $\M$ is closed in $\BMod_{\C|\C}(\Alg(\V^{\E}_{\uty}))$. We have 
\[ \int_{(S^1 \times \Rb; \Rb)}(\C; \M; \emptyset) \simeq \Fun_{\E}(\X, \X) \]
where $\X$ is the unique (up to equivalence) left $\C$-module in $\cat$ such that $\M \simeq \Fun_{\C}(\X, \X)$. 
\end{thm}
\begin{proof}
By the equivalences $Z(\M^{\rev}, \E) \simeq \C \boe \overline{\C} \simeq Z(\C, \E)$, there exists a $\C$-module $\X$ such that $\M \simeq \Fun_{\C}(\X, \X)$ by Thm.\,\ref{Thm-Morita-eq}.
Therefore, we have $\int_{(S^1 \times \Rb; \Rb)}(\C; \M; \emptyset) \simeq \C \boxtimes_{Z(\C, \E)} \M \simeq \C \boxtimes_{Z(\C, \E)} \Fun_{\C}(\X, \X) \simeq \Fun_{\E}(\X, \X)$, which maps $\unit_{\C} \boxtimes_{Z(\C, \E)} \unit_{\M}$ to $\id_{\X}$.
The last equivalence is due to Thm.\,\ref{lem-functor}.
\end{proof}

\begin{conj}
Given any closed stratified surface $\Sigma$ and an anomaly-free coefficient system $A$ in $\cat$ on $\Sigma$, we have $\int_{\Sigma} A \simeq (\E, u_{\Sigma})$, where $u_{\Sigma}$ is an object in $\E$. 
\end{conj}

\appendix
\section{Appendix}

 \subsection{Central functors and other results}

Let $\D$ be a braided monoidal category with the braiding $c$ and $\M$ a monoidal category. 
\begin{defn}
\label{central defn 1}
A \emph{central structure} of a monoidal functor $F: \D \rightarrow \M$ is a braided monoidal functor $F': \D \rightarrow Z(\M)$ such that $F = \forget \circ F'$, where $\forget: Z(\M) \rightarrow \M$ is the forgetful functor. 
\end{defn}
A \emph{central functor} is a monoidal functor equipped with a central structure.
For any monoidal functor $F: \D \rightarrow \M$, the central structure of $F$ given in Def.\,\ref{central defn 1} 
is equivalent to the central structure of $F$ given in the Def.\,\ref{central defn 2}.

\begin{defn}
\label{central defn 2}
A \emph{central structure} of a monoidal functor $F: \D \rightarrow \M$ is a natural isomorphism $\sigma_{d, m}: F(d) \otimes m \rightarrow m \otimes F(d)$, $d \in \D$, $m \in \M$ which is natural in both variables such that the diagrams 
\begin{equation}
\label{center1}
\begin{split}
\xymatrixcolsep{0.5pc}
\xymatrixrowsep{1.2pc}
\xymatrix{
F(d) \otimes m \otimes m' \ar[rr]^{\sigma_{d, m \otimes m'}} \ar[dr]_{\sigma_{d, m}, 1}&  & m \otimes m' \otimes F(d)\\
  & m \otimes F(d) \otimes m'\ar[ur]_{1, \sigma_{d, m'}}&}
\end{split}
\end{equation} 
\begin{equation}
\label{center2}
\begin{split}
\xymatrix{
F(d) \otimes F(d') \otimes m \ar[d]_{J_{d, d'}, 1} \ar[r]^{1, \sigma_{d', m}} &F(d) \otimes m \otimes F(d') \ar[r]^{\sigma_{d, m}, 1} & m \otimes F(d) \otimes F(d') \ar[d]^{1, J_{d,d'}} \\
F(d \otimes d') \otimes m \ar[rr]^{\sigma_{d \otimes d', m}} & & m \otimes F(d \otimes d')
 }
\end{split}
\end{equation}
\begin{equation}
\label{center3}
\begin{split}
\xymatrix{
F(d) \otimes F(d') \ar[r]^{J_{d, d'}} \ar[d]_{\sigma_{d, F(d')}} & F(d \otimes d') \ar[d]^{F(c_{d, d'})} \\
 F(d') \otimes F(d) \ar[r]^{J_{d', d}} & F(d' \otimes d)
}
\end{split}
\end{equation}
commute for any $d, d' \in \D$ and $m, m' \in \M$, where $J$ is the monoidal structure of $F$. 
\end{defn}

\begin{prop}
Suppose $F: \D \to \M$ is a central functor. For any $d \in \D, m \in \M$, the following two diagrams commute
\begin{equation} 
\label{diagram-c}
\begin{split}
\xymatrixcolsep{0.5pc}
\vcenter{\xymatrix{
F(d) \otimes \unit_{\M} \ar[rr]^{\sigma_{d, \unit_{\M}}} \ar[rd]_{r_{F(d)}} & & \unit_{\M} \otimes F(d) \ar[ld]^{l_{F(d)}} \\
& F(d) &
}}
\qquad 
\vcenter{\xymatrix{
F(\unit_{\D}) \otimes m \ar[rr]^{\sigma_{\unit_{\D}, m}} \ar[rd]_{l_m} & & m \otimes F(\unit_{\D}) \ar[ld]^{r_m} \\
& m &
}}
 \end{split}
 \end{equation}
 Here $l_m: F(\unit_{\D}) \otimes m = \unit_{\M} \otimes m \to m$ and $r_m: m \otimes F(\unit_{\D}) = m \otimes \unit_{\M} \to m$, $m \in \M$ are the unit isomorphisms of the monoidal category $\M$. 
\end{prop}

\begin{proof}
Consider the diagram:
\[ \xymatrix@=4ex{
& F(d) \otimes (\unit_{\M} \otimes \unit_{\M}) \ar[r]^{\sigma_{d, \unit_{\M} \otimes \unit_{\M}}} \ar[d]^{1, l_{\unit_{\M}}} & (\unit_{\M} \otimes \unit_{\M}) \otimes F(d) \ar[d]^{l_{\unit_{\M}}, 1} \ar[rd] & \\
(F(d) \otimes \unit_{\M}) \otimes \unit_{\M} \ar[ru] \ar[r]^(0.55){r_{F(d)}, 1} \ar[rd]_(0.4){\sigma_{d, \unit_{\M}}, 1} & F(d) \otimes \unit_{\M} \ar[r]^{\sigma_{d, \unit_{\M}}} & \unit_{\M} \otimes F(d) & \unit_{\M} \otimes (\unit_{\M} \otimes F(d)) \ar[l]_(0.55){l_{\unit_{\M} \otimes F(d)}} \\
 & (\unit_{\M} \otimes F(d)) \otimes \unit_{\M} \ar[r] \ar[u]_{l_{F(d)}, 1} & \unit_{\M} \otimes (F(d) \otimes \unit_{\M}) \ar[ru]_{1, \sigma_{d, \unit_{\M}}} \ar[lu]_(0.4){l_{F(d) \otimes \unit_{\M}}} &
}\]
The outward hexagon commutes by the diagram (\ref{center1}).
The left-upper, right-upper and middle-bottom triangles commute by the monoidal category $\M$.
The middle-up square commutes by the naturality of the central structure $\sigma_{d,m}: F(d) \otimes m \to m \otimes F(d)$, $\forall d \in \D, m \in \M$.
The right-down square commutes by the naturality of the unit isomorphism $l_m:\unit_{\M} \otimes m \simeq m$, $m \in \M$.
Then the left-down triangle commutes. Since $- \otimes \unit_{\M} \simeq \id_{\M}$ is the natural isomorphism, the left triangle of (\ref{diagram-c}) commutes.

Consider the diagram:
\[ \xymatrix@=4ex{
& m \otimes F(\unit_{\D} \otimes \unit_{\D}) \ar[ld]_{1, J}  \ar[d]^{1, r_{\unit_{\D}}}& \\
m \otimes F(\unit_{\D}) \otimes F(\unit_{\D}) \ar[r]^(0.6){1, r_{F(\unit_{\D})}} \ar@/_0.4pc/[r]_(0.6){r_{m \otimes F(\unit_{\D})}} & m \otimes F(\unit_{\D})  & F(\unit_{\D} \otimes \unit_{\D}) \otimes m \ar[lu]_{\sigma_{\unit_{\D} \otimes \unit_{\D}, m}} \ar[dd]^{J,1} \ar[ld]^{r_{\unit_{\D}}, 1} \\
& F(\unit_{\D}) \otimes m  \ar[u]^{\sigma_{\unit_{\D}, m}} & \\ 
F(\unit_{\D}) \otimes m \otimes F(\unit_{\D}) \ar[uu]^{\sigma_{\unit_{\D}, m}, 1} \ar@/^0.5pc/[ru]^(0.4){r_{F(\unit_{\D}) \otimes m}} \ar[ru]_{1,r_m} & & F(\unit_{\D}) \otimes F(\unit_{\D}) \otimes m  \ar[ll]^{1, \sigma_{\unit_{\D}, m}} \ar@/_0.5pc/[lu]_(0.4){r_{F(\unit_{\D})}, 1} \ar[lu]^(0.6){1, l_m}
}\]
The outward diagram commutes by the diagram (\ref{center2}).
The right-upper square commutes by the naturality of the central structure $\sigma_{d, m}: F(d) \otimes m \to m \otimes F(d)$, $d \in D, m \in \M$.
The left square commutes by the naturality of the unit isomorphism $r_m: m \otimes \unit_{\M} \simeq m$, $m \in \M$.
The left-upper and right-down triangles commute by the monoidal functor $F$.
Three parallel arrows equal by the triangle diagrams of the monoidal category $\M$.
Then the bottom triangle commutes. Since $F(\unit_{\D}) \otimes - = \unit_{\M} \otimes - \simeq \id_{\M}$ is the natural isomorphism, the right triangle of (\ref{diagram-c}) commutes.
\end{proof}

Let $A$ be a separable algebra in a multifusion category $\C$ over $\E$. We use ${}_A\!\C$ (or $\C_A$, ${}_A\!\C_A$) to denote the category of left $A$-modules (or right $A$-modules, $A$-bimodules) in $\C$.

\begin{prop}
\label{E-free-C}
Let $\C$ be a multifusion category over $\E$ and $A$ a separable algebra in $\C$. Then the diagram 
\[ \xymatrix{
T_{\C}(e) \otimes x \otimes_A y^R \ar[r]^{c_{e, x \otimes_A y^R}} \ar[d]_{c_{e, x}, 1} & x \otimes_A y^R \otimes T_{\C}(e)  \\
x \otimes T_{\C}(e) \otimes_A y^R \ar[r]_{h} & x \otimes_A T_{\C}(e) \otimes y^R \ar[u]_{1, c_{e,y^R}}
} \]
commutes for $e \in \E$, $x, y \in \C_A$, where $c$ is the central structure of the central functor $T_{\C}: \E \to \C$.
\end{prop}
\begin{proof}
The functor $y \mapsto y^R$ defines an equivalence of right $\C$-modules $(\C_A)^{\op|L} \simeq {}_A\!{\C}$.
For $x \in \C_A$, we use $p_x$ to denote the right $A$-action on $x$. For $y^R \in {}_A\!{\C}$, we use $q_{y^R}$ to denote the left $A$-action on $y^R$.
Obviously, $T_{\C}(e) \otimes x$ belongs to $\C_A$ and $y^R \otimes T_{\C}(e)$ belongs to ${}_A\!{\C}$. 
The right $A$-action on $x \otimes T_{\C}(e)$ is induced by $x \otimes T_{\C}(e) \otimes A \xrightarrow{1,c_{e,A}} x \otimes A \otimes T_{\C}(e) \xrightarrow{p_x, 1} x \otimes T_{\C}(e)$.
The left $A$-action on $T_{\C}(e) \otimes y^R$ is induced by $A \otimes T_{\C}(e) \otimes y^R \xrightarrow{c^{-1}_{e, A},1} T_{\C}(e) \otimes A \otimes y^R \xrightarrow{1,q_{y^R}} T_{\C}(e) \otimes y^R$.
It is routine to check that $c_{e, x}$ is a morphism in $\C_A$ and $c_{e, y^R}$ is a morphism in ${}_A\!{\C}$.

The morphism $c_{e, x \otimes_A y^R}$ is induced by
\[ \xymatrix{
T_{\C}(e) \otimes x \otimes A \otimes y^R \ar@<.5ex>[r]^{1,p_x,1} \ar@<-.5ex>[r]_{1,1,q_{y^R}} \ar[d]_{c_{e, x \otimes A \otimes y^R}} & T_{\C}(e) \otimes x \otimes y^R \ar[d]^{c_{e, x \otimes y^R}} \ar[r] & T_{\C}(e) \otimes x \otimes_A y^R \ar[d]^{c_{e, x \otimes_A y^R}} \\
x \otimes A \otimes y^R \otimes T_{\C}(e) \ar@<.5ex>[r]^{p_x,1,1} \ar@<-.5ex>[r]_{1,q_{y^R},1} & x \otimes y^R \otimes T_{\C}(e) \ar[r] & x \otimes_A y^R \otimes T_{\C}(e) 
} \]
The composition $(1_x \otimes_A c_{e,y^R}) \circ h \circ (c_{e,x} \otimes_A 1_{y^R})$ is induced by
\[ \xymatrix{
T_{\C}(e) \otimes x \otimes A \otimes y^R \ar@<.5ex>[r]^{1,p_x,1} \ar@<-.5ex>[r]_{1,1,q_{y^R}} \ar@<-5ex>@/_3pc/[ddd]_{c_{e, x \otimes A \otimes y^R}} \ar[d]_{c_{e,x},1,1} & T_{\C}(e) \otimes x \otimes y^R \ar[d]^{c_{e,x},1} \ar[r]  & T_{\C}(e) \otimes x \otimes_A y^R \ar[d]^{c_{e,x},1} \\
x \otimes T_{\C}(e) \otimes A \otimes y^R \ar@<.5ex>[r]^{p_{x \otimes T_{\C}(e)}} \ar@<-.5ex>[r]_{1,1,q_{y^R}} \ar[d]_{1,c_{e,A},1} & x \otimes T_{\C}(e) \otimes y^R \ar[d]^{1} \ar[r] & x \otimes T_{\C}(e) \otimes_A y^R \ar[d]^{h} \\
x \otimes A \otimes T_{\C}(e) \otimes y^R \ar@<.5ex>[r]^{p_x,1,1} \ar@<-.5ex>[r]_{1,q_{T_{\C}(e) \otimes y^R}} \ar[d]_{1,1,c_{e,y^R}} & x \otimes T_{\C}(e) \otimes y^R \ar[r] \ar[d]^{1,c_{e,y^R}} & x \otimes_A T_{\C}(e) \otimes y^R \ar[d]^{1,c_{e,y^R}} \\
x \otimes A \otimes y^R \otimes T_{\C}(e) \ar@<.5ex>[r]^{p_x,1,1} \ar@<-.5ex>[r]_{1,q_{y^R},1} & x \otimes y^R \otimes T_{\C}(e) \ar[r] & x \otimes_A y^R \otimes T_{\C}(e) 
} \]
Since $c_{e, x \otimes y^R} = (1_x \otimes c_{e,y^R}) \circ (c_{e,x} \otimes 1_{y^R})$, the composition $(1_x \otimes_A c_{e,y^R}) \circ h \circ (c_{e,x} \otimes_A 1_{y^R})$ equals to $c_{e, x \otimes_A y^R}$ by the universal property of coequalizers.
\end{proof}

\begin{prop}
\label{A-fun-bim-eq}
Let $\C$ be a multifusion category over $\E$ and $A$ a separable algebra in $\C$.
There is an equivalence $\Fun_{\C}(\C_A, \C_A) \simeq ({{}_A\!{\C}_A})^{\rev}$ of multifusion categories over $\E$.
\end{prop}
\begin{proof}
By \cite[Prop.\,7.11.1]{Etingof}, 
the functor $\Phi: ({{}_A\!{\C}_A})^{\rev} \to \Fun_{\C}(\C_A, \C_A)$ is defined as $x \mapsto - \otimes_A x$ and the inverse of $\Phi$ is defined as $f \mapsto f(A)$.
The monoidal structure on $\Phi$ is defined as
\[ \Phi(x \otimes^{\rev}_A y) = - \otimes_A (y \otimes_A x) \simeq (- \otimes_A y) \otimes_A x = \Phi(x) \circ \Phi(y) \]
for $x, y \in  ({{}_A\!{\C}_A})^{\rev}$. 
Recall the central structures on the functors $I: \E \to ({}_A\!{\C}_A)^{\rev}$ and $\hat{T}: \E \to \Fun_{\C}(\C_A, \C_A)$ in Expl.\,\ref{Ex-A-bimodule} and Expl.\,\ref{fun-C-E} respectively.
The structure of monoidal functor over $\E$ on $\Phi$ is induced by
\[ \Phi(I(e)) = \Phi(T_{\C}(e) \otimes^{\rev} A) = - \otimes_A (A \otimes T_{\C}(e)) \cong - \otimes T_{\C}(e) \xrightarrow{c_{e, -}^{-1}} T_{\C}(e) \otimes - = \hat{T}^e \]
for $e \in \E$, where $c$ is the central structure of the functor $T_{\C}: \E \to \C$. 
Next we want to check that $\Phi$ is a monoidal functor over $\E$. Consider the diagram for $e \in \E, x \in  ({{}_A\!{\C}_A})^{\rev}$:
\[ \xymatrix{
\Phi(I(e) \otimes^{\rev}_A x) \ar[r] \ar[d]_{\Phi(\sigma_{e, x})} & \Phi(I(e)) \circ \Phi(x) \ar[r] & \hat{T}^e \circ \Phi(x) \ar[d]^{\bar{\sigma}_{e, \Phi(x)}} \\
\Phi(x \otimes_A^{\rev} I(e)) \ar[r] & \Phi(x) \circ \Phi(I(e)) \ar[r] & \Phi(x) \circ \hat{T}^e
} \]
The central structure $\sigma_{e,x}$ is induced by $x \otimes_A A \otimes T_{\C}(e) \xrightarrow{c^{-1}_{e, x \otimes_A A}} T_{\C}(e) \otimes x \otimes_A A \cong T_{\C}(e) \otimes A \otimes_A x \xrightarrow{c_{e, A}, 1} A \otimes T_{\C}(e) \otimes_A x$. The central structure $\bar{\sigma}_{e, \Phi(x)}$ is induced by $T_{\C}(e) \otimes (- \otimes_A x) \simeq (T_{\C}(e) \otimes -) \otimes_A x$.
The commutativity of the above diagram is due to the commutativity of the following diagram
\[ \xymatrix{
- \otimes_A T_{\C}(e) \otimes x \otimes_A A \ar[d]_{\cong} \ar[r]^{1, c_{e ,x \otimes_A A}}& - \otimes_A x \otimes_A A \otimes T_{\C}(e) \ar[r]^{c^{-1}_{e, - \otimes_A x \otimes_A A}} \ar@{.>}[d]_{\Phi(\sigma_{e,x})} & T_{\C}(e) \otimes - \otimes_A x \otimes_A A \ar[d]^{\cong} \\
- \otimes_A T_{\C}(e) \otimes A \otimes_A x \ar[r]_{1, c_{e, A}, 1} &- \otimes_A A \otimes T_{\C}(e) \otimes_A x \ar[r]_{c^{-1}_{e, - \otimes_A A}, 1} & T_{\C}(e) \otimes - \otimes_A A \otimes_A x
} \]
The upper horizontal composition $c^{-1}_{e, -\otimes_A x} \circ (1 \otimes_A c_{e, x})$ is induced by
\[ \xymatrix{
- \otimes A \otimes T_{\C}(e) \otimes x \ar@<.5ex>[r]^{p_-,1,1} \ar@<-.5ex>[r]_{1,q_{T_{\C}(e) \otimes x}} \ar[d]_{1,1,c_{e,x}} \ar@<-5ex>@/_2pc/[dd]_{c^{-1}_{e, -\otimes A},1} & - \otimes T_{\C}(e) \otimes x \ar[r] \ar[d]^{1, c_{e,x}} \ar@<5ex>@/^2pc/@{.>}[dd]|(0.3){c^{-1}_{e,-},1}& - \otimes_A T_{\C}(e) \otimes x \ar[d]^{1, c_{e,x}} \\
- \otimes A \otimes x \otimes T_{\C}(e) \ar@<.5ex>[r]^{p_-,1,1} \ar@<-.5ex>[r]_{1,q_x,1} \ar[d]_{c^{-1}_{e, -\otimes A \otimes x}} & - \otimes x \otimes T_{\C}(e) \ar[r] \ar[d]^{c^{-1}_{e, - \otimes x}}& - \otimes_A x \otimes T_{\C}(e) \ar[d]^{c^{-1}_{e ,- \otimes_A x}} \\
T_{\C}(e) \otimes - \otimes A \otimes x \ar@<.5ex>[r]^{1,p_-,1} \ar@<-.5ex>[r]_{1,1,q_x} & T_{\C}(e) \otimes - \otimes x \ar[r] & T_{\C}(e) \otimes - \otimes_A x
} \]
Here $(-, p_-)$ and $(A, m)$ belong to $\C_A$ and $(x, p_x, q_x)$ belong to ${{}_A\!{\C}_A}$. $q_{T_{\C}(e) \otimes x}$ is defined as $A \otimes T_{\C}(e) \otimes x \xrightarrow{c_{e, A}^{-1}, 1} T_{\C}(e) \otimes A \otimes x \xrightarrow{1,q_x} T_{\C}(e) \otimes x$. 
The lower horizontal composition $c^{-1}_{e, - \otimes_A A}  \circ (1 \otimes_A c_{e, A})$ is induced by
\[ \xymatrix{
- \otimes A \otimes T_{\C}(e) \otimes A \ar@<.5ex>[r]^{p_-, 1, 1} \ar@<-.5ex>[r]_{1,q_{T_{\C}(e) \otimes A}} \ar[d]_{1,1,c_{e,A}} \ar@<-5ex>@/_2pc/[dd]_{c^{-1}_{e, - \otimes A},1} & - \otimes T_{\C}(e) \otimes A \ar[r] \ar[d]^{1, c_{e, A}} \ar@<5ex>@/^2pc/@{.>}[dd]|(0.3){c^{-1}_{e,-},1} & - \otimes_A T_{\C}(e) \otimes A \ar[d]^{1, c_{e, A}} \\
- \otimes A \otimes A \otimes T_{\C}(e) \ar@<.5ex>[r]^{p_-,1,1} \ar@<-.5ex>[r]_{1,m,1} \ar[d]_{c^{-1}_{e, - \otimes A \otimes A}} & - \otimes A \otimes T_{\C}(e) \ar[r] \ar[d]^{c^{-1}_{e, -\otimes A}} & - \otimes_A A \otimes T_{\C}(e) \ar[d]^{c^{-1}_{e, - \otimes_A A}} \\
T_{\C}(e) \otimes - \otimes A \otimes A \ar@<.5ex>[r]^{1, p_-,1} \ar@<-.5ex>[r]_{1,1,m} & T_{\C}(e) \otimes - \otimes A \ar[r] & T_{\C}(e) \otimes - \otimes_A A
} \]
Since $x \otimes_A A \cong x \cong A \otimes_A x$, the compositions $c^{-1}_{e, -\otimes_A x \otimes_A A} \circ (1 \otimes_A c_{e, x \otimes_A A})$ and $(c^{-1}_{e, - \otimes_A A} \otimes_A 1_x) \circ (1 \otimes_A c_{e,A} \otimes_A 1_x)$ are  equal by the universal property of cokernels. 
\end{proof}

\begin{prop}
\label{prop-ACB-Fun}
Let $\C$ be a multifusion category over $\E$ and $A, B$ be separable algebras in $\C$. 
\begin{itemize}
\item[(1)] There is an equivalence ${}_A\!\C \boxtimes_{\C} \C_B \xrightarrow{\simeq} {}_A\!\C_B$, $x \boxtimes_{\C} y \mapsto x \otimes y$ in $\BMod_{\E|\E}(\cat)$.
\item[(2)] There is an equivalence $\Fun_{\C}(\C_A, \C_B) \xrightarrow{\simeq} {}_A\!\C_B$, $f \mapsto f(A)$ in $\BMod_{\E|\E}(\cat)$, whose inverse is defined as $x \mapsto - \otimes_A x$.
\end{itemize}
\end{prop} 
\begin{proof}
(1) The functor $\Phi: {}_A\!\C \boxtimes_{\C} \C_B \to {}_A\!\C_B$, $x \boxtimes_{\C} y \mapsto x \otimes y$ is an equivalence by \cite[Thm.\,2.2.3]{Liang}. 
Recall the $\E$-$\E$ bimodule structure on ${}_A\!\C_B$ and ${}_A\!\C \boxtimes_{\C} \C_B$ by Expl.\,\ref{expl-ABC-E} and Expl.\,\ref{22} respectively.
The left $\E$-module structure on $\Phi$ is defined as
\[ \Phi(e \odot (x \boxtimes_{\C} y)) = \Phi((T_{\C}(e) \otimes x) \boxtimes_{\C} y) = (T_{\C}(e) \otimes x) \otimes y \to T_{\C}(e) \otimes (x \otimes y) = e \odot \Phi(x \boxtimes_{\C} y) \]
for $e \in \E$, $x \boxtimes_{\C} y \in {}_A\!\C \boxtimes_{\C} \C_B$.
The right $\E$-module structure on $\Phi$ is defined as 
\[ \Phi((x \boxtimes_{\C} y) \odot e) = \Phi(x \boxtimes_{\C} (y \otimes T_{\C}(e))) = x \otimes (y \otimes T_{\C}(e)) \to (x \otimes y) \otimes T_{\C}(e) = \Phi(x \boxtimes_{\C} y) \odot e \]  
Check that $\Phi$ satisfies the diagram (\ref{CD-bim-fun}).
\[ \xymatrix@=4ex{
\Phi((T_{\C}(e) \otimes x) \boxtimes_{\C} y) \ar[r]^{c_{e,x},1} \ar[d] & \Phi((x \otimes T_{\C}(e)) \boxtimes_{\C} y) \ar[r]^{b_{x, T_{\C}(e), y}} & \Phi(x \boxtimes_{\C}(T_{\C}(e) \otimes y)) \ar[r]^{1, c_{e,y}} & \Phi(x \boxtimes_{\C}(y \otimes T_{\C}(e))) \ar[d] \\
T_{\C}(e) \otimes \Phi(x \boxtimes_{\C} y) \ar[rrr]_{c_{e, x \otimes y}} & && \Phi(x \boxtimes_{\C} y) \otimes T_{\C}(e)
} \]
Here $c$ is the central structure of the central functor $T_{\C}: \E \to \C$.
The above diagram commutes by the diagram (\ref{center1}).

(2) Since $\C_A$ and $\C_B$ belongs to $\BMod_{\C|\E}(\cat)$, the category $\Fun_{\C}(\C_A, \C_B)$ belongs to $\BMod_{\E|\E}(\cat)$.
The $\E$-$\E$ bimodule structure on $\Fun_{\C}(\C_A, \C_B)$ in Expl.\,\ref{33} is defined as 
\[ (e \odot f \odot \tilde{e})(-) \coloneqq f(- \otimes T_{\C}(e)) \otimes T_{\C}(\tilde{e}) \]
for $e, \tilde{e} \in \E$, $f \in \Fun_{\C}(\C_A, \C_B)$ and $- \in \C_A$.

The functor $\Psi: {}_A\!\C_B \to \Fun_{\C}(\C_A, \C_B)$, $x \mapsto \Psi^x \coloneqq - \otimes_A x$ is an equivalence by \cite[Cor.\,2.2.6]{Liang}. The left $\E$-module structure on $\Psi$ is defined as
\[ \Psi^{e \odot x} = - \otimes_A (T_{\C}(e) \otimes x) \cong (- \otimes T_{\C}(e)) \otimes_A x = \Psi^x(- \otimes T_{\C}(e)) = e \odot \Psi^x  \]
The right $\E$-module structure on $\Psi^x$ is defined as
\[ \Psi^{x \odot e} = - \otimes_A (x \otimes T_{\C}(e)) \cong (- \otimes_A x) \otimes T_{\C}(e) = \Psi^x \odot e \]
Recall the monoidal natural isomorphism $(v_e)_{\Psi^x}: e \odot \Psi^x \Rightarrow \Psi^x \odot e$ in Expl.\,\ref{33}:
\[ (e \odot \Psi^x)(-) = \Psi^x(- \otimes T_{\C}(e)) \xrightarrow{c^{-1}_{e,-}} \Psi^x(T_{\C}(e) \otimes -) \xrightarrow{s^{\Psi^x}} T_{\C}(e) \otimes \Psi^x(-) \xrightarrow{c_{e, \Psi^x(-)}} \Psi^x(-) \otimes T_{\C}(e) = (\Psi^x \odot e)(-) \]
Check $\Psi$ satisfies the diagram (\ref{CD-bim-fun}).
\[ \xymatrix@=4ex{
\Psi^{e \odot x} = - \otimes_A (T_{\C}(e) \otimes x) \ar[dd]_{1,c_{e,x}} \ar[r] & (- \otimes T_{\C}(e)) \otimes_A x = \Psi^x(- \otimes T_{\C}(e)) \ar[d]^{c^{-1}_{e,-},1} \\
& T_{\C}(e) \otimes - \otimes_A x \ar[d]^{c_{e, - \otimes_A x}} \\
\Psi^{x \odot e} = - \otimes_A (x \otimes T_{\C}(e)) \ar[r] & (- \otimes_A x) \otimes T_{\C}(e) =\Psi^x(-) \otimes T_{\C}(e)
} \]
The above diagram commutes by Prop.\,\ref{E-free-C}.
\end{proof}

\begin{lem}
\label{equivalencethm}
 Let $M$ and $N$ be separable algebras in $\E$.
The functor $\Phi: {{}_M\!{\E}} \boxtimes_{\E} \E_N\to {{}_M\!{\E}_N}$, $x \boxtimes_{\E} y \mapsto x \otimes y$ is an equivalence of categories.
The inverse of $\Phi$ is defined as $z \mapsto \colim\big((M \otimes M) \boe z \rightrightarrows M \boe z \big)$ for any $z \in {{}_M\!{\E}_N}$.
\end{lem}

\begin{proof}
 The inverse of $\Phi$ is denoted by $\Psi$.
\begin{align*}
\Psi \circ \Phi (x \boe y) & = \Psi(x \otimes y) = \colim\big(\boe(M \otimes M, x \otimes y) \rightrightarrows \boe (M, x \otimes y)\big)\\
& \simeq \colim\big( \boe (M \otimes M \otimes x, y) \rightrightarrows \boe (M \otimes x, y)\big) \simeq \boe(M \otimes_M x, y) \simeq x \boe y
\end{align*}
The first equivalence is due to the balanced $\E$-module functor $\boe$. The second equivalence holds because the functor $\boe$ preserves colimits.
\begin{align*}
\Phi \circ \Psi (z) & = \Phi \big(\colim\big((M \otimes M) \boe z \rightrightarrows  M \boe z \big)\big) \simeq \colim\big( \Phi ((M \otimes M) \boe z) \rightrightarrows \Phi (M \boe z)\big) \\
&\simeq \colim\big((M \otimes M) \otimes z \rightrightarrows M \otimes z\big) \simeq  M \otimes_M z \simeq z
\end{align*}
The first equivalence holds because $\Phi$ preserves colimits.
\end{proof}

\begin{lem}
\label{lem-AEA}
Let $A$ be a separable algebra in $\E$.
There is an equivalence ${}_A\!\E \simeq \E_A$ of right $\E$-module categories.
\end{lem}

\begin{proof}
We define a functor $F: {}_A\!\E \to \E_A$, $(x, q_x: A \otimes x \to x) \mapsto (x, p_x: x \otimes A \xrightarrow{r_{x, A}} A \otimes x \xrightarrow{q_x} x)$ and a functor $G: \E_A \to {}_A\!\E$, $(y, p_y: y \otimes A \to y) \mapsto (y, q_y: A \otimes y \xrightarrow{r_{A, y}} y \otimes A \xrightarrow{p_y} y)$, 
where $r$ is the braiding of $\E$.
Since $r_{x, y} \circ r_{y, x} = \id_{x \otimes y}$ for all $x, y \in \E$, then $F \circ G = \id$ and $G \circ F = \id$.

The right $\E$-action on $\E_A$ is defined as
$(y, p_y) \otimes e = (y \otimes e, p_{y \otimes e}: y \otimes e \otimes A \xrightarrow{1, r_{e, A}} y \otimes A \otimes e \xrightarrow{p_y, 1} y \otimes e)$.
We have $F((x, q_x) \otimes e) = F(x \otimes e, q_{x \otimes e}: A \otimes x \otimes e \xrightarrow{q_x, 1} x \otimes e) = (x \otimes e, p_{x \otimes e}: x \otimes e \otimes A \xrightarrow{r_{x \otimes e, A}} A \otimes x \otimes e \xrightarrow{q_x, 1} x \otimes e)$
and $F(x, q_x) \otimes e = (x, p_x) \otimes e = (x \otimes e, p_{x \otimes e}: x \otimes e \otimes A \xrightarrow{1, r_{e,A}} x \otimes A \otimes e \xrightarrow{r_{x,A},1} A \otimes x \otimes e \xrightarrow{q_x, 1} x \otimes e)$.
Then the right $\E$-module structure on $F$ is the identity natural isomorphism 
$F((x, q_x) \otimes e) = F(x, q_x) \otimes e$.
\end{proof}

\begin{lem}
\label{thm-internal-inverse}
Let $\C$ and $\M$ be pivotal fusion categories and $\M$ a left $\C$-module in $\fcat$. 
There are isomorphisms $[x, y]^R_{\C} \cong [y, x]_{\C} \cong [x, y]^L_{\C} $
for $x, y \in \M$.
\end{lem}
\begin{proof}
Since $\M$ is a pivotal fusion category,
there is a one-to-one correspondence between traces on $\M$ and natural isomorphisms 
\[\eta^{\M}_{x,y}: \homm_{\M}(x, y) \to \homm_{\M}(y, x)^*\]
 for $x, y \in \M$ by \cite[Prop.\,4.1]{S}.
Here both $\homm_{\M}(-, -)$ and $\homm(-,-)^*$ are functors from $\M^{\op} \times \M \to \vect$.
For $c \in \C$, we have composed natural isomorphisms
\begin{align*} 
\homm_{\C}(c, [x, y]_{\C}) & \simeq \homm_{\M}(c \odot x, y) \xrightarrow{\eta^{\M}} \homm_{\M}(y, c \odot x)^* \simeq \homm_{\M}(c^L \odot y, x)^* \\
 &\simeq \homm_{\C}(c^L, [y,x]_{\C})^* \simeq \homm_{\C}([y,x]^R_{\C}, c)^* \xrightarrow{(\eta^{\C})^{-1}} \homm_{\C}(c, [y, x]^R_{\C}) ,
 \end{align*}
\begin{align*}
\homm_{\C}(c, [x, y]_{\C}) \simeq \homm_{\M}(c \odot x, y) \simeq \homm_{\M}(x, c^R \odot y) \xrightarrow{\eta^{\M}} \homm_{\M}(c^R \odot y, x)^* \\
\simeq \homm_{\C}(c^R, [y, x]_{\C})^* \xrightarrow{(\eta^{\C})^{-1}} \homm_{\C}([y, x]_{\C}, c^R) \simeq \homm_{\C}(c, [y, x]^L_{\C})
\end{align*}
By Yoneda lemma, we obtain $[x, y]^R_{\C} \cong [y, x]_{\C} \cong [x, y]^L_{\C}$.
\end{proof}

\subsection{The monoidal 2-category $\cat$}
   \label{m-cate1}
  For objects $\M, \N$ in a 2-category $\BB$, the hom category $\BB(\M, \N)$ denotes the category of 1-morphisms from $\M$ to $\N$ in $\BB$ and 2-morphisms in $\BB$. For 1-morphisms $f, g \in \BB(\M, \N)$, the set $\BB(\M, \N)(f, g)$ denotes the set of all 2-morphisms in $\BB$ with domain $f$ and codomain $g$.

    \begin{defn}
    The product 2-category $\cat \times \cat$ is the 2-category defined by the following data:
    \begin{itemize}
    \item The objects are pairs $(\A, \B)$ for $\A, \B \in \cat$.
    \item For objects $(\A, \B)$, $(\C, \D) \in \cat \times \cat$, a 1-morphism from $(\A, \B)$ to $(\C, \D)$ is a pair $(f, g)$ where $f: \A \rightarrow \C$ and $g: \B \rightarrow \D$ are 1-morphisms in $\cat$. 
    \item The identity 1-morphism of an object $(\A, \B)$ is $1_{(\A, \B)} \coloneqq (1_{\A}, 1_{\B})$.
    \item For 1-morphisms $(f, g), (p, q) \in (\cat \times \cat)((\A, \B), (\C, \D))$, a 2-morphism from $(f, g)$ to $(p, q)$ is a pair $(\alpha, \beta)$ where $\alpha: f \Rightarrow p$ and $\beta: g \Rightarrow q$ are 2-morphisms in $\cat$.   
    \item For 1-morphisms $(f, g), (p, q), (m, n) \in (\cat \times \cat)((\A, \B), (\C, \D))$, and 2-morphisms $(\alpha, \beta) \in (\cat \times \cat)((\A, \B), (\C, \D))((f, g), (p, q))$, and $(\gamma, \delta) \in (\cat \times \cat)((\A, \B), (\C, \D))((p, q), (m, n))$, the vertical composition is $(\gamma, \delta) \circ (\alpha, \beta) \coloneqq (\gamma \circ \alpha, \delta \circ \beta)$.
    \item For 1-morphisms $(f, g) \in (\cat \times \cat)((\A, \B), (\C, \D))$, $(p, q) \in (\cat \times \cat)((\C, \D), (\M, \N))$, the horizontal composition of 1-morphisms is $(p, q) \circ (f, g) \coloneqq (p \circ f, q \circ g)$.
    \item For 1-morphisms $(f, g), (f', g') \in (\cat \times \cat)((\A, \B), (\C, \D))$, and $(p, q), (p', q') \in (\cat \times \cat)((\C, \D), (\M, \N))$, and 2-morphisms $(\alpha, \beta) \in (\cat \times \cat)((\A, \B), (\C, \D))((f, g), (f', g'))$, and $(\gamma, \delta) \in (\cat \times \cat)((\C, \D), (\M, \N))((p, q), (p', q'))$, the horizontal composition of 2-morphisms is $(\gamma, \delta) * (\alpha, \beta) \coloneqq (\gamma * \alpha, \delta * \beta)$.
    \end{itemize} 
    It is routine to check that the above data satisfy the axioms (i)-(vi) of  \cite[Prop.\,2.3.4]{yau}.
    \end{defn}

    Next, we define a pseudo-functor $\boxtimes_{\E}: \cat \times \cat \rightarrow \cat$ as follows.
    \begin{itemize}
    \item For each object $(\A, \B) \in \cat \times \cat$, an object $\A \boxtimes_{\E} \B$ in $\cat$ exists (unique up to equivalence).     
    \item  For a 1-morphism $(f, g) \in (\cat \times \cat)((\A, \B), (\C, \D))$, 
    a 1-morphism $f \boxtimes_{\E} g: \A \boxtimes_{\E} \B \rightarrow \C \boxtimes_{\E} \D$ in $\cat$ is induced by the universal property of the tensor product $\boxtimes_{\E}$:
    \[
     \xymatrix{
    \A \times \B \ar[r]^{\boxtimes_{\E}} \ar[d]_{f, g} & \A \boxtimes_{\E} \B \ar[d]^{\exists ! f \boxtimes_{\E} g}  \\
    \C \times \D \ar[r]_{\boxtimes_{\E}} & \C \boxtimes_{\E} \D \ultwocell<>{t_{fg}\;\;}
    } 
    \]
    Notice that for all $x \in \A, e \in \E, y \in \B$, the balanced $\E$-module structure on the functor $\boe \circ (f \times g)$ is induced by
    \[  f(x \odot e) \boe g(y) \xrightarrow{(s^r_f)^{-1} \boe 1} (f(x) \odot e) \boe g(y) \xrightarrow{b^{\C\D}_{f(x), e, g(y)}} f(x) \boe (e \odot g(y)) \xrightarrow{1 \boe s^l_g}  f(x) \boe g(e \odot y) \]
        where $(g, s^l_g): \B \to \D$ is the left $\E$-module functor, $(f, s^r_f): \A \to \C$ is the right $\E$-module functor, and the natural isomorphism $b^{\C\D}$ is the balanced $\E$-module structure on the functor $\boe: \C \times \D \rightarrow \C \boe \D$.     
    
    For a 2-morphism $(\alpha, \beta): (f, g) \Rightarrow (p, q)$ in $(\cat \times \cat)((\A, \B), (\C, \D))$, a 2-morphism $\alpha \boxtimes_{\E} \beta: f \boxtimes_{\E} g \Rightarrow p \boxtimes_{\E} q$ in $\cat$ is defined by the universal property of $\boe$:
    \[ \vcenter{\xymatrix{
  \A \times \B \ar[r]^{\boxtimes_{\E}} \ar[d]_{f,g} \dtwocell<>{^<-3>t_{fg}} & \A \boxtimes_{\E} \B \dtwocell<7>_{f \boxtimes_{\E} g\quad}^{\quad p \boxtimes_{\E} q}{^\exists ! \alpha \boxtimes_{\E} \beta} \\
  \C \times \D \ar[r]_{\boxtimes_{\E}} & \C \boxtimes_{\E} \D
  }}
  =
  \vcenter{\xymatrix{
  \A \times \B \ar[r]^{\boxtimes_{\E}} \dtwocell<5>_{f,g}^{p,q}{^\alpha, \beta} \dtwocell<>{^<-7>t_{pq}} & \A \boxtimes_{\E} \B \ar[d]^{p \boxtimes_{\E} q} \\
  \C \times \D \ar[r]_{\boxtimes_{\E}} & \C \boxtimes_{\E} \D
  }}   \]
It is routine to check that $\boe: (\cat \times \cat)((\A, \B), (\C, \D)) \rightarrow \cat(\A \boe \B, \C \boe \D)$ is a local functor. 
That is, for 2-morphisms $(\alpha, \beta): (f, g) \Rightarrow (p, q)$ and $(\delta, \tau): (p, q) \Rightarrow (m, n)$ in $(\cat \times \cat)((\A, \B), (\C, \D))$, The equations $(\delta \circ \alpha) \boe (\tau \circ \beta) = (\delta \boe \tau) \circ (\alpha \boe \beta)$ and $1_f \boe 1_g = 1_{f \boe g}$ hold.
    
    \item For all 1-morphisms $f \boxtimes_{\E} g : \A \boxtimes_{\E} \B \rightarrow \C \boxtimes_{\E} \D$, $p \boxtimes_{\E} q: \C \boxtimes_{\E} \D \rightarrow \M \boxtimes_{\E} \N$ in $\cat$, the lax functoriality constraint $(p \boxtimes_{\E} q) \circ (f \boxtimes_{\E} g) \simeq^{t_{fg}^{pq}} (p \circ f) \boxtimes_{\E} (q \circ g)$ is defined by the universal property of $\boe$:
 \[ \vcenter{\xymatrix{
 & \A \times \B \ar[r]^{\boxtimes_{\E}} \ar[dl]_{f, g} \dtwocell<>{^t_{fg}} & \A \boxtimes_{\E} \B \ar[dl]^(0.4){\exists ! f \boxtimes_{\E} g} \ar[dd]^{\exists ! (p \circ f) \boxtimes_{\E} (q \circ g)} \ddtwocell<>{^<4>\exists ! t_{fg}^{pq}} \\
 \C \times \D \ar[r]^{\boxtimes_{\E}} \ar[dr]_{p, q} & \C \boxtimes_{\E} \D \ar[dr]^(0.6){\exists ! p \boxtimes_{\E} q} \dtwocell<>{^t_{pq}} & \\
 & \M \times \N \ar[r]_{\boxtimes_{\E}} & \M \boxtimes_{\E} \N
 }} 
 =
 \vcenter{\xymatrix{
 & \A \times \B \ar[r]^{\boxtimes_{\E}} \ar[dl]_{f, g} \ar[dd]_{p \circ f, q \circ g} \ddtwocell<>{^<-6>t_{pf, qg}} & \A \boxtimes_{\E} \B  \ar[dd]^{\exists ! (p \circ f) \boxtimes_{\E} (q \circ g)} \\
 \C \times \D  \ar[dr]_{p, q} &   & \\
 & \M \times \N \ar[r]_{\boxtimes_{\E}} & \M \boxtimes_{\E} \N
 }}
  \]
  where the identity 2-morphism is always abbreviated.
  
  \item For 1-morphisms $1_{\A} \boxtimes_{\E} 1_{\B}: \A \boxtimes_{\E} \B \rightarrow \A \boxtimes_{\E} \B$ in $\cat$, the lax unity constraint $1_{\A} \boxtimes_{\E} 1_{\B} \simeq^{t_{\A\B}} 1_{\A \boxtimes_{\E} \B}$ is defined by the universal property of $\boe$:
    \[ \vcenter{\xymatrix{
  \A \times \B \ar[r]^{\boxtimes_{\E}} \ar@/_1pc/[d]_{1_{\A}, 1_{\B}} \dtwocell<>{^<-1>t_{1_{\A}, 1_{\B}}} & \A \boxtimes_{\E} \B \dtwocell<5>_{1_{\A} \boxtimes_{\E} 1_{\B}\qquad}^{\qquad1_{\A \boxtimes_{\E} \B}}{^\exists !t_{\A\B}} \\
  \A \times \B \ar[r]_{\boxtimes_{\E}} & \A \boxtimes_{\E} \B
  }}
  =
  \vcenter{\xymatrix{
  \A \times \B \ar[r]^{\boxtimes_{\E}} \ar@/_0.6pc/[d]_{1_{\A}, 1_{\B}} \ar@/^0.6pc/[d]^{1_{\A \times \B}} \dtwocell<>{^<-8>\id} & \A \boxtimes_{\E} \B \ar[d]^{1_{\A \boxtimes_{\E} \B}} \\
  \A \times \B \ar[r]_{\boxtimes_{\E}} & \A \boxtimes_{\E} \B
  }}   \]
  where we choose the identity 2-morphism $\id: \boe \circ 1_{\A \times \B} \Rightarrow 1_{\A \boe \B} \circ \boe$ for convenience.
    \end{itemize}   
    It is routine to check that the above data satisfy the lax associativity, the lax left and right unity of \cite[(4.1.3),(4.1.4)]{yau}.
    
           \begin{rem}
         The left (or right) $\E$-module structure on $\A \boe \B$ is induced by
        \[ \vcenter{\hbox{\xymatrix{
    \E \times \A \times \B \ar[r]^{1, \boe} \ar[d]_{\odot, 1} \drtwocell<>{^t^l_{\A\B}\quad} & \E \times \A \boe \B \ar[d]^{\exists ! \odot} \\
    \A \times \B \ar[r]_{\boe} & \A \boe \B
    }}} 
    \qquad \quad
    \vcenter{\hbox{\xymatrix{
    \A \times \B \times \E \ar[r]^{\boe, 1} \ar[d]_{1, \odot} \drtwocell<>{^t^r_{\A\B} \quad} & \A \boe \B \times \E \ar[d]^{\exists ! \odot} \\    
    \A \times \B \ar[r]_{\boe} & \A \boe \B
    }}}
    \]   
    
    \end{rem}

    The n-fold product 2-category $\cat \times \dots \times \cat$ is written as $(\cat)^n$ such that $\cat$ has a set of objects.
    The 2-category $\ccat^{\ps}((\cat)^n, \cat)$ contains pseudofunctors $(\cat)^n \rightarrow \cat$ as objects, strong transformations between such pseudofunctors as 1-morphisms, and modifications between such strong transformations as 2-morphisms.    
    
    \begin{lem}
    We claim that $\cat$ is a monoidal 2-category.
    \end{lem}
    
    \begin{proof}
    A monoidal 2-category $\cat$ consists of the following data.
    \begin{enumerate}[i]
    \item The 2-category $\cat$ is equipped with the pseudo-functor $\boe: \cat \times \cat \to \cat$ and the tensor unit $\E$. 
    \item    The associator is a strong transformation $\alpha: \boxtimes_{\E} \circ (\boxtimes_{\E} \times \id) \Rightarrow \boxtimes_{\E} \circ (\id \times \boxtimes_{\E})$ in the 2-cateogry $\ccat^{\ps}((\cat)^3, \cat)$.
     For each $(\A, \B, \C) \in (\cat)^3$,  $\alpha$ contains an invertible 1-morphism $\alpha_{\A, \B, \C}: (\A \boe \B) \boe \C \rightarrow \A \boe (\B \boe \C)$ induced by 
    \[
    \xymatrix{
    \A \times \B \times \C \ar[rr]^{\boxtimes_{\E} \circ (\boxtimes_{\E} \times \id)} \ar[drr]_{\boxtimes_{\E} \circ (\id \times \boxtimes_{\E}) } && (\A \boxtimes_{\E} \B) \boxtimes_{\E} \C \ar[d]^{\exists ! \alpha_{\A, \B, \C}} \\
    && \A \boxtimes_{\E} (\B \boxtimes_{\E} \C) \ultwocell<>{<1>d^{\alpha}_{\A\B\C}\quad\;\,}
    }
    \]  
    For each 1-morphism $(f_1, f_2, f_3): (\A, \B, \C) \rightarrow (\A', \B', \C')$ in $(\cat)^3$, $\alpha$ contains an invertible 2-morphism $\alpha_{f_1, f_2, f_3}: (f_1 \boe (f_2 \boe f_3)) \circ \alpha_{\A, \B, \C} \Rightarrow \alpha_{\A', \B', \C'} \circ ((f_1 \boe f_2) \boe f_3)$ induced by
    \[ \vcenter{\hbox{
    \xymatrixcolsep{1pc}
    \xymatrix{
    \A \times \B \times \C \ar[r]^{\boxtimes_{\E}, 1} \ar[dd]_{f_1, f_2,f_3} \ar[rd]_{1, \boxtimes_{\E}} & (\A \boxtimes_{\E} \B) \times \C \ar[r]^{\boxtimes_{\E}} \drtwocell<>{^d^{\alpha}_{\A\B\C}\quad\;\;} & (\A \boxtimes_{\E} \B) \boxtimes_{\E} \C \ar[d]|{\alpha_{\A,\B,\C}} \ar[rd]^{(f_1 \boxtimes_{\E} f_2) \boxtimes_{\E} f_3} & \ddtwocell<>{^<9> \exists ! \alpha_{f_1f_2f_3}} \\
   \drtwocell<>{^1_{f_1}, t_{f_2f_3} \qquad} & \A \times (\B \boxtimes_{\E} \C) \ar[r]^{\boxtimes_{\E}} \ar[d]|(0.5){f_1, f_2 \boxtimes_{\E} f_3} \drtwocell<>{^t_{f_1, f_2f_3}\qquad} & \A \boxtimes_{\E} (\B \boxtimes_{\E} \C) \ar[d]|(0.5){f_1 \boxtimes_{\E} (f_2 \boxtimes_{\E} f_3)} & (\A' \boxtimes_{\E} \B') \boxtimes_{\E} \C' \ar[ld]^{\alpha_{\A',\B',\C'}} \\
    \A' \times \B' \times \C' \ar[r]_{1, \boxtimes_{\E}} & \A' \times (\B' \boxtimes_{\E} \C') \ar[r]_{\boxtimes_{\E}} & \A' \boxtimes_{\E} (\B' \boxtimes_{\E} \C') &
    }}}\]
     \[ \parallel\]
    \[\vcenter{\hbox{
    \xymatrixcolsep{1pc}
    \xymatrix{
    \A \times \B \times \C \ar[r]^{\boxtimes_{\E}, 1} \ar[dd]_{f_1, f_2,f_3} \ddtwocell<>{^<-4>t_{f_1f_2}, 1_{f_3}} & (\A \boxtimes_{\E} \B) \times \C \ar[r]^{\boxtimes_{\E}} \ar[d]^{(f_1 \boxtimes_{\E} f_2), f_3}  & (\A \boxtimes_{\E} \B) \boxtimes_{\E} \C  \ar[rd]^{(f_1 \boxtimes_{\E} f_2) \boxtimes_{\E} f_3} \dtwocell<>{^<2>t_{f_1f_2, f_3}} &  \\
    & (\A' \boxtimes_{\E} \B') \times \C' \ar[rr]^{\boxtimes_{\E}} & \dtwocell<>{^<6>\qquad d^{\alpha}_{\A'\B'\C'}} & (\A' \boxtimes_{\E} \B') \boxtimes_{\E} \C' \ar[ld]^{\alpha_{\A',\B',\C'}} \\
    \A' \times \B' \times \C' \ar[r]_{1, \boxtimes_{\E}} \ar[ru]^{\boxtimes_{\E}, 1} & \A' \times (\B' \boxtimes_{\E} \C') \ar[r]_{\boxtimes_{\E}} & \A' \boxtimes_{\E} (\B' \boxtimes_{\E} \C') &
    }}}
     \]

    \item The left unitor and right unitor are strong transformations $l: \E \boxtimes_{\E} - \Rightarrow -$ and $r: - \boxtimes_{\E} \E \Rightarrow -$ in $\ccat^{\ps}(\cat, \cat)$. For each $\A \in \cat$, $l$ and $r$ contain invertible 1-morphisms $l_{\A}: \E \boe \A \rightarrow \A$ and $r_{\A}: \A \boe \E \rightarrow \A$ respectively.
  \[
    \vcenter{\hbox{\xymatrix{
    \E \times \A \ar[r]^{\boxtimes_{\E}}  \ar[rd]_{\odot} & \E \boxtimes_{\E} \A \ar[d]^{\exists !  l_{\A}}  \\
    & \A    \ultwocell<>{<2>d_{\A}^l\;}                
     }}}
    \qquad
    \vcenter{\hbox{\xymatrix{
     \A \times \E \ar[r]^{\boxtimes_{\E}}  \ar[rd]_{\odot} & \A \boxtimes_{\E} \E \ar[d]^{\exists ! r_{\A}}  \\
      &         \A \ultwocell<>{<2>d_{\A}^r\;}
    }}}
    \]
 For each 1-morphism $f: \A \rightarrow \B$ in $\cat$, $l$ and $r$ contain invertible 2-morphisms $\beta^l_f: f \circ l_{\A} \Rightarrow l_{\B} \circ (1_{\E} \boe f)$ and $\beta^r_f: f \circ r_{\A} \Rightarrow  r_{\B} \circ (f \boe 1_{\E})$ respectively.
    \[ 
    \vcenter{\hbox{\xymatrix{
    \E \times \A \ar[r]^{\boxtimes_{\E}} \ar[dd]_{1_{\E}, f} \ar[rd]_{\odot} \drtwocell<>{^<-2>d^l_{\A}\;} & \E \boxtimes_{\E} \A \ar[d]^{l_{\A}} \ar[rd]^{1_{\E} \boxtimes_{\E} f} \ddtwocell<>{^<-5>\exists ! \beta^l_f} & \\
    & \A \ar[d]^{f} & \E \boxtimes_{\E} \B \ar[ld]^{l_{\B}} \\
    \E \times \B \ar[r]_{\odot} \utwocell<>{<6>s^l_f} & \B
    }}}
    =
    \vcenter{\hbox{\xymatrix{
    \E \times \A \ar[r]^{\boxtimes_{\E}} \ar[dd]_{1_{\E}, f}  & \E \boxtimes_{\E} \A \ar[rd]^{1_{\E} \boxtimes_{\E} f} \dtwocell<>{^<4>t_{1_{\E}f}} & \\
    &   & \E \boxtimes_{\E} \B \ar[ld]^{l_{\B}} \\
    \E \times \B \ar[r]_{\odot} \ar@/^1.6pc/[rru]^{\boxtimes_{\E}} & \B \utwocell<>{d^l_{\B}}
    }}}
     \]
       \[ 
    \vcenter{\hbox{\xymatrix{
    \A \times \E \ar[r]^{\boxtimes_{\E}} \ar[dd]_{f,1_{\E}} \ar[rd]_{\odot} \drtwocell<>{^<-2>d^r_{\A}\;} & \A \boxtimes_{\E} \E \ar[d]^{r_{\A}} \ar[rd]^{f \boxtimes_{\E} 1_{\E}} \ddtwocell<>{^<-5>\exists ! \beta^r_f} & \\
    & \A \ar[d]^{f} & \B \boxtimes_{\E} \E \ar[ld]^{r_{\B}} \\
    \B \times \E \ar[r]_{\odot} \utwocell<>{<6>s^r_f} & \B
    }}}
    =
    \vcenter{\hbox{\xymatrix{
    \A \times \E \ar[r]^{\boxtimes_{\E}} \ar[dd]_{f,1_{\E}}  & \A \boxtimes_{\E} \E \ar[rd]^{f \boxtimes_{\E} 1_{\E}} \dtwocell<>{^<4>t_{f1_{\E}}} & \\
    &   & \B \boxtimes_{\E} \E \ar[ld]^{r_{\B}} \\
    \B \times \E \ar[r]_{\odot} \ar@/^1.6pc/[rru]^{\boxtimes_{\E}} & \B \utwocell<>{d^r_{\B}}
    }}}
     \]
     where $(f, s^l_f): \A \rightarrow \B$ is a left $\E$-module functor and $(f, s^r_f): \A \rightarrow \B$ is a right $\E$-module functor.

     \item The pentagonator is a modification $\pi$ in $\ccat^{\ps}((\cat)^4, \cat)$. For each $\A, \B, \C, \D \in \cat$, $\pi$ consists of an invertible 2-morphism $\pi_{\A, \B, \C, \D}: (1_{\A} \boxtimes_{\E} \alpha_{\B,\C,\D}) \circ \alpha_{\A, \B \boxtimes_{\E} \C, \D} \circ (\alpha_{\A, \B, \C} \boxtimes_{\E} 1_{\D}) \Rightarrow \alpha_{\A, \B, \C \boxtimes_{\E} \D} \circ \alpha_{\A \boxtimes_{\E} \B, \C, \D}$ induced by (where, for example, $\A \boxtimes_{\E} \B$ is abbreviated to $\A\B$): 
     \[ \xymatrix{
     \A \times \B \times \C \times \D \ar[r]^{\boxtimes_{\E}, 1, 1} \ar[ddd]_{1, 1, \boxtimes_{\E}} \ar[rd]_{1, \boxtimes_{\E}, 1} & \A\B \times \C \times \D \ar[r]^{\boxtimes_{\E}, 1} \dtwocell<>{^<-2>d^{\alpha}_{\A\B\C}, 1_{1_{\D}}} & (\A\B)\C \times \D \ar[r]^{\boxtimes_{\E}} \ar[d]_{\alpha_{\A,\B,\C}, 1_{\D}}  \drtwocell<>{^<-1>t_{\alpha_{\A,\B,\C}, 1_{\D}}\qquad\;\,} & ((\A\B)\C)\D \ar[d]^(0.6){\alpha_{\A,\B,\C} \boxtimes_{\E} 1_{\D}} \ar@/^1pc/[rd]^{\alpha_{\A\B, \C, \D}} &  \\
    \ddtwocell<>{^<-6>1_{1_{\A}}, d^{\alpha}_{\B\C\D}} & \A \times \B\C \times \D \ar[r]_{\boxtimes_{\E}, 1} \ar[d]_{1, \boxtimes_{\E}} \drrtwocell<>{^d^{\alpha}_{\A, \B\C, \D} \qquad} & \A(\B\C) \times \D \ar[r]_{\boxtimes_{\E}} &(\A(\B\C))\D \ar[d]_{\alpha_{\A, \B\C, \D}} \dtwocell<>{^<-6>\exists ! \pi_{\A,\B,\C,\D}}& (\A\B)(\C\D) \ar[ldd]^{\alpha_{\A, \B, \C\D}} \\
     & \A \times (\B\C) \D \ar[rr]^{\boxtimes_{\E}} \ar[d]_{1_{\A}, \alpha_{\B,\C,\D}} \drrtwocell<>{^t_{1_{\A}, \alpha_{\B,\C,\D}} \qquad\;\,} & & \A((\B\C) \D) \ar[d]_{1_{\A} \boxtimes_{\E} \alpha_{\B,\C,\D}} & \\
    \A \times \B \times \C\D \ar[r]^{1, \boxtimes_{\E}} & \A \times \B(\C\D) \ar[rr]^{\boxtimes_{\E}} & & \A(\B(\C\D)) &
     } \]
     \[ \parallel \]
     \[ \xymatrix{
     \A \times \B \times \C \times \D \ar[r]^{\boe, 1, 1} \ar[d]_{1, 1, \boe} & \A\B \times \C \times \D \ar[r]^{\boe, 1} \ar[d]_{1, \boe} \drrtwocell<>{^d^{\alpha}_{\A\B, \C, \D}\qquad} & (\A\B)\C \times \D \ar[r]^{\boe} & ((\A\B)\C)\D \ar[d]^{\alpha_{\A\B, \C, \D}} \\
     \A \times \B \times \C\D \ar[r]_{\boe, 1} \ar[rd]_{1, \boe} & \A\B \times \C\D \ar[rr]^{\boe} \drrtwocell<>{^d^{\alpha}_{\A, \B, \C\D} \qquad}  & & (\A\B)(\C\D) \ar[d]^{\alpha_{\A, \B, \C\D}} \\
     &\A \times \B(\C\D) \ar[rr]_{\boe} &  & \A(\B(\C\D))
     } \]
 
     \item The middle 2-unitor $\mu$ is a modification in $\ccat^{\ps}((\cat)^2, \cat)$.
     For each $(\B, \A) \in (\cat)^2$, $\mu$ consists of an invertible 2-morphism $\mu_{\B, \A}: (1_{\B} \boe l_{\A}) \circ \alpha_{\B,\E,\A} \Rightarrow 1_{\B \boe \A} \circ (r_{\B} \boe 1_{\A})$ induced by
      \[
     \xymatrixcolsep{0.8pc}
     \xymatrix{
     \B \times \E \times \A \ar[r]^{\boe, 1} \ar[d]_{1, \boe} \ar@/_4pc/[dd]_{1_{\B}, \odot} & \B \boe \E \times \A \ar[r]^{\boe} \dtwocell<>{^<3>d^{\alpha}_{\B\E\A}\quad\;}& (\B \boe \E) \boe \A \ar[dl]_{\alpha_{\B,\E,\A}} \ar[d]^{r_{\B} \boe 1_{\A}} \ddtwocell<>{^<6>\exists ! \mu_{\B,\A}} & \\
    \dtwocell<>{^<5>1_{1_{\B}}, d^l_{\A}}  \B \times \E \boe \A \ar[r]^{\boe} \ar[d]^{1, l_{\A}} \dtwocell<>{^<-9>t_{1_{\B}, l_{\A}}} & \B \boe (\E \boe \A) \ar[dr]_{1_{\B} \boe l_{\A}}&  \B \boe \A \ar[d]^{1_{\B \boe \A}} & \\
      \B \times \A \ar[rr]_{\boe} & & \B \boe \A &
     }\]
  \[ \parallel \]
   \[ \xymatrixcolsep{0.7pc}
     \xymatrix{
     \B \times \E \times \A \ar[r]^{\boe, 1} \ar[dd]_{1_{\B}, \odot} \ar[rd]_{\odot, 1_{\A}} \dtwocell<>{^<-9>d^r_{\B},1_{1_{\A}}} & \B \boe \E \times \A \ar[r]^{\boe} \ar[d]^{r_{\B}, 1_{\A}} \dtwocell<>{^<-8>t_{r_{\B},1_{\A}}} & (\B \boe \E) \boe \A \ar[d]^{r_{\B} \boe 1_{\A}} \\
     & \B \times \A \ar[r]_{\boe} \dtwocell<>{^<6>b^{\B\A}} \ar[dr]_{\boe} & \B \boe \A \ar[d]^{1_{\B \boe \A}} \dtwocell<>{^<3> \id} \\
     \B \times \A \ar[rr]_{\boe}&  &  \B \boe \A 
     }  \]
    where $b^{\B\A}$ is the balanced $\E$-module structure on the functor $\boe: \B \times \A \rightarrow \B \boe \A$. 

      \item The left 2-unitor $\lambda$ is a modification in $\ccat^{\ps}((\cat)^2, \cat)$. For each $(\B, \A) \in (\cat)^2$, $\lambda$ consists of an invertible 2-morphism $\lambda_{\B\A}: l_{\B \boe \A} \circ \alpha_{\E,\B,\A} \Rightarrow l_{\B} \boe 1_{\A}$ induced by
     \[ 
     \indent\hspace{-1em}
      \vcenter{\hbox{
     \xymatrixcolsep{0.8pc}
     \xymatrix{
     \E \times \B \times \A \ar[r]^{\boe, 1} \ar[rd]_{1, \boe} \ar@/_2pc/[rdd]_{\odot, 1} & \E \boe \B \times \A \ar[r]^{\boe} \drtwocell<>{^d^{\alpha}_{\E\B\A}\quad\;} & (\E \boe \B) \boe \A \ar[d]^{\alpha_{\E,\B,\A}} \ar@/^4.5pc/[dd]^(0.8){l_{\B} \boe 1_{\A}} \ddtwocell<>{^<-8>\exists!\lambda_{\B,\A}}  \\
     & \E \times \B \boe \A \ar[r]^{\boe} \ar[dr]_{\odot} \dtwocell<>{^<-11>d^l_{\B\A}} \dtwocell<>{^<4>t^l_{\B\A}} & \E \boe (\B \boe \A) \ar[d]^{l_{\B \boe \A}}  \\
     & \B \times \A \ar[r]_{\boe} & \B \boe \A 
     }}} 
      =
      \indent\hspace{-1em}
     \vcenter{\hbox{
     \xymatrixcolsep{0.7pc}
     \xymatrix{
     \E \times \B \times \A \ar[r]^{\boe, 1} \ar@/_1pc/[ddr]_{\odot,1} \ddrtwocell<>{^<-3>d^l_{\B},1_{1_{\A}} \qquad} & \E \boe \B \times \A \ar[r]^{\boe} \ar[dd]^{l_{\B},1_{\A}} \ddtwocell<>{^<-9>t_{l_{\B},1_{\A}}} & (\E \boe \B) \boe \A \ar[dd]^{l_{\B} \boe 1_{\A}} \\
     & &  \\
      &\B \times \A \ar[r]_{\boe}  &  \B \boe \A 
     }}}     
     \]

     \item The right 2-unitor $\rho$ is a modification in $\ccat^{\ps}((\cat)^2, \cat)$.
     For each $(\B, \A)$ in $(\cat)^2$, $\rho$ consists of an invertible 2-morphism $\rho_{\B, \A}: (1_{\B} \boe r_{\A}) \circ \alpha_{\B,\A,\E} \Rightarrow r_{\B \boe \A}$ induced by
     \[ 
     \indent\hspace{-1em} 
     \vcenter{\hbox{
     \xymatrixcolsep{0.8pc}
     \xymatrix{
     \B \times \A \times \E \ar[r]^{\boe, 1} \ar[rd]_{1, \boe} \ar@/_2pc/[rdd]_{1, \odot} & \B \boe \A \times \E \ar[r]^{\boe} \drtwocell<>{^d^{\alpha}_{\B\A\E}\quad\;}& (\B \boe \A) \boe \E \ar[d]^{\alpha_{\B,\A,\E}} \ar@/^4.5pc/[dd]^(0.8){r_{\B \boe \A}} \ddtwocell<>{^<-8>\exists!\rho_{\B,\A}}  \\
    \dtwocell<>{^<-8>1_{1_{\B}}, d^r_{\A}} & \B \times \A \boe \E \ar[r]^{\boe} \ar[d]^{1, r_{\A}} \dtwocell<>{^<-7>t_{1_{\B}, r_{\A}}} & \B \boe (\A \boe \E) \ar[d]^{1_{\B} \boe r_{\A}}  \\
     & \B \times \A \ar[r]_{\boe} & \B \boe \A 
     }}} 
     \; =
     \indent\hspace{-1em}
      \vcenter{\hbox{
     \xymatrixcolsep{0.7pc}
     \xymatrix{
     \B \times \A \times \E \ar[r]^{\boe, 1} \ar[dd]^{1, \odot}  & \B \boe \A \times \E \ar[r]^{\boe} \ar[ddr]_{\odot} \ddtwocell<>{^<-11>d^r_{\B\A}} \ddtwocell<>{^<4>t^r_{\B\A}} & (\B \boe \A) \boe \E \ar[dd]^{r_{\B \boe \A}} \\
     & &  \\
     \B \times \A \ar[rr]_{\boe}&  &  \B \boe \A 
     }}}     
     \]
     \end{enumerate}
  It is routine to check that $\alpha$, $l$, $r$ satisfy the lax unity and the lax naturality of \cite[Def.\,4.2.1]{yau}, and $\pi$, $\mu$, $\lambda$, $\rho$ satisfy the modification axiom of \cite[Def.\,4.4.1]{yau}.  
    It is routine to check that the above data satisfy the non-abelian 4-cocycle condition, the left normalization and the right normalization of \cite[(11.2.14), (11.2.16), (11.2.17)]{yau}.                  
    \end{proof}

\subsection{The symmetric monoidal 2-category $\cat$}
\label{m-cate2}
Let $(\cat, \boe, \E, \alpha, l, r, \pi, \mu, \lambda, \rho)$ be the monoidal 2-category. 
 For objects $\A, \B \in \cat$, the braiding $\tau$ consists of an invertible 1-morphism $\tau_{\A,\B}: \A \boe \B \rightarrow \B \boe \A$ defined as 
\[ \xymatrix{
\A \times \B \ar[r]^{\boe} \ar[d]_{s_{\A, \B}}  \drtwocell<>{^d^{\tau}_{\A, \B}\quad} & \A \boe \B  \ar[d]^{\exists ! \tau_{\A\B}}  \\
\B \times \A \ar[r]_{\boe} & \B \boe \A
} \]
where $s$ switches the two objects.
For objects $\A, \B, \C \in \cat$, the left hexagonator $R_{-|--}$ and the right hexagonator $R_{--|-}$ consist of invertible 2-morphisms $R_{\A|\B, \C}$ and $R_{\A, \B|\C}$ respectively.

\[
\xymatrix{
\A \times \B \times \C \ar[r]^{\boe, 1} \ar[dr]^{1, \boe} \ar[dd]_{s_{\A, \B \times \C}} & \A \boe \B \times \C \ar[r]^{\boe} \drtwocell<>{^d^{\alpha}_{\A,\B,\C} \quad\;\;} & (\A \boe \B) \boe \C \ar[d]_{\alpha_{\A, \B, \C}} \ar[rd]^{\tau_{\A\B}, 1} & \\
& \A \times \B \boe \C \ar[r]^{\boe} \ar[d]_{\tau_{\A, \B \boe \C}}  \drtwocell<>{^d^{\tau}_{\A,\B\C} \quad\;\;} & \A \boe (\B \boe \C) \ar[d]_{\tau_{\A, \B \boe \C}} & (\B \boe \A) \boe \C \ar[d]^{\alpha_{\B, \A, \C}} \dtwocell<>{^<9>\exists ! R_{\A|\B, \C}} \\
\B \times \C \times \A \ar[r]^{\boe, 1} \ar[dr]_{1, \boe} & \B \boe \C \times \A \ar[r]^{\boe} \drtwocell<>{^d^{\alpha}_{\B,\C,\A} \quad\;\;} & (\B \boe \C) \boe \A \ar[d]_{\alpha_{\B,\C,\A}} & \B \boe (\A \boe \C) \ar[dl]^{1, \tau_{\A, \C}} \\
& \B \times \C \boe \A \ar[r]_{\boe} & \B \boe (\C \boe \A) &
}\]
\[\parallel \]
\[\xymatrix{
\A \times \B \times \C \ar[r]^{\boe, 1} \ar[d]_{s_{\A,\B}, 1} \ar@/_4pc/[dd]_{s_{\A, \B \times \C}} \drtwocell<>{^d^{\tau}_{\A\B}, 1 \quad} & \A \boe \B \times \C \ar[r]^{\boe} \ar[d]^{\tau_{\A\B}, 1}  \drtwocell<>{^t_{\tau, 1}\;\;} & (\A \boe \B) \boe \C \ar[d]^{\tau_{\A,\B}, 1}  \\
\B \times \A \times \C \ar[r]^{\boe, 1} \ar[d]_{1, s_{\A,\C}} \ar[dr]^{1, \boe} \ddtwocell<>{^<-7>1, d^{\tau}_{\A,\C}} & \B \boe \A \times \C \ar[r]^{\boe} \drtwocell<>{^d^{\alpha}_{\B,\A,\C} \quad\;\;} & (\B \boe \A) \boe \C \ar[d]^{\alpha_{\B,\A,\C}} \\
\B \times \C \times \A  \ar[rd]_{1, \boe} & \B \times \A \boe \C \ar[r]^{\boe} \ar[d]^{1, \tau_{\A,\C}} \drtwocell<>{^t_{1, \tau}\;\;} & \B \boe (\A \boe \C) \ar[d]^{1, \tau_{\A, \C}} \\
& \B \times \C \boe \A \ar[r]_{\boe} & \B \boe (\C \boe \A) 
}\]
\[ \xymatrix{
\A \times \B \times \C \ar[r]^{1, \boe} \ar[dr]_{\boe, 1} \ar[dd]_{s_{\A \times \B, \C}} & \A \times \B \boe \C  \ar[r]^{\boe} \drtwocell<>{^d^{\alpha^{-1}}_{\A,\B,\C} \quad\;\;} & \A \boe (\B \boe \C) \ar[d]^{\alpha^{-1}_{\A,\B,\C}} \ar[rd]^{1, \tau_{\B,\C}} & \\
& \A \boe \B \times \C \ar[r]^{\boe} \ar[d]_{s_{\A \boe \B, \C}} \drtwocell<>{^d^{\tau}_{\A\B, \C} \quad\;\;} & (\A \boe \B) \boe \C \ar[d]_{\tau_{\A \boe \B, \C}}  \dtwocell<>{^<-9> \exists ! R_{\A, \B|\C}} & \A \boe (\C \boe \B) \ar[d]^{\alpha^{-1}_{\A,\C,\B}}  \\
\C \times \A \times \B \ar[r]^{1, \boe} \ar[rd]_{\boe, 1} & \C \times \A \boe \B \ar[r]^{\boe} \drtwocell<>{^d^{\alpha^{-1}}_{\C,\A, \B} \quad\;\;} & \C \boe (\A \boe \B)  \ar[d]^{\alpha^{-1}_{\C,\A,\B}} & (\A \boe \C) \boe \B  \ar[ld]^{\tau_{\A,\C}, 1}  \\
& \C \boe \A \times \B  \ar[r]_{\boe} & (\C \boe \A) \boe \B & 
} \]
\[ \parallel \]
\[ \xymatrix{
\A \times \B \times \C \ar[r]^{1, \boe} \ar[d]_{1, s_{\B,\C}}  \ar@/_4pc/[dd]_{s_{\A \times \B, \C}} \drtwocell<>{^1, d^{\tau}_{\B,\C} \quad\;\;} & \A \times \B \boe \C \ar[r]^{\boe} \ar[d]^{1, \tau_{\B,\C}} \drtwocell<>{^t_{1, \tau}\;\;} & \A \boe (\B \boe \C) \ar[d]^{1, \tau_{\B,\C}}  \\
\A \times \C \times \B \ar[r]^{1, \boe} \ar[d]_{s_{\A,\C}, 1} \ar[dr]^{\boe, 1} \ddtwocell<>{^<-7> d^{\tau}_{\A,\C}, 1} & \A \times \C \boe \B  \ar[r]^{\boe} \drtwocell<>{^d^{\alpha^{-1}}_{\A,\C,\B} \quad\;\;} &  \A \boe (\C \boe \B) \ar[d]^{\alpha^{-1}_{\A,\C,\B}} \\
\C \times \A \times \B \ar[dr]_{\boe, 1} & \A \boe \C \times \B  \ar[r]^{\boe} \ar[d]^{\tau_{\A,\C}, 1} \drtwocell<>{^t_{\tau, 1}\;\;} &  (\A \boe \C) \boe \B \ar[d]^{\tau_{\A,\C}, 1} \\
& \C \boe \A \times \B  \ar[r]_{\boe} & (\C \boe \A) \boe \B
} \]
For objects $\A, \B \in \cat$, the syllepsis $\nu$ consists of an invertible 2-morphism $\nu_{\A,\B}$ defined as
\[ \vcenter{\hbox{\xymatrix{
\A \times \B \ar[r]^{\boe} \ar[d]_{s_{\A, \B}} \drtwocell<>{^d^{\tau}_{\A,\B} \quad}& \A \boe \B \ar[d]^{\tau_{\A, \B}} \ar@/^3.6pc/[dd]^{1_{\A \boe \B}} \ddtwocell<>{^<-6>\exists !\nu_{\A,\B}} \\
\B \times \A \ar[d]_{s_{\B, \A}} \ar[r]_{\boe} \drtwocell<>{^d^{\tau}_{\B, \A} \quad} & \B \boe \A \ar[d]^{\tau_{\B,\A}}  \\
\A \times \B \ar[r]_{\boe} & \A \boe \B
}}} 
=
\vcenter{\hbox{\xymatrix{
\A \times \B \ar[r]^{\boe} \ar[d]_{s_{\A, \B}} \ar@/^2pc/[dd]^{1_{\A \times \B}} & \A \boe \B  \ar[dd]^{1_{\A \boe \B}}  \\
\B \times \A \ar[d]_{s_{\B, \A}}   &    \\
\A \times \B \ar[r]_{\boe} & \A \boe \B
}}} \]
where we choose the identity 2-morphism $\id: \boe \circ 1_{\A \times \B} \Rightarrow 1_{\A \boe \B} \circ \boe$ for convenience. 
It is routine to check that $(\cat, \tau, R_{-|--}, R_{--|-}, \nu)$ is a symmetric monoidal 2-category \cite[Def.\,12.1.6, 12.1.15, 12.1.19]{yau}.


\begin{thebibliography}{KWak2}
\bibitem[AF1]{AF}
David Ayala, John Francis. Factorization homology of topological manifolds. Journal of Topology, 2015, 8(4):1045-1084.
\bibitem[AF2]{Francis}
David Ayala, John Francis. A factorization homology primer. In: Handbook of Homotopy theory. Chapman and Hall/CRC, 2020, 39-101.
\bibitem[AFT1]{AFT1}
David Ayala, John Francis, Hiro Lee Tanaka. Local structures on stratified spaces. Advances in Mathematics, 2017, 307:903-1028.
\bibitem[AFT2]{AFT2}
David Ayala, John Francis, Hiro Lee Tanaka. Factorization homology of stratified spaces. Selecta Mathematica, 2016, 23(1):293-362.
\bibitem[AFR]{AFR}
David Ayala, John Francis, Nick Rozenblyum. Factorization homology i: Higher categories. Advances in Mathematics, 2018, 333:1024-1177. 
\bibitem[AKZ]{LiangFH}
Yinghua Ai, Liang Kong, Hao Zheng. Topological orders and factorization homology. Advances in Theoretical and Mathematical Physics, 2017, 21(8):1854-1894.
\bibitem[BBJ1]{BBJ}
David Ben-Zvi, Adrien Brochier, David Jordan. Integrating Quantum groups over surfaces. Journal of Topology, 2018, 11(4):874-917.
\bibitem[BBJ2]{BBJ2}
David Ben-Zvi, Adrien Brochier, David Jordan. Quantum character varieties and braided module categories. Selecta Mathematica, 2018, 24(5):4711-4748. 
\bibitem[BD]{BD}
Alexander Beilinson, Vladimir Drinfeld. Chiral algebras. American Mathematical Society, Providence, R.I, 2004.
\bibitem[CG]{CG}
Kevin Costello, Owen Gwilliam. Factorization algebras in perturbative quantum field theory. Cambridge University Press, 2016.
\bibitem[DGNO]{EGNO1}
Vladimir Drinfeld, Shlomo Gelaki, Dmitri Nikshych, Victor Ostrik. On braided fusion categoires I. Selecta Mathematica, 2010, 16(1):1-119.
\bibitem[DMNO]{DMNO}
Alexei Davydov, Michael M\"{u}ger, Dmitri Nikshych, Victor Ostrik. The Witt group of non-degenerate braided fusion categories. Journal f\"{u}r die reine und angewandte Mathematik (Grelles Journal), 2013, 677:135-177.
\bibitem[DNO]{DNO}
Alexei Davydov, Dmitri Nikshych, Victor Ostrik. On the structure of the Witt group of braided fusion categories.
Selecta Mathematica, 2012, 19(1):237-269. 
\bibitem[EGNO]{Etingof}
Pavel Etingof, Shlomo Gelaki, Dmitri Nikshych, Victor Ostrik. Tensor categories. American Mathematical Society, 2015.
\bibitem[ENO]{ENO}
Pavel Etingof, Dmitri Nikshych, Victor Ostrik. Fusion categories and homotopy theory. Quantum Topology, 2010, 1(3):209-273.
\bibitem[F]{F}
John Francis. The tangent complex and Hochschild cohomology of $E_n$-rings. Compositio Mathematica, 2013, 149:430-480.
\bibitem[FG]{FG}
John Francis, Dennis Gaitsgory. Chiral koszul duality. Selecta Mathematica, 2012, 18:27-87.
\bibitem[JY]{yau}
Niles Johnson, Donald Yau. 2-Dimensional Categories. Oxford University Press, 2021.
\bibitem[KZ]{Liang}
Liang Kong, Hao Zheng. The center functor is fully faithful. Advanced in Mathematics, 2018, 339:749-779.
\bibitem[L]{Lu}
Jacob Lurie. On the classification of topological field theories. Current Developments in Mathematics, 2008, (1):129-280.

\bibitem[LKW]{TL}
Tian Lan, Liang Kong, Xiao-Gang Wen. Modular extensions of unitary braided fusion categories and 2+1D topological/SPT orders with symmetries. Communications in Mathematical Physics, 2016, 351(2):709-739.

\bibitem[GHR]{GHR}
C\'{e}sar Galindo, Seung-Moon Hong, Eric C.Rowell. Generalized and quasi-localizations of braid group representations. International Mathematics Research Notices, 2013, (3):693-731.

\bibitem[Su]{SL}
Long Sun. The symmetric enriched center functor is fully faithful. Communications in Mathematical Physics, 2022, 395:1345-1382.

\bibitem[S]{S}
Gregor Schaumann. Traces on module categories over fusion categories. ScienceDirect, 2013, 379:382-425.
\bibitem[W]{Wen}
Xiao-Gang Wen. Choreographed entanglement dances: Topological states of quantum matter. Science, 2019, 363(6429). 
\end{thebibliography}
\end{document}